 \documentclass[12pt]{article}

\usepackage{eucal}
\usepackage{color}

\usepackage{amsmath,bbm}
\usepackage{amssymb}
\usepackage{epic}

\usepackage{graphicx}
\usepackage{pict2e}

\title{Coupled GUE-minor Processes}

 \author{ 
  Mark Adler\thanks{  
  Department of Mathematics, Brandeis University,
 Waltham, Mass 02454, USA. E-mail: adler@brandeis.edu.
 The support of a Simons Foundation Grant 
 \# 278931 is
 gratefully acknowledged.}~
  ~ and~~~Pierre
 van Moerbeke\thanks{ Department of Mathematics,
 Universit\'e de Louvain, 1348 Louvain-la-Neuve, Belgium and Brandeis University, Waltham, MA 02454, USA. E-mail: pierre.vanmoerbeke@uclouvain.be and
 vanmoerbeke@brandeis.edu. The support of a Simons Foundation Grant 
 \# 280945 is
 gratefully acknowledged.
 \newline 2000{\em Mathematics Subject Classification}. Primary:60G60, 60G65, 35Q53; secondary: 60G10, 35Q58. {\em Key words and Phrases}: Random matrices, GUE-ensemble, spectrum of minors, coupled random matrices, domino tilings of double Aztec diamonds, tacnode processes, GUE-minor kernel.}}

\date{}

\newcommand{\MAT}[1]{\left(\begin{array}{*#1c}}
\newcommand{\mat}{\end{array}\right)}
\newcommand{\qed}{\leavevmode\unskip\nobreak\penalty200\hskip2pt\null
\nobreak\hfill\rule{1.1ex}{1.1ex}
\medbreak }

\newcommand{\I}{{\rm i}}

\newcommand{\BH}{{\mathbb H}}

\newcommand{\BP}{{\mathbb P}}

\newcommand{\BZ}{{\mathbb Z}}

\newcommand{\al}{\alpha}

\newcommand{\Id}{\mathbbm{1}}

\newenvironment{proof}{\medskip\noindent{\it Proof:\/} }{\qed}
\newenvironment
        {remark}{\medskip\noindent\underline{\it Remark:\/} }{\medbreak}

\newcommand{\om}{\omega}
\newcommand{\Om}{\Omega}

\newcommand{\la}{\langle}
\newcommand{\ra}{\rangle}

\newcommand{\dt}{\delta}
\newcommand{\Dt}{\Delta}
 \newcommand{\vr}{\varepsilon}

\newcommand{\BR}{{\mathbb R}}
\newcommand{\lb}{\lambda}

\newcommand{\BK}{{\mathbb K}}

\def\be#1\ee{\begin{equation}#1\end{equation}}
\def\bea#1\eea{\begin{eqnarray}#1\end{eqnarray}}
\def\bean#1\eean{\begin{eqnarray*}#1\end{eqnarray*}}

 \newtheorem{definition}{Definition}[section]
 \newtheorem{theorem}[definition]{Theorem}
 \newtheorem{lemma}[definition]{Lemma}
 \newtheorem{corollary}[definition]{Corollary}
 \newtheorem{proposition}[definition]{Proposition}

\catcode `!=11

\newdimen\squaresize
\newdimen\thickness
\newdimen\Thickness
\newdimen\ll! \newdimen \uu! \newdimen\dd! \newdimen \rr! \newdimen
\temp!

\def\sq!#1#2#3#4#5{%
\ll!=#1 \uu!=#2 \dd!=#3 \rr!=#4
\setbox0=\hbox{%
 \temp!=\squaresize\advance\temp! by .5\uu!
 \rlap{\kern -.5\ll!
 \vbox{\hrule height \temp! width#1 depth .5\dd!}}%
 \temp!=\squaresize\advance\temp! by -.5\uu!
 \rlap{\raise\temp!
 \vbox{\hrule height #2 width \squaresize}}%
 \rlap{\raise -.5\dd!
 \vbox{\hrule height #3 width \squaresize}}%
 \temp!=\squaresize\advance\temp! by .5\uu!
 \rlap{\kern \squaresize \kern-.5\rr!
 \vbox{\hrule height \temp! width#4 depth .5\dd!}}%
 \rlap{\kern .5\squaresize\raise .5\squaresize
 \vbox to 0pt{\vss\hbox to 0pt{\hss $#5$\hss}\vss}}%
}
 \ht0=0pt \dp0=0pt \box0
}

\def\vsq!#1#2#3#4#5\endvsq!{\vbox to \squaresize{\hrule
width\squaresize height 0pt%
\vss\sq!{#1}{#2}{#3}{#4}{#5}}}

\newdimen \LL! \newdimen \UU! \newdimen \DD! \newdimen \RR!

\def\vvsq!{\futurelet\next\vvvsq!}
\def\vvvsq!{\relax
  \ifx     \next l\LL!=\Thickness \let\continue=\skipnexttoken!
  \else\ifx\next u\UU!=\Thickness \let\continue=\skipnexttoken!
  \else\ifx\next d\DD!=\Thickness \let\continue=\skipnexttoken!
  \else\ifx\next r\RR!=\Thickness \let\continue=\skipnexttoken!
  \else\def\continue{\vsq!\LL!\UU!\DD!\RR!}%
  \fi\fi\fi\fi
  \continue}

\def\skipnexttoken!#1{\vvsq!}

\def\place#1#2#3{\vbox to 0pt{\vss
\rlap{\kern#1\squaresize
  \raise#2\squaresize\hbox{$#3$}}
\vss}}

\catcode `!=12

\makeatletter

\newcommand{\Rmnum}[1]{\expandafter\@slowromancap\romannumeral #1@}
\makeatother

\begin{document}

\sloppy
 \maketitle
 

\tableofcontents

 \begin{abstract}
 This paper deals with two GUE-matrices, coupled together through some inequalities between the spectra of the first few (small) principal minors. The main results of the paper is to  show that the spectra of the principal minors of these coupled matrices behave statistically as the domino tilings of finitely overlapping Aztec diamonds when their sizes get very large, with horizontal and vertical dominos being equally likely. This extends naturally a result of Johansson and Nordenstam \cite{JoNo}, stating that the spectra of the principal minors of a GUE-matrix behave statistically as domino tilings of an Aztec diamond, near the middle of its edge. Given the spectra of the two coupled matrices, the joint spectra of the underlying principal minors of the two GUE-matrices are uniformly distributed in a certain ``double cone". In particular, this leads to two GUE-matrices sharing the same real line, with one spectrum being completely  to the left of the other spectrum; this gives a new and simple extension of GUE.  Also notice that all statements concerning these coupled random matrices have a domino tiling counterpart.  
  \end{abstract}

\section{Introduction}

Given a random matrix, the problem of knowing, not only the statistics of the spectrum, but also of the statistics of the joint spectra of its principal minors, has been studied in many different contexts. In particular Baryshnikov \cite{Bary} has shown that, given the spectrum of a GUE-matrix, the interlacing spectra of its principal minors are uniformly distributed in a cone. 

Along a different line of thought, it is also known that domino tilings of Aztec diamonds, when the size gets large, have a  stochastic behavior within a arctic circle and a brick-like pattern outside, with Airy process fluctuations near generic points of the circle; see \cite{EKLP,EKLP2,JPS,FS03,Johansson3,Jo02b,Jo03b}. Since the work of Johansson and Nordenstam \cite{JoNo}, it is known that random domino tilings of large size Aztec diamonds behave, near the point of tangency of the inscribed arctic circle with the edge of the Aztec diamond, like the spectra of the principal minors of a GUE-matrix. This was also done in a more prosa\"ic context in the work of Okounkov and Reshetikhin \cite{OR}. 

In this paper, we introduce two GUE-matrices, which are coupled together through some inequalities between the spectra of the first few (small) principal minors; these constraints define a coupling between two GUE-matrices. One of the main results of the paper is to  show that the spectra of the principal minors of these coupled matrices behave statistically as the domino tilings of finitely overlapping Aztec diamonds, near the overlap, when their sizes get very large. The latter model and an associated determinantal point process have been investigated in \cite{ACJvM}, and will be called the {\em edge-tacnode process}. It turns out this gives a very natural extension of Baryshnikov's result, summarized as follows: given the spectra of the two matrices, the joint spectra of the underlying principal minors of the two GUE-matrices are uniformly distributed in a certain ``double cone", to be specified later.

Opposed to this, when the amount of overlap tends to infinity in the same way as the size of the diamond, and such that the arctic circles merely touch, then one is led to the {\em tacnode process} statistics, studied in many different situations \cite{AFvM12,AJvM,DKZ10,Joh10,FV}; the latter will not be considered in this work.

Consider a matrix $  A\in \mbox{GUE}(n)$ and the random spectra  ${\bf x}^{(k)}=(x^{(k)}_1\geq\dots\geq x^{(k)}_k) \in \BR^k$ of its principal minors $A^{(k)}$. The consecutive sets of eigenvalues are known to interlace; this is denoted by ${\bf x}={\bf x}^{(n)}\succ {\bf x}^{(n-1)}\succ \ldots \succ {\bf x}^{(1)}$. 
Consider uniform measure 
 $$
 d\mu_{{\bf x}^{(n)}}=\prod_1^{n-1} d{\bf x}^{(k)}\Id_{{\bf x}^{(k+1)}\succ {\bf x}^{(k)} }$$ on the cone $${\mathcal C}_{\bf x}=\{({\bf x}^{(n-1)}, \ldots , {\bf x}^{(1)}) \in \BR^{n(n-1)/2}\mbox{ such that }{\bf x} \succ {\bf x}^{(n-1)}\succ \ldots \succ {\bf x}^{(1)}\}.
 $$
%
 As mentioned, according to \cite{Bary}, upon fixing the spectrum ${\bf x}$ of the GUE-matrix $A^{(n)}$, the spectra
  ${\bf x}^{(k)}$ of the minors $A^{(k)}$ for $1\leq k\leq n-1$ are uniformly distributed, with the volume expressed in terms of the Vandermonde,
%
\be\begin{aligned}
\BP^{\mbox{\tiny GUE}}\left(\bigcap_{k=1}^{n-1}\{{\bf x}^{(k)}\in d{\bf x}^{(k)}\}~\Bigr|~
{\bf x}^{(n)}={\bf x}
 \right)=
  \frac{d\mu_{\bf x} }{\mbox{Vol}( {\mathcal C}_{\bf x}  ) } 
, 
 \mbox{with }\mbox{Vol}( {\mathcal C}_{\bf x} )=\frac{ \Dt_n({\bf x})}{\prod_1^{n-1}k!}, 
   \end{aligned}\label{Bary1}\ee
and thus we have \be\begin{aligned}
\BP^{\mbox{\tiny GUE}}\Bigl(\bigcap_{k=1}^n\{{\bf x}^{(k)}\in d{\bf x}^{(k)}\}
 \Bigr)=
 \rho_n^{ \mbox{\tiny GUE}}({\bf x}^{(n)})d{\bf x}^{\!(n)}
  \frac{d\mu_{\bf x}}{\mbox{Vol}( {\mathcal C}_{\bf x}  ) } 
,
  \end{aligned}\label{Bary1'}\ee
 in terms of the distribution of the spectrum of a GUE$(n)$-matrix:
 \be
\rho_n^{ \mbox{\tiny GUE}}({\bf x}):=c_n
\Dt_n^2({\bf x}) \prod_1^n e^{-x_i^2},~~~\mbox{with}~~
c_n=\frac {2^{n^2/2}}{(2\pi)^{n/2}\prod_1^{n-1} j!}.
\label{GUE}
\ee
Baryshnikov's uniform distribution result follows, in a very indirect way, from taking a continuous limit of a uniform  probability result on a discrete model involving semi-standard Young tableaux and a probability measure, inherited via the RSK-correspondence.

As mentioned before, based on the Johansson's work \cite{Jo02b,Johansson3}, Johansson and Nordenstam \cite{JoNo} show that the statistics of the eigenvalues ${\bf x}^{(k)}$ of the successive principal minors are given by a determinantal process induced by domino tilings of Aztec diamonds in the scale $n^{-1/2}$, when the size $n$ of the diamond gets very large.
Recall from \cite{JoNo} and \cite{OR} the GUE-minor kernel describing the determinantal process of the interlacing eigenvalues of the successive principal minors of a GUE-matrix $A^{(n)}$; i.e., \footnote{with  $\BH^{m}(z)$ defined for $m\geq 1$ as
$
\begin{aligned}
   \BH^{m}(z)&:=\Id _{z\geq 0} {z^{m-1}}/{(m-1)!}
 \end{aligned} 
 $. Here and below we shall define functions, which involve integration over small circles $\Gamma_0$ and an imaginary line $L:=0^{+}+i\BR\! \uparrow$; the line $L$ needs to be always to the right of the contour $\Gamma_0$}
\be
\begin{aligned}
 \BK^ {\mbox{\tiny minor}}  (n,x;n',x')   
 := &-\Id_{n>n'}2^{n-n'}\BH^{n-n'}(x-x')
\\
&+\frac{2}{(2\pi \I)^2}
\int_{\Gamma_0}dz\int_{L
 } \frac{dw}{w-z}
 \frac{e^{-z^2+2zx}}{e^{-w^2+2wx'}}\frac{w^{n'}}{z^n}
.\end{aligned}
\label{8.6}\ee

\noindent We now consider in some detail the following two models: 

\bigbreak

{\bf I.   Two coupled GUE-matrices:} 
Given a real number $\beta\geq 0$, consider two independent GUE-matrices $A$ and $B$,
$$
A\in \beta .\Id_{\infty}+\mbox{GUE}(\infty),~,~B \in -\beta .\Id_{\infty}+\mbox{GUE}(\infty)
.$$
Then, for a fixed $k\geq 1$, its principal minors $A^{(k)}$ and $B^{(k)}$ of size $k$ are independent $\pm \beta .\Id_{k}+$ GUE$(k)$ matrices, with spectra ${\bf x}^{(k)}=(x_1^{(k)}\leq\dots\leq x_k^{(k)})$ and ${\bf y}^{(k)}=(y_1^{(k)}\leq\dots\leq y_k^{(k)})$, satisfying
the well-known GUE-distribution $\rho_k^{ \mbox{\tiny GUE}}({\bf x}-\beta)$ and $\rho_k^{ \mbox{\tiny GUE}}({\bf y}+\beta)$.  \footnote{${\bf x}-\beta=(x_1-\beta,\ldots,x_k-\beta)$.} 

Given integers $n\geq \rho\geq 1$,  consider the spectra 
%
\be \ldots,{\bf y}^{(n)},\ldots,{\bf y}^{(\rho )},~{\bf y}^{(\rho-1)}\cup {\bf x}^{(1)} ,\ldots, {\bf y}^{(1)}\cup {\bf x}^{(\rho-1)}  ,~{\bf x}^{(\rho )},\ldots, {\bf x}^{(n)},\ldots\label{spectra1}\ee
of the successive minors $B^{( i )} $ and $A^{(i)}$, with the   combined spectra of 
 $B^{(i)}$ and $A^{(\rho-i )} $ for $1\leq i\leq \rho-1$, {\em satisfying $\rho$ spectral inequalities},
\be \begin{aligned}
  \max \mbox{\em spec}(B^{(i)})\leq \min \mbox{\em spec}(A^{(\rho-i+1)})
   \mbox{ ,  for $1\leq i\leq \rho$},
  \end{aligned}\label{constraint}\ee
equivalent to (as illustrated in Figs. 1 and 3)
$$   {  y}^{(i)}_i\leq x ^{(\rho-i+1)}_1 , \mbox{    for $1\leq i\leq \rho$}  . $$

  \newpage
   
     \vspace*{-5cm}
  
 \hspace*{  2.7cm}  {\includegraphics[width=140mm,height=260mm]{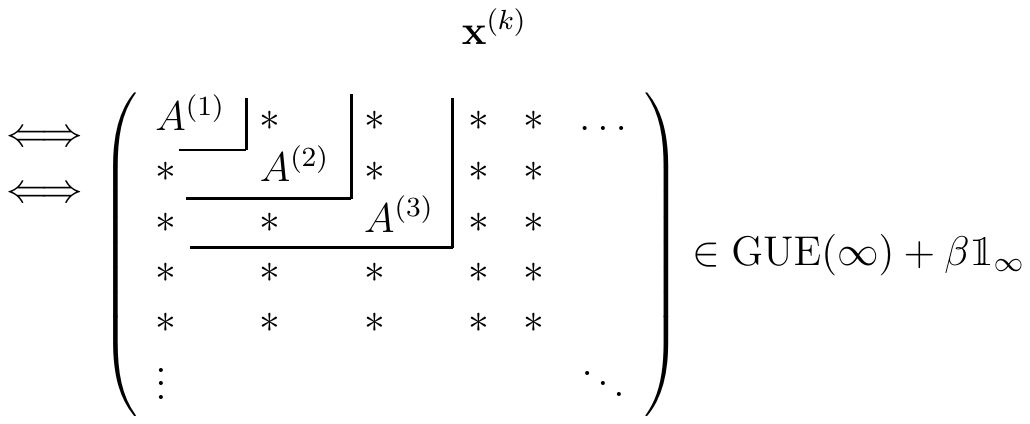}}

  \vspace*{-28cm}
  
  \hspace*{-4.5cm}    {\includegraphics[width=140mm,height=260mm]{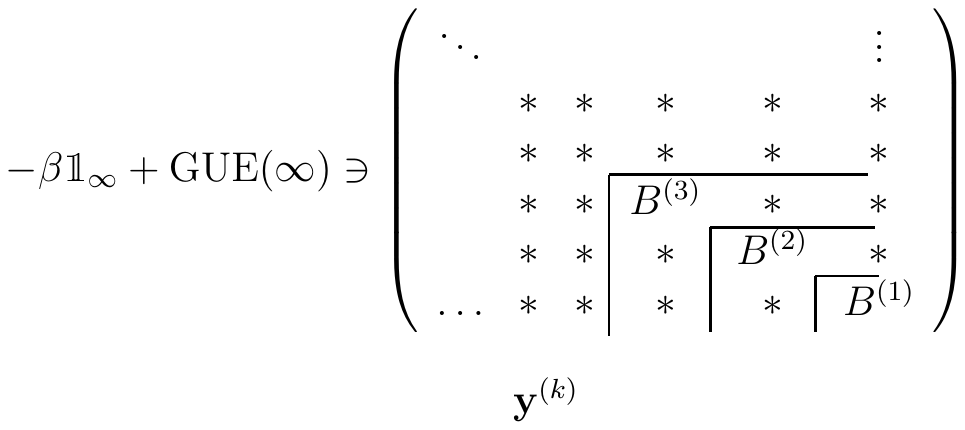}}

   \vspace*{-16.4cm}
 \noindent {\footnotesize Figure 1. Two matrices $B$ and $A$ with interacting minors for $\rho=3$: ~ $\mbox{spec} ~B^{(3)}\leq \mbox{spec} ~A^{(1)}$, ~$\mbox{spec} ~B^{(2)}\leq \mbox{spec} ~A^{(2)}$,~$\mbox{spec} ~B^{(1)}\leq \mbox{spec} ~A^{(3)}$. Consider the spectra ${\bf y}^{(k)}$ and ${\bf x}^{(k)}$ of the principal minors of $B$ and $A$, combined as indicated by the arrows:  $\ldots,{\bf y}^{(3)}, ~{\bf y}^{(2)}\cup {\bf x}^{(1)} , ~ {\bf y}^{(1)}\cup {\bf x}^{(2)} ,~{\bf x}^{(3 )},\ldots.$}

  \bigbreak

 This conditioning implies a coupling of the matrices $A$ and $B$, so that the probability of any event $E$ will  in effect be a conditional probability 
\be
\BP(E):=\BP^{\mbox{\tiny GUE}}\left( E ~ \Bigr|~\bigcap_1^\rho \{ {  y}^{(i)}_i\leq x ^{(\rho-i+1)}_1   \}\right),
 \label{condProb}\ee
with $\BP^{\mbox{\tiny GUE}}$ denoting the probability of $A$ and $B$ viewed as independent GUE-matrices.

\bigbreak

{\bf II. Double Aztec Diamonds:}   
   As alluded to before, consider domino tilings of two overlapping Aztec diamonds $A$ and $B$, each of size $N$, with an overlap of size $\rho$, with $B$ shifted below $A$ by merely one square; this double Aztec diamond model was introduced in \cite{AJvM}. See Figs. 2 and 4 for an illustration: the diamond $A$ ($B$) corresponds to the blue side (red side). 
   
   In analogy with the single Aztec diamond, domino tilings of double Aztec diamonds lead to a determinantal process of particles as well. Here also,  macroscopicaly there are frozen regions within two arctic circles or ellipses (overlapping or not) inscribed in the two diamonds when the size $N$ gets large. We will be interested in the situation, where the overlap $\rho$ of the two diamonds remains fixed, while $N\to \infty$; this problem was studied in \cite{ACJvM}.  
  For finite $N$ and $\rho$, the domino tiling leads to a determinantal process, which  in the limit $N\to\infty$, keeping the overlap $\rho$ finite, tends to the so-called  {\em edge-tacnode process}, another determinantal process; a summary of \cite{ACJvM} will be presented in section 3.
  
     This process consists of dot-particles ($\in \BR$) as in the simulation of Fig. 2, blue dots ${\bf x}^{(u)}$ and red dots ${\bf y}^{(u)}$ belonging to a sequence of vertical lines through the left-hand diamond $A$  and through the right-hand diamond $B$ respectively, with $\rho-1$ vertical lines in the overlap region containing $\rho$ dots, some red, some blue, exactly as careful inspection of Fig. 2 would reveal. The index $(u)$ refers to the vertical lines.  We assemble the blue dots ${\bf x}$'s and the red dots ${\bf y}$'s as follows:  ($2\leq \dt\leq \rho-2$)
  \be \ldots,{\bf y}^{(n)},\ldots,{\bf y}^{(\rho )},{\bf y}^{(\rho-1)}\cup {\bf x}^{(1)} ,\ldots, {\bf y}^{(\rho-\dt)}\cup {\bf x}^{(\dt)} ,\dots, {\bf y}^{(1)}\cup {\bf x}^{(\rho-1)}  ,{\bf x}^{(\rho )},\ldots, {\bf x}^{(n)},\ldots\label{spectra2}\ee

\newpage

   \hspace*{2cm} \includegraphics[height=2in]{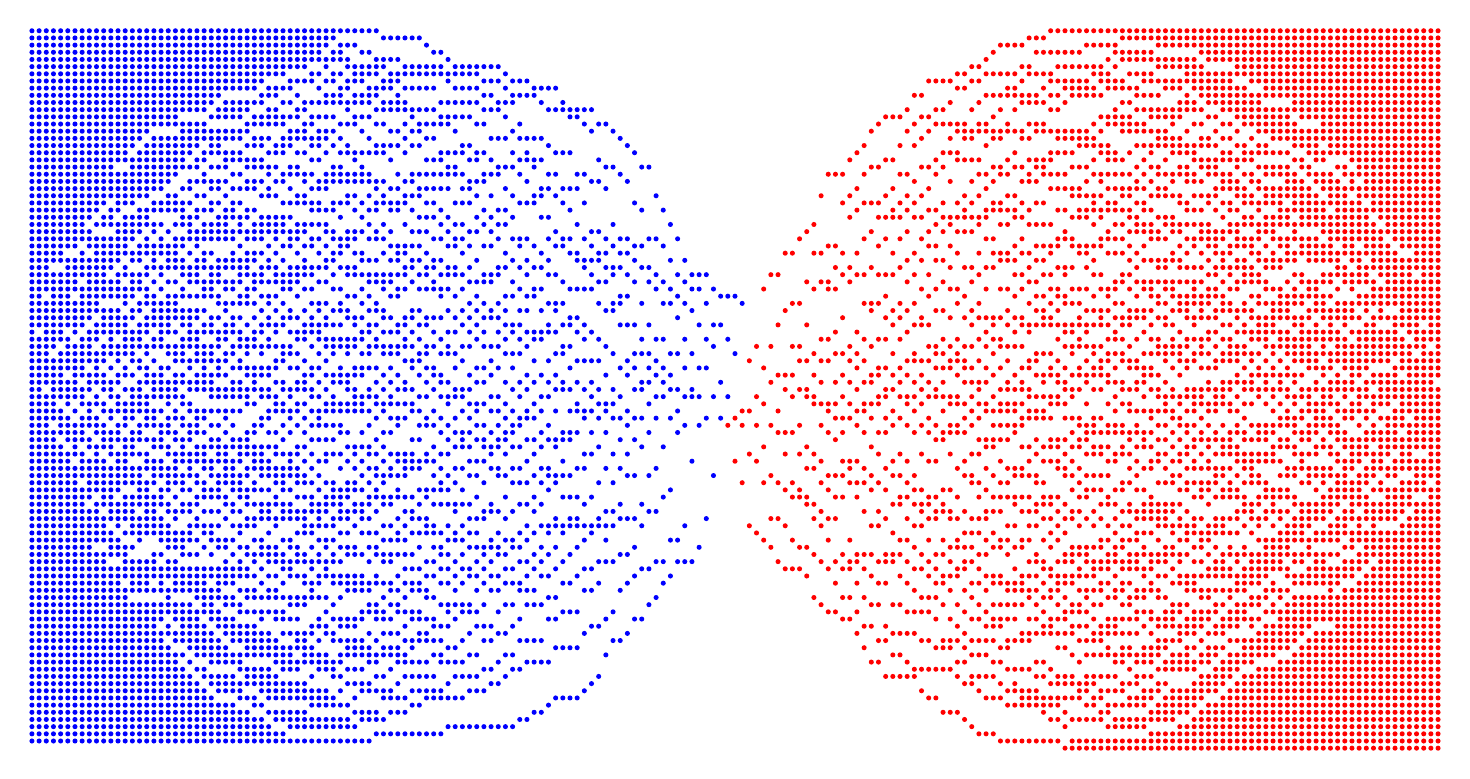}
  
   \hspace{ .1cm}
 
   \setlength{\unitlength}{0.030in}\begin{picture}(0,0)
   \thinlines
                \put(91.0,20){\line(0, 1){50}}
                 \put(93.0,20){\line(0, 1){50}}
                  \put(95.0,20){\line(0, 1){50}}
        \put(45,80){\makebox(0,0) {  $\mbox{\tiny Blue ${\bf x}$-dots~~~~~~~A}$ }}
        \put(135,80){\makebox(0,0) {  $\mbox{\tiny B~~~~~~~~~~~Red ${\bf y}$-dots}$ }}
 \end{picture}

\vspace*{-1cm}

   \noindent {\footnotesize Figure 2. The top picture shows the blue and red $\mathbb{L}$-particles, ${\bf x}^{(k)}$ and ${\bf y}^{(k)}, $ in a simulation for $n=100$ and $\rho=4$. Since $n$ is very large compared to $\rho$, this represents the edge-tacnode process, as given by the blue and red dots belonging to consecutive vertical lines. Notice the two arctic circles and their near-tangency in the overlap of the diamonds. (Courtesy of Sunil Chhita)}
  
  \vspace*{.5cm}

   %
   %
   More details on this model will be explained in section 2. In viewing the process as a determinantal process with levels $u\in \BZ $, it will be convenient to reparametrize the sequences (\ref{spectra1}) and (\ref{spectra2}) by means of a single notation ${\bf z}^{(u)}$, $u \in \BZ$, with ${\bf z}^{(0)}={\bf y}^{(\rho)}$ and ${\bf z}^{(\rho)}={\bf x}^{(\rho)}$, as follows:
   \be
   \dots,{\bf z}^{(\rho-n)}, \ldots,{\bf z}^{(0 )},{\bf z}^{(1 )},\ldots,{\bf z}^{(\dt )},\dots,
   {\bf z}^{(\rho-1 )}, {\bf z}^{(\rho )},\dots, {\bf z}^{(n )},\ldots
 \label{spectra3} \ee
 The two notations, on the one hand (\ref{spectra1}) or (\ref{spectra2}) and on the other hand (\ref{spectra3}) will be used invariably throughout the paper. Sometimes (\ref{spectra1}) or (\ref{spectra2}) will be more convenient and sometimes (\ref{spectra3}) will be handier. 
   \medbreak
  
 The main message of this paper is the following equivalence: 
  $$\left\{\begin{array}{l}
 \mbox{\em The statistics of the eigenvalues (\ref{spectra1})}
 \\
 \mbox{\em of the consecutive minors $A^{(i)}$ and $B^{(i)}$}
\\ \mbox{\em of the coupled GUE-matrices $A$ and $B$}
 \end{array}\right\}
  \Leftrightarrow 
 \left\{\begin{array}{l}
 \mbox{\em The statistics of the }
 \\\mbox{\em dot-particles (\ref{spectra2}) in the}
 \\
 \mbox{\em double-Aztec diamonds}
\\ \mbox{\em when its size $N\to \infty$}
  \end{array}\right\}
  $$


The analogue of Baryshnikov's result here is that, given the spectra ${\bf x}={\bf x}^{(n)}$ and ${\bf y}={\bf y}^{(n)}$ of the coupled matrices $A$ and $B$ for $n\geq \rho$, the intermediate spectra in the sequence 
(\ref{spectra2}) interlace and are uniformly distributed within a certain double cone ${\mathcal C}^{(n)}_{{\bf x},{\bf y}}$ to be defined in (\ref{cone}) and (\ref{unifmeas}); see Theorem \ref{Th1.2}. For the diamonds, the proof of this statement and the interlacing will be inherited from taking a continuous limit of a discrete uniformity result on a determinantal point process defined on double Aztec diamonds, while for the coupled GUE-matrices it is inherited from the single GUE result. { So, the results follow from replacing the single Aztec diamond model, including RSK and Young diagrams by this double Aztec diamond model.}

The uniform measure on the double cone is an indication that ``statistically speaking" the double Aztec diamonds are pushed together very gently, so that the two arctic circles separating the frozen and stochastic region merely touch, without getting much deformed, as appears from Fig. 2.

In \cite{ACJvM}, the edge-tacnode kernel was obtained, as a perturbation of the GUE-minor  kernel (\ref{8.6}) and has the following form, with a kernel ${\cal K}^\beta ( \lambda ,\kappa )$ and functions ${\cal A}^{\beta,z }_{u  }(\kappa)$ and ${\cal B}^{\beta,z }_{v }(\lambda)$ to be defined in (\ref{defAB})\footnote{In terms of the notation $M_{[i,j]}:=\left( M(\lb,\kappa)\right)_{i\leq \lb, \kappa\leq j}.$}: 
\be
\begin{aligned}
\BK_{\beta, \rho}^  {\mbox{\tiny TAC}} (u_1,z_1;u_2,z_2)
 &=  \BK^ {\mbox{\tiny minor}}  (u_1,\beta-z_1 ;~u_2 ,\beta-z_2 ) \\
&~~~+2\Bigl\la (\Id \! -\! {\cal K}^\beta ( \lambda ,\kappa )) _{_{\mbox{\tiny $  [-1,-\rho]$}}} ~{\cal A}^{\beta,z_1-\beta}_{u_1 }(\kappa), {\cal B}^{\beta,z_2-\beta}_{u_2 }(\lambda)\Bigr\ra
  _{_{ \ell^2(-\rho,\ldots,\infty) }}.
 \\
 \end{aligned}\label{dGUE}\ee

\noindent We now state three main Theorems, which will be valid for each of the two models, whether they be eigenvalues of the GUE-minors model, or the dot-particles in the double Aztec diamond model. The first statement  concerns interlacing and an associated determinantal point process, the second one deals with the joint ${\bf x}^{(n)}$ and ${\bf y}^{(n)}$-distribution in terms of uniform measure on the double cone $ {\mathcal C}^{(n)}_{{\bf x},{\bf y}}$  and the last one provides the distribution of the eigenvalues ${\bf x}^{(n)}$ or ${\bf y}^{(n)}$ level by level.

\begin{theorem} \label{Th1.1} Given integers $n\geq \rho\geq 1$, the ${\bf z}^{(u)}$'s, defined in (\ref{spectra3}), with
$$
{\bf z}^{(u)}\in \left\{\begin{aligned}
&\BR^{u}\mbox{ for $ u \geq \rho$}
\\
&\BR^{ \max(\rho-u,\rho)}\mbox{ for $ u\leq \rho $}
\end {aligned}\right.$$
  interlace as follows\footnote{${\bf z}^{(\rho)}\prec\ldots\prec {\bf z}^{(n)}$ and ${\bf z}^{(0)}\prec\ldots\prec {\bf z}^{(\rho-n)}$ both intertwine in the usual way, and otherwise we use the   
 intertwining $ {\bf z} \curlyeqprec {\bf z}'$,  meaning that  both vectors $ {\bf z}$ and  $ {\bf z}' \in \BR^{\rho}$ intertwine, and satisfy ${  z}_{\rho} \leq {  z}'_{\rho}$.}, (see Fig. 3)
  \be\begin{aligned}
 {\bf z}^{(\rho-n)}\succ\ldots \succ {\bf z}^{(0)}\curlyeqprec   {\bf z}^{(1)} 
\curlyeqprec\ldots \curlyeqprec 
  {\bf z}^{(\rho-1)}
\curlyeqprec {\bf z}^{(\rho)}\prec\ldots\prec {\bf z}^{(n)} 
.\end {aligned}\label{xy-interlace}\ee
  This sequence of vectors ${\bf z}^{(u)}$ form a determinantal point process with kernel given by the edge-tacnode kernel $\BK_{\beta, \rho}^  {\mbox{\tiny TAC}} (u_1,z_1;u_2,z_2)$.
  %

\end{theorem}

 In view of (\ref{spectra1}) and (\ref{constraint}), it is natural to consider the following ``{\it double cone}", given arbitrary ${\bf x}={\bf x}^{(n)}={\bf z}^{(n)}$ and ${\bf y}={\bf y}^{( n)}={\bf z}^{(\rho-n)}\in \BR^n$,
%
 %
   \be
   {\mathcal C}^{(n)}_{{\bf x},{\bf y}} =\left\{\begin{array}{l}  {\bf y}  \succ {\bf y}^{(n-1)} \succ\ldots \succ    {\bf y}^{(1)}\mbox{  for any ${\bf y}^{(i)}\in \BR^i$ }
   \\
   {\bf x} \succ {\bf x}^{(n-1)}\succ \ldots \succ {\bf x}^{(1)} \mbox{  for any ${\bf x}^{(i)}\in \BR^i$}\\
    \mbox{and for $1\leq i\leq \rho$, subjected to}
   \\
   \max  ({\bf y}^{(i)})={  y}^{(i)}_i\leq {  x}^{(\rho-i+1)}_1=\min  (  {\bf x}^{(\rho-i+1)}) \\ 
 \end{array} \right\} \subset \BR^{n(n-1)} .\label{cone}\ee
 The superscript on ${\cal C}^{(n)}$ will be omitted, when unnecessary. The interlacing conditions (\ref{xy-interlace}) and (\ref{cone}) are obviously equivalent.
 Then consider
 uniform measure on ${\mathcal C}_{{\bf x},{\bf y}}$,
 \be
  \begin{aligned}d\mu_{{\bf x},{\bf y}}&( {\bf x}^{(n-1)} ,\ldots,{\bf x}^{( 1)},{\bf y}^{(n-1)},\ldots,{\bf y}^{( 1)})
  :=d\mu_{\bf x}d\mu_{\bf y}
  \prod_{i=1}^{\rho } \Id_{y_i^{(i)}\leq x_1^{(\rho-i+1)}}   
 \end{aligned},
 \label{unifmeas}\ee
and the volume $
\mbox{ Vol} ({\mathcal C}^{(n)}_{{\bf x},{\bf y}})=\int_{\BR^{n(n-1)}}d\mu_{{\bf x},{\bf y}}
$ of the cone ${\mathcal C}^{(n)}_{{\bf x},{\bf y}}$. 
Also define the kernel
$$\begin{aligned}
{\cal K}^\beta (\lambda,\kappa)&:=\oint_{\Gamma_0}  \frac{d\zeta} {(2\pi \I)^2 }\int_{ 0^{+}+i\BR\uparrow} \frac{d\om}{ \om-\zeta   }
   \frac{e^{-2\zeta^2+4 \beta \zeta }} { e^{-2\omega^2+4 \beta \omega }}
\frac{\zeta^{\kappa } }{\omega^{\lambda +1}}.
\end{aligned}
$$

\bigbreak

\noindent This enables us, as already mentioned, to state the analogue of Baryshnikov's result, namely that fixing the spectra ${\bf x}$ and ${\bf y}$ of the matrices $A$ and $B$ in the coupled GUE-matrix model {\bf (I)}, the intermediate spectra form a uniformly distributed double cone (\ref{cone}). Moreover this fact holds as well for model {\bf (II)}. 

 \begin{theorem} \label{Th1.2}
  For both models {\bf I} and {\bf II}, we have
  \be\begin{aligned}
 \BP& \left(
  \bigcap_{1}^{n-1}
  \{{{\bf x}^{(k)}\in d{\bf x}^{(k)} ,~{\bf y}^{(k)}\in d{\bf y}^{(k)}  }  \}\Bigr|~{{\bf x}^{(n)}={\bf x}^{ } ,~{\bf y}^{(n)}=  {\bf y}^{ }  }
  \right)= \frac {    d\mu_ {{\bf x} ,{\bf y}} }{  \mbox{  Vol} ({\mathcal C}^{(n)}_{{\bf x}{\bf y} })  },
   \end{aligned}
 \label{14}\ee
with an explicit expression for $\mbox{  Vol} ({\mathcal C}^{(n)}_{{\bf x}{\bf y} })$ given by formula (\ref{vol1}) in Theorem \ref{Th:JProb}. The joint density of  ${\bf x}^{(n)}, \ldots , {\bf x}^{(1)} $ and $    {\bf y}^{(n)}, \ldots , {\bf y}^{(1)}$  
  is given by, setting ${\bf x}^{(n)}={\bf x}^{ }$ and ${\bf y}^{(n)}={\bf x}^{ }$, 
 %
%
  %
 \be\begin{aligned}
 \BP& \left(
  \bigcap_{1}^{n}
  \{{{\bf x}^{(k)}\in d{\bf x}^{(k)} ,~{\bf y}^{(k)}\in d{\bf y}^{(k)}  }  \}
   %
  \right)
  =\frac{\rho_n^{ \mbox{\tiny GUE}}({\bf x}\!-\!\beta) 
   \rho_n^{\mbox{\tiny GUE}}( {\bf y}\!+\!\beta)}
 {\det(\Id- {\mathcal K}^{\beta}
 )_{[-1, -\rho]} }
 \frac { d{\bf x}d{\bf y}  d\mu_ {{\bf x} ,{\bf y}} }{\mbox{\em Vol} ({\mathcal C} _{{\bf x} }) \mbox{\em Vol} ({\mathcal C}_{ {\bf y} }) }
   .\end{aligned}
 \label{15}\ee
As a by-product, given two independent GUE-matrices $A$ and $B$, one has the following expression for the probability (see (\ref{condProb}) for the distinction between $\BP $ and $\BP^{\mbox{\tiny GUE}}$):
      \be\begin{aligned}\BP^{\mbox{\tiny GUE}}\Bigl(\bigcap_1^\rho &\left\{ \max \mbox{\em spec}(B^{(i)})\leq \min \mbox{\em spec}(A^{(\rho-i+1)})
\right\}\Bigr)
=\BP^{\mbox{\tiny GUE}} \left(\bigcap _1^{\rho}\{y_i^{(i)}\leq  x_{1}^{(\rho-i+1)} 
 \}\right)
\\&
 \!\!\!\!=\det(\Id- {\mathcal K}^{\beta}
 )_{[-1, -\rho]}   
 \end{aligned}\label{ineq}\ee
$${\footnotesize
=\det \left(
\begin{array}{cccccccccc}
a_0- {\widetilde c}_{0}& -{\widetilde c}_{-1}& -{\widetilde c}_{-2}&\ldots
&\dots&-{\widetilde c}_{ 1-\rho}
\\
a_1- {\widetilde c}_{1}&a_0- {\widetilde c}_{0}& -{\widetilde c}_{-1}& -{\widetilde c}_{-2}&\ldots
&-{\widetilde c}_{ 2-\rho}
\\
~~~~\ddots&~~~~~~\ddots&~~~~~~\ddots&
    \\ &&&&\ddots\\&&&\ddots&\ddots&   -{\widetilde c}_{-1}
\\  a_{\rho-1}- {\widetilde c}_{\rho-1}&&&&a_1- {\widetilde c}_{1}&    a_0- {\widetilde c}_{0}%
\end{array}
\right),
}$$
 
  \newpage

\vspace*{-8cm}    
    \hspace*{-2cm}    
 { \includegraphics[width=200mm,height=300mm]{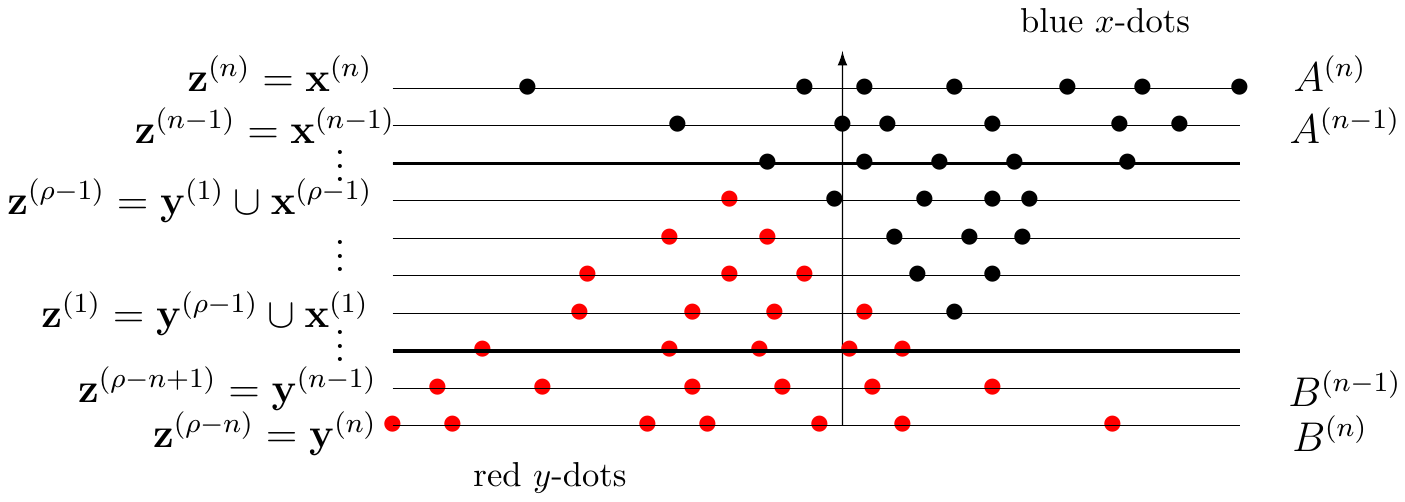}}

 \vspace*{-19cm}    
{\footnotesize\rm Figure 3. The interlacing set of blue and red eigenvalues ($\in \BR$) of the minors $A^{(k)}$ and $B^{(k)}$, subjected to the constraint (\ref{constraint}), or the blue and red dots in the double Aztec diamond model, for $\rho=5$.}

\vspace*{1cm}

  %

\noindent where this Toeplitz matrix consists of quantities $a_i$ and $\widetilde c_i  $, defined in (\ref{Cinv}) and (\ref{ck}).
\end{theorem}




\remark  ~
Formula (\ref{14}) holds as well for two arbitrary levels, not necessarily symmetric with regard to the overlap; it suffices to replace the cone $   {\mathcal C}^{(n)}_{{\bf x}{\bf y} }  $ by the cone defined in the same way, but for ${\bf x}\in \BR^n$ and ${\bf y}\in \BR^m$ belonging to different levels $n$ and $m$. The same is true for probability (\ref{15}) and will also be true for probability (\ref{joint}). This fact will follow after showing that the two models are equivalent by repeating the arguments of section 2. 

 \bigbreak

  The next statement deals with the distribution of the ${\bf x}^{(n)}$, the ${\bf y}^{(n)}$ for $n\geq \rho$ and the ${\bf y}^{(i)}\cup {\bf x}^{(\rho-i)}  $ for $1\leq i\leq \rho-1$, as in (\ref{xy-interlace}).
%
%
This will involve Gaussian moments on a half-line
$$\Phi_n(\eta):= \frac{2^n}{\sqrt{\pi} n!} \int^\infty_0 \xi^n e^{-(\xi-\eta)^2} d\xi 
,~~~~n\geq 0,$$
the usual Vandermonde $\Dt_n(x)$ and an extension for $0\leq\rho<n$, widely used in this work:
\be
 \begin{aligned}
 \widetilde \Dt_{n,\rho}^{\beta \pm}(x)&:= 
 \det\left(\begin{aligned}
& 1
\\
&x  \\
& \vdots  
\\
&x ^{n-\rho-1} 
 \\
&\Phi_{n-\rho} (\beta\pm x ) \\
&\vdots  \\
& \Phi_{n-1} (\beta\pm x ) \\
 \end{aligned}
\right)_{x=x_1,\ldots, x_{n}},~~\mbox{with}~~
\widetilde \Dt_{n,0}^{\beta \pm}(x)=\Dt_n(x). \end{aligned} \label{vander2}\ee
The distributions (\ref{outover}) and (\ref{Kout}) for the eigenvalues or the dots at a fixed level is a somewhat surprising extension of GUE, of interest in its own right:
 
\begin{theorem}\label{Th1.4}
   For both models  and {\bf  outside the overlap}   ($n\geq \rho$) , the density of ${\bf z}^{(n)}=
   {\bf x}^{(n)}$ 
   and ${\bf z}^{(\rho-n)}=
   {\bf y}^{(n)}$ 
   is given by\footnote{
    $c_{\rho,\delta} :=\prod_1^{\dt-1} \tfrac {2^\ell}{\ell !}  (\det(\Id- {\mathcal K}^{\beta})_{[-1,-\rho]} )^{-1}
  $ and $c_{\rho,0} :=   (\det(\Id- {\mathcal K}^{\beta})_{[-1,-\rho]} )^{-1}
  $.}, setting $\dt=n-\rho\geq 0$,         
  \be
\begin{aligned}
\BP& \left(
 {\bf x}^{(n)}\in d{\bf x}
 \right) =  c_{\rho,\delta}  \widetilde \Dt^{\beta+}_{n,\rho}(x) \Dt_{n}(x)  \prod _1^{n}
\frac {e^{-(x_i-\beta)^2} }{\sqrt{\pi}} d{\bf x}
\\
 \BP& \left(
 {\bf y}^{(n)}\in d{\bf y}  
  \right) 
 =c_{\rho,\dt} \Dt_{n}(y) \widetilde \Dt^{\beta-}_{n,\rho}(y)\prod _1^{n}
\frac {e^{-(y_i+\beta)^2} }{\sqrt{\pi}}d{\bf y} , \end{aligned}
 \label{outover}\ee
  whereas {\bf in the overlap} ($1\leq \dt\leq  \rho-1$) the joint density of the points ${\bf z}^{(\dt)}=({\bf y}^{(\rho-\dt)},{\bf x}^{(\dt)})$, 
   taken together, is given by  (with ${\bf z}\in \BR^{^\rho}$) 
  \be
 \begin{aligned}
 \BP&(\{({\bf y}^{(\rho-\dt)},{\bf x}^{(\dt)})\in d{\bf z}\})
 \\ \\& = {c_{\rho,0}}  
  \det
  \left( \begin{array}{c}
\frac {z^{\dt-1}}{\sqrt{\pi}} e^{-(\beta-z)^2}\\
 \vdots
 \\
  \frac {1}{\sqrt{\pi}} e^{-(\beta-z)^2}\\
  \textcolor[rgb]{1.00,0.00,0.00}{\Phi_{0}(\beta-z)}\\
   \vdots\\
   \textcolor[rgb]{1.00,0.00,0.00}{ \Phi_{\rho-\dt -1}(\beta-z)}
 \end{array} \right)_{z=z_1,\ldots,z_\rho}
 \times ~
  \det \left( \begin{array}{c}
 \textcolor[rgb]{1.00,0.00,0.00}{\frac {z^{\rho-\dt-1}}{\sqrt{\pi}} e^{-(z+\beta )^2}}\\
 \vdots
 \\
  \textcolor[rgb]{1.00,0.00,0.00}{\frac {1}{\sqrt{\pi}} e^{-(y+\beta )^2}}\\
  \Phi_{0}(z+\beta)\\
   \vdots\\
    \Phi_{ \dt -1}(z+\beta)
\end{array} \right)_{z=z_1,\ldots,z_\rho}\!\!\!\!\!d{\bf z}
  . \end{aligned}
   \label{Kout}\ee
Finally, the joint density of ${\bf z}^{(n)}=
   {\bf x}^{(n)}$ 
   and ${\bf z}^{(\rho-n)}=
   {\bf y}^{(n)}$ 
    for $n\geq \rho$ is given by
  \be\begin{aligned}
 \BP& \left(
  \{{{\bf x}^{(n)}\in d{\bf x} ,~{\bf y}^{(n)}\in d{\bf y}  }  \}
 \right)
 =\frac{\rho_n^{ \mbox{\tiny GUE}}({\bf x}\!-\!\beta) 
 ~ \rho_n^{\mbox{\tiny GUE}}( {\bf y}\!+\!\beta)}
 {\det(\Id- {\mathcal K}^{\beta}
 )_{[-1, -\rho]}}
 \frac {  \mbox{\em Vol} ({\mathcal C}^{(n)}_{{\bf x}{\bf y} }) d{\bf x}d{\bf y} }{\mbox{\em Vol} ({\mathcal C}_{{\bf x} }) \mbox{\em Vol} ({\mathcal C}_{ {\bf y} })    }
. \end{aligned}
 \label{joint}\ee

   \end{theorem}
   
   The proofs of these statements will be an interplay between the two models, on the one hand the coupled random matrix model and on the other hand the double Aztec model.

   Section 2 will deal with the GUE-matrices with the coupling, whereas in section 3, we discuss what is needed to know about double Aztec diamonds. In particular we will show a uniformity result for double Aztec diamonds. After a notational and technical section 4, we will establish in section 5 the distribution of the particles at any given level, inside as well as outside the overlap. Section 6 will deal with the joint distribution of the particles at two different levels. Section 7 deals with an expression for the volume of the double cone (\ref{cone}). This is crucial in connecting the two models; this connection is done in section 8. 
   
   \medbreak 
   
   The authors acknowledge sharp and useful questions from G\'erard Ben Arous,  and helpful discussions with Alexei Borodin and Neil O'Connell.

\section{The coupled GUE-random matrix model}

In this  section, we prove Theorem \ref{Th1.2} for model {\bf I}.

\noindent {\em Proof of Theorem \ref{Th1.2}:}
Using the independence of the random matrices $A$ and $B$, using (\ref{Bary1'}) and (\ref{unifmeas}) and remembering the conditional probability (\ref{condProb}), one expresses the left hand side of (\ref{15}) as : (setting ${\bf x}={\bf x}^{(n)}$ and ${\bf y}={\bf y}^{(n)}$)
$$
\begin{aligned}
\BP^{\mbox{ }}&\left(\bigcap_{k=1}^n\{ {\bf x}^{(k)}\in d{\bf x}^{(k)},~
 {\bf y}^{(k)}\in d{\bf y}^{(k)}\} 
    \right)
 \\
&=\BP^{\mbox{\tiny GUE}} \left(\bigcap_{k=1}^n\{ {\bf x}^{(k)}\in d{\bf x}^{(k)},~
 {\bf y}^{(k)}\in d{\bf y}^{(k)}\} 
 ~\Bigr |~\begin{array}{l}\max  ({\bf y}^{(i)}) \leq  \min  (  {\bf x}^{(\rho-i+1)})\\
\mbox{   for  }1\leq i\leq \rho \end{array} \right)
\\&=\frac {\BP^{\mbox{\tiny GUE}} \left( \bigcap_{k=1}^n\{{\bf x}^{(k)}\in d{\bf x}^{(k)}\} \right)\BP^{\mbox{\tiny GUE}} \left( \bigcap_{k=1}^n\{{\bf y}^{(k)}\in d{\bf y}^{(k)}\}\right)
\prod_1^\rho \Id_{y_i^{(i)}\leq  x_{1}^{(\rho-i+1)}}}{\BP^{\mbox{\tiny GUE}} \left(y_i^{(i)}\leq  x_{1}^{(\rho-i+1)} ,~
 1\leq i\leq \rho\right)}.
\end{aligned}$$
\be \begin{aligned}
&=\frac{\rho_n^{ \mbox{\tiny GUE}}({\bf x}\!-\!\beta)   \rho_n^{\mbox{\tiny GUE}}( {\bf y}\!+\!\beta) }{\mbox{ Vol} ({\mathcal C}_{{\bf x}  })
\mbox{ Vol} ({\mathcal C}_{ {\bf y} })}   
~~\frac {d{\bf x}^{ } d\mu_{\bf x}   d{\bf y}^{ } d\mu_{\bf y}
 \prod_1^\rho \Id_{y_i^{(i)}\leq  x_{1}^{(\rho-i+1)}}
 }
{\BP^{\mbox{\tiny GUE}} \left(\bigcap _1^{\rho}\{y_i^{(i)}\leq  x_{1}^{(\rho-i+1)} 
 \}\right)}
 \\
 &
  =\frac{\rho_n^{ \mbox{\tiny GUE}}({\bf x}\!-\!\beta) 
   \rho_n^{\mbox{\tiny GUE}}( {\bf y}\!+\!\beta)}
{\BP^{\mbox{\tiny GUE}} \left(\bigcap _1^{\rho}\{y_i^{(i)}\leq  x_{1}^{(\rho-i+1)} 
 \}\right)}
~  \frac { d{\bf x}d{\bf y}  d\mu_ {{\bf x} ,{\bf y}} }{\mbox{ Vol} ({\mathcal C}_{{\bf x} }) \mbox{ Vol} ({\mathcal C}_{ {\bf y} }) }
   .\end{aligned}
  \label{15''}\ee 
This equals the right hand side of (\ref{15}), upon using the definition (\ref{unifmeas}) of uniform measure on the double cone ${\mathcal C}_{{\bf x} {\bf y}} $ and the identity (\ref{ineq}). Integrating the formula just obtained with regard to the intermediate variables ${\bf x}^{(1)},\ldots,{\bf x}^{(n-1)}, 
{\bf y}^{(1)},\ldots,{\bf y}^{(n-1)}$ in the double cone ${\mathcal C}^{(n)}_{{\bf x} {\bf y}} $ gives the probability 
 \be\begin{aligned}
 \BP& \left(
  {{\bf x}^{(n)}\in d{\bf x}^{ } ,~{\bf y}^{(n)}\in d{\bf y}^{ }  }   
   %
  \right)
  =\frac{\rho_n^{ \mbox{\tiny GUE}}({\bf x}\!-\!\beta) 
   \rho_n^{\mbox{\tiny GUE}}( {\bf y}\!+\!\beta)}
 {\BP^{\mbox{\tiny GUE}} \left(\bigcap _1^{\rho}\{y_i^{(i)}\leq  x_{1}^{(\rho-i+1)} 
 \}\right) }
  \frac{  {\mbox{ Vol} ({\mathcal C}^{(n)}_{{\bf x}{\bf y} }) 
 }{ d{\bf x}d{\bf y} } 
 }{{\mbox{ Vol} ({\mathcal C}_{{\bf x} }) \mbox{ Vol} ({\mathcal C}_{ {\bf y} }) }}
 .\end{aligned}
 \label{15'}\ee
 Finally, the conditional probability below is obtained by dividing formula (\ref{15''}) by  (\ref{15'}), yielding
 \be\begin{aligned}
 \BP& \left(
  \bigcap_{1}^{n-1}
  \{{{\bf x}^{(k)}\in d{\bf x}^{(k)} ,~{\bf y}^{(k)}\in d{\bf y}^{(k)}  }  \}\Bigr|~{{\bf x}^{(n)}={\bf x}^{ } ,~{\bf y}^{(n)}=  {\bf y}^{ }  }
  \right)= \frac {    d\mu_ {{\bf x} ,{\bf y}} }{  \mbox{  Vol} ({\mathcal C}^{(n)}_{{\bf x}{\bf y} })  }
  . \end{aligned}
 \label{14'}\ee
This establishes Theorem \ref{Th1.2} and formula (\ref{joint}) of Theorem \ref{Th1.4} for the coupled GUE-model, except for formula (\ref{ineq}). The analogue for the double Aztec diamonds will be shown in section 6.  \qed



 \newpage

\vspace*{-6cm}

\hspace*{-4.6cm}\includegraphics[height=8.5in]{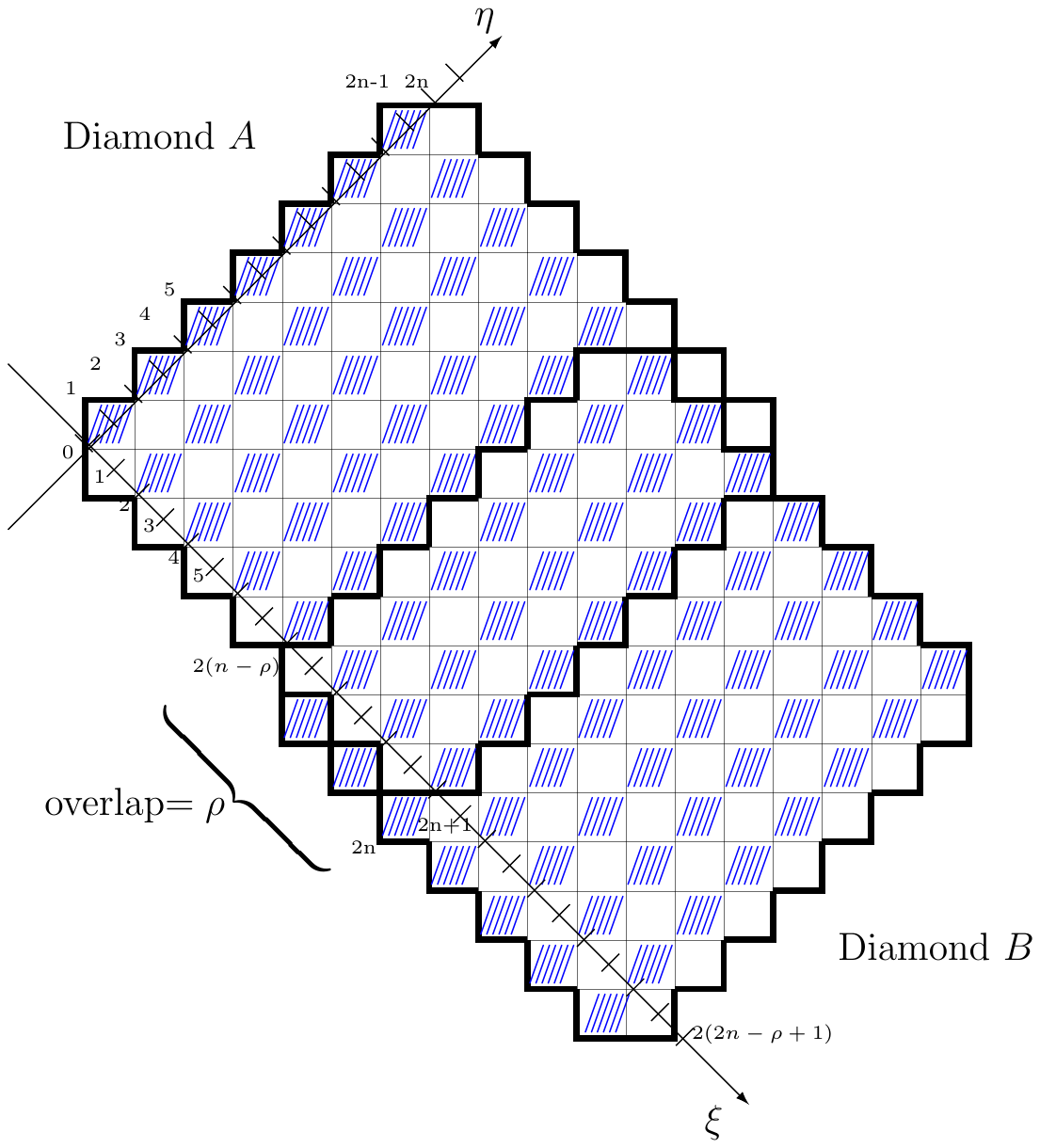}

\vspace*{-26cm}

\hspace*{1.9cm}\includegraphics[height=8.5in]{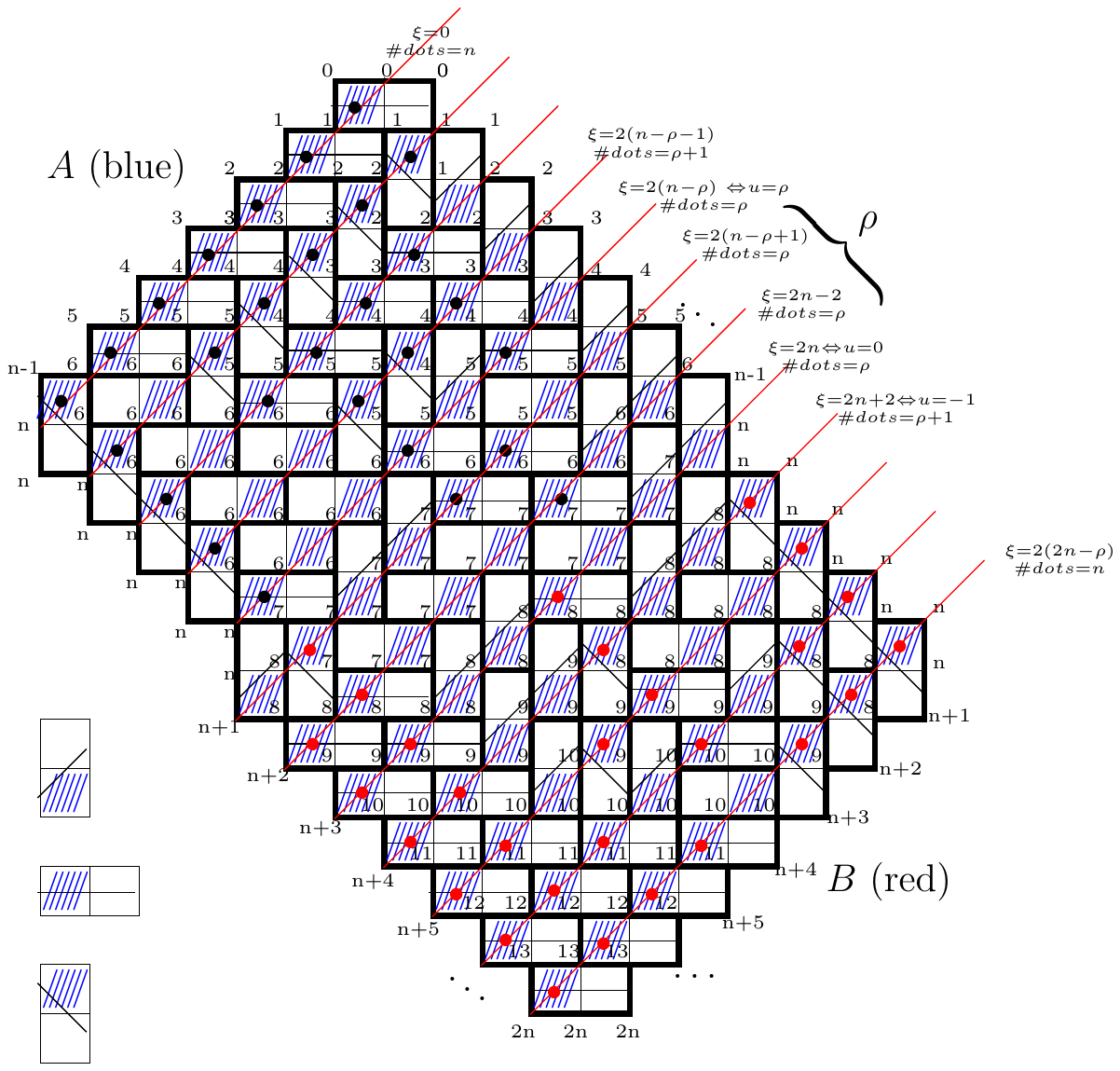}

\vspace*{-7cm}

\hspace*{-3cm}\includegraphics[height=10in]{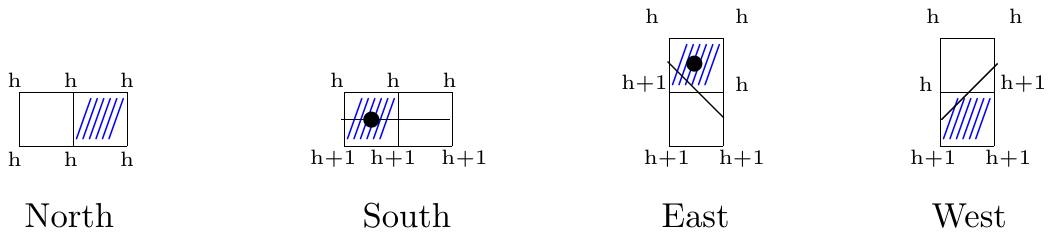}

\vspace*{-19.5cm}

\noindent {\footnotesize Figure 4. The two pictures represent two overlapping Aztec diamonds of size $n=7$ and overlap $\rho=3$, together  with the $(\xi, \eta)$ coordinates. The overlap contains $\rho$ lines (through black squares) $\xi=2s$ for $n-\rho<s\leq n$. To the four domino configurations, we associate four different height functions indicated in the lower part of the figure. The right picture gives a domino tiling, the corresponding heights and its level curves, corresponding to the heights $(\tfrac 12,\ldots,n-\tfrac 12,n+\tfrac 12,\ldots, 2n-\tfrac 12)$. Each time a line $\xi=2i$ intersects level lines of height $(\tfrac 12,\ldots,n-\tfrac 12)$ put a blue ${\mathbb L}$-dot and for the level lines $(n+\tfrac 12,\ldots, 2n-\tfrac 12)$ a red ${\mathbb L}$-dot in the center of the square where it occurs. For individual domino's this corresponds to a dot in two of the four configurations of domino's.}

\section{Double Aztec diamonds, the ${\mathbb L}$- and the Tacnode-Process}

 Consider two overlapping Aztec diamonds $A$ and $B$, of equal sizes $n$ and overlap $\rho$, with opposite orientations; i.e., the upper-left square for diamond $A$ is black and is white for diamond $B$. The size $n$ is the number of squares on the upper-left side and the amount of overlap $\rho$ counts the number of lines of black squares common to both diamonds $A$ and $B$.  Let $\xi,\eta$ be the system of coordinates on the diamonds introduced in \cite{AJvM} and \cite{ACJvM}; see Fig.~4. The even lines $\xi=2k$ for $0\leq k\leq 2n-\rho$  and the odd lines $\eta=2k-1$ for $1\leq k\leq n $ run through black squares. The $\rho$ even lines $\xi=2(n-\rho+1), \ldots, 2n$ belong to the overlap of the two diamonds. The map  between the $\xi$-coordinate here 
%
 %
 %
and the $u$-coordinate introduced just before (\ref{spectra3}) is given by
 \be u=n-\xi/2. \label{map}\ee  Cover this double diamond randomly by dominos, horizontal and vertical ones, as on the right side of  Fig.~4. The position of a domino on the Aztec diamond corresponds to four different patterns, given in Fig.~4 below: North, South,
 East, West.

 
 
 
  
 

 
 Together with this arbitrary domino-tiling of the double Aztec diamond $A\cup B$, one defines a piecewise-linear random surface, by means of a height function $h$ specified by the heights, prescribed on the single dominos according to Fig.~4 below; this height can be taken to be piecewise-linear on each domino. This height function is different from the usual one by Cohn, Kenyon, Propp \cite{CKP}, but related to it by an affine transformation. Let the upper-most edge of the double diamond $A\cup B$ have height $h=0$. Then, regardless of the covering by dominos, the height function along the boundary of the double diamond will always 
be as indicated as on the left-side of Fig.~4, with height $h=2n$ along the lower-most edge of the double diamond. Away from the boundary the height function will depend on the tiling.

 { \bf  \em   The ${\mathbb L}
   $-process} is specified by putting a dot in the middle of the black square when the line $\xi=2s$ in $(\xi,\eta)$-coordinates for $0\leq s\leq 2n-\rho$ intersects a level curve. We call these dots {\em $\mathbb{L}$-particles}. See the right-side of Figure~4 for an example.  More precisely we can put a blue dot when intersecting $A$-level curves and a red dot when intersecting $B$-level curves to distinguish the dots coming from the two Aztec diamonds; see Fig.~4. In other terms, put a dot in the black square each time the random surface goes down one unit along the line $\xi=2s$, in going from left to right. 
We put a probability on the domino tilings, specified by assigning the weight $a>0$ on vertical dominoes and the weight $1$ on horizontal dominoes. This enables us to compute, in principle, the probability of each domino-tiling.
%
%

 Here, we are concerned with the probabilities of the following kinds of events, where $[k,\ell ] $ is an interval of odd integers along the $\eta$-axis 
 \be\begin{aligned}
& \left\{\mbox{The line}~  \{ \xi=2s \}  ~\mbox{has an $\eta$-gap} \supset [k,\ell]
 \right\}
 \\& :=\left\{\mbox{Interval $[k,\ell] \subset \{ \xi\!=\!2s \}$ in $\eta$-coordinates contains no dot-particles}\right \}
\\&=\left\{\mbox{The random surface is flat along the $\eta$-interval   $[k,\ell]
 \subset \{ \xi=2s \}$}\right\}
 \\&=\left\{\mbox{Dominos covering $[k,\ell] \subset \{ \xi\!=\!2s \}$ are pointing to the left of} \right. \\
& \left. \, \mbox{~~~~or above the line $\{ \xi\!=\!2s \}$}\right\}
 \\&=\left\{\mbox{Dominos covering $[k,\ell] \subset \{ \xi\!=\!2s \}$ are West or North in upper Fig.~4}\right\}
\end{aligned}
 \label{events}\ee

  In \cite{ACJvM}, Proposition 1.3, it has been shown that, with regard to the $\eta$-coordinate, 
 the lines $\xi=2i$ contain blue and red ${\mathbb L}$-particles, in varying numbers, satisfying the interlacing patterns of Fig. 5. In the table below $h_{\mbox{\tiny left}}$ and $h_{\mbox{\tiny right}}$ refer to the height of the most-left and most-right point along $\xi=2i$, with the difference of height determining the number of dots on that line; it will be c:
  \be\begin{array}{llllllllll}
 \mbox{lines $\xi=2i$}&\vline   &\mbox{$h_{\mbox{\tiny left}}$}  &\vline&\mbox{$h_{\mbox{\tiny right}}$} &\vline &\Dt h  & \vline& \mbox{\# of blue and red dots} \\
 \hline
 0\leq  i< n-\rho &\vline  &n  & \vline&  i & \vline&n-i&\vline& \mbox{$n-i$ blue dots  }\\
 n-\rho\leq i\leq n &\vline&\rho+ i-\frac 12 &\vline& i+\frac 12 &\vline&   \rho &\vline& \left\{\begin{aligned} &\mbox{$\rho+i-n $ red dots}\\ 
 &\mbox{to the left of}\\&\mbox{$\ n-i$ blue dots}   \end{aligned}\right.  \\
  n < i\leq  2n-\rho &\vline&\rho+i&\vline&n&\vline&   \rho\!+\!i\!-\!n &\vline& \mbox{$\rho+i-n$ red dots  }\\
 \end{array}
 \label{table}\ee
 The blue dots interlace among themselves, as well as the red dots, and this with regard to the $\eta$-coordinate; in the overlap of the two diamonds, 
  we have that the  right-most red dot on the line $\xi=2i+2$ for $n-\rho\leq i\leq n$ is to the left of the left-most blue dot on the line $\xi=2i$. This thus implies an interlacing of the red and blue dots, taken together, in the overlap region and this level by level; the interlacing is always meant in terms of the $\eta$-coordinate. This interlacing is precisely as in (\ref{xy-interlace}) and Fig. 3, with $u $ as in (\ref{map}).


 In \cite{ACJvM}, Proposition 1.1, it was shown that the ${\mathbb L} $-particles on the successive lines $\{\xi=2s \}$ for $1\leq s\leq 2n-\rho$ form a determinantal point process with correlation kernel
\be
\begin{aligned}
{\mathbb L}_{n,\rho }
 (\xi_1,\eta_1;\xi_2,\eta_2)  =&(1+a^2){\mathbb L}^{(0)}_{n }(\xi_1,\eta_1;\xi_2,\eta_2)
\\  
&- (1+a^2)\langle ((I-{K}_n)^{-1}_{_{\geq n-\rho+1}} A_{\xi_1,\eta_1}) (k), B_{\xi_2,\eta_2}(k) \rangle_{_{\geq n-\rho+1 }}.
\end{aligned}
\label{Lkernel}
\ee
given by a perturbation of      
  a kernel ${\mathbb L}^{(0)}_{n }$ by an inner-product\footnote{The ${\mathbb L}^{(0)}_{n }$-kernel is closely related to the one-Aztec diamond kernel. In (\ref{Lkernel}), we use the inner-product $\la f(k), g(k)\ra_{\geq \alpha}=\sum_{\alpha}^{\infty} f(k)g(k) 
$ in $\ell^2[\alpha,\infty]$.} of two functions $A$ and $B$ involving the resolvent of yet another kernel $K_n$; remember $a>0$ is the weight of the vertical dominos. The probability of an event like  (\ref{events}) is given by the Fredholm determinant of the ${\mathbb L}_{n,\rho }$-kernel over the gap $[k,\ell]$. Here and beyond we shall use the map $u=n-\xi/2$ as in (\ref{map}). 

 \begin{proposition}\label{Prop2.1}
 Consider any configuration of $n$ dots ${\bf z}^{(n)}$ and ${\bf z}^{(-\dt)}$ at levels $u=n$ and $u=-\dt$ on the double Aztec diamond $A\cup B$, or as given by Figure 5. Then for $a=1$, all configuration of dots satisfying the interlacing (\ref{xy-interlace})  are equally likely.
\end{proposition}

\begin{proof}    
 Consider any configuration of dots on the double Aztec diamond satisfying the interlacing, with horizontal and vertical dominos being equally likely. Then consider the line  $\xi=2i-2$ with the $k$th and $k+1$st dot, counted from the right boundary, denoted $B$ and $A$ of respective heights $ h=i+k-3/2$ and $h=i+k-1/2$.

\newpage

 \vspace*{-7cm}    
    \hspace*{-4.3cm}    
 { \includegraphics[width=200mm,height=300mm]{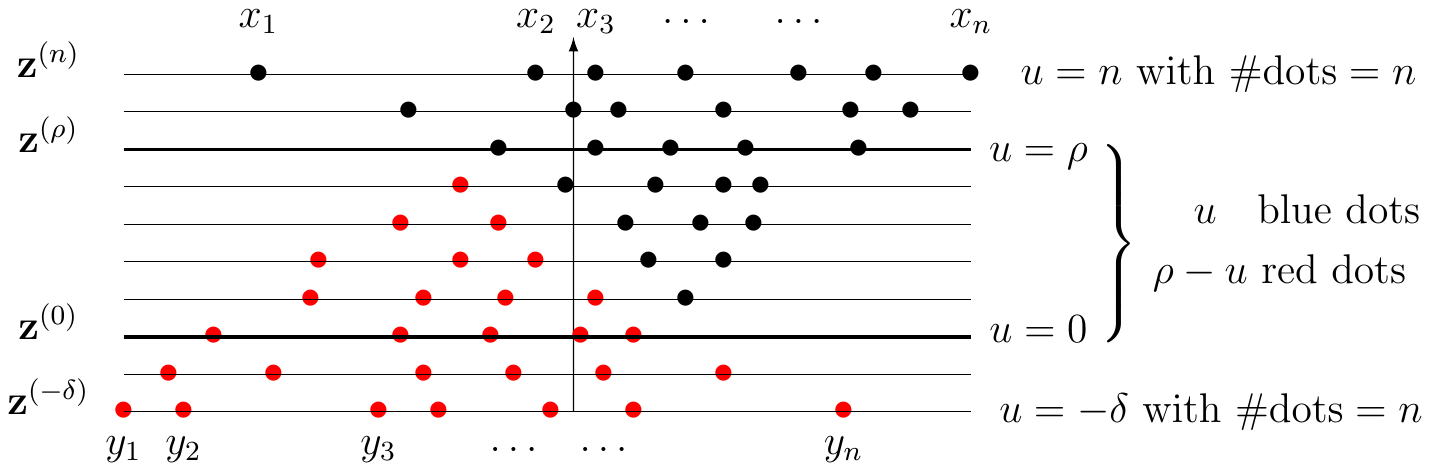}}

 \vspace*{-19cm}    
{\footnotesize Figure 5. The interlacing set of red and blue dots $(\in \BZ)$, subjected to the constraint (\ref{constraint}), with $u=n-\xi/2 $ in terms of the coordinate $\xi$ of Figure 4. Here $\dt$ is such that $n=\rho+\dt$.}

\vspace*{.5cm}

%
\noindent Then by (\ref{table}) the $k$th dot on the line $\xi=2i$, denoted $C$, must have height $h=i+k-1/2$  and must interlace $A$ and $B$; that is $\eta(A)\leq\eta(C)\leq \eta(B)$; see the arguments in the proof of Lemma 4.2 in \cite{ACJvM}. We show there are exactly two ways of connecting the dots $A$ and $C$ (of same height $i+k-1/2$). Note that $B$ and $C$ can never be connected, because they have different heights. 
We distinguish three cases:
 
  \medbreak
   
   (i) $\eta(A)<\eta(C)< \eta(B)$.    
   Having two consecutive black squares on a line $\eta=2j+1$, the upper one {\em without} a dot and the lower one {\em with} a dot, as the top-left side of Figure 7, one checks that these two squares can only be covered by dominos in exactly four different ways, given by Figure 6. This leads to a level line going through the black square with the dot and ending up in the middle of the left side of some black square. From there on, one is facing two consecutive black squares having no dots. That implies a configuration as in Figure 6, until one reaches the vicinity of point $A$. There one has two possibilities as indicated on the top-figures of Figure 7. This shows the two different ways of connecting these two points of same height.

   (ii) $\eta(A)=\eta(C)< \eta(B)$. Also here we have exactly  two ways to connect $C$ and $A$, as given by the lower two figures in Figure 7.

     (iii) $\eta(A)<\eta(C)= \eta(B)$. This case also leads to two different paths, very similar to the ones in case (i).
     
       So, given a fixed configuration of dots at levels $u=-\dt$ and $u=n$,  any configuration of acceptable interlacing dots on the intermediate levels $-\dt<u<n$ leads to the same number of level curves, connecting dots of the same height; this was shown above level by level. But level curves determine the domino tiling in a unique way. Since all domino tilings are equally likely ($a=1$), this implies that all configurations of dots, satisfying the interlacing, are equally likely.   \end{proof}

  \newpage

\vspace*{-5cm}

 \hspace*{-1cm}\includegraphics[width=180mm,height=225mm]{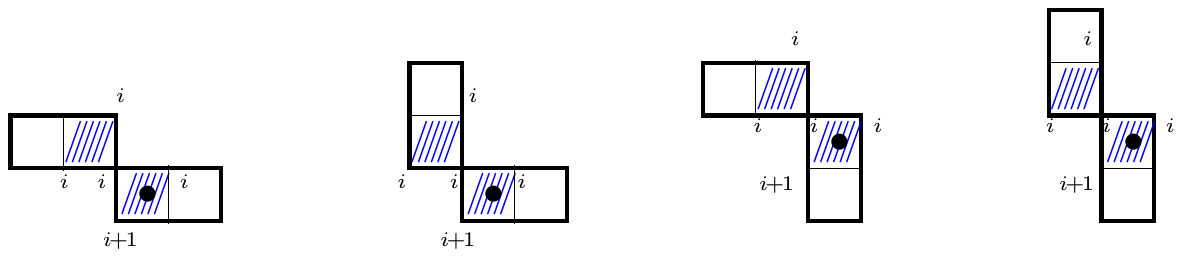}
 
  \vspace*{-17cm}

  {\footnotesize Figure 6. If two consecutive black squares on a line $\eta=2j+1$ are such that the upper one is {\em without} a dot and the lower one is {\em with} a dot, then these are the only possible domino covers.}



\vspace*{-3cm}

 \hspace*{-1cm}\includegraphics[width=170mm,height=215mm]{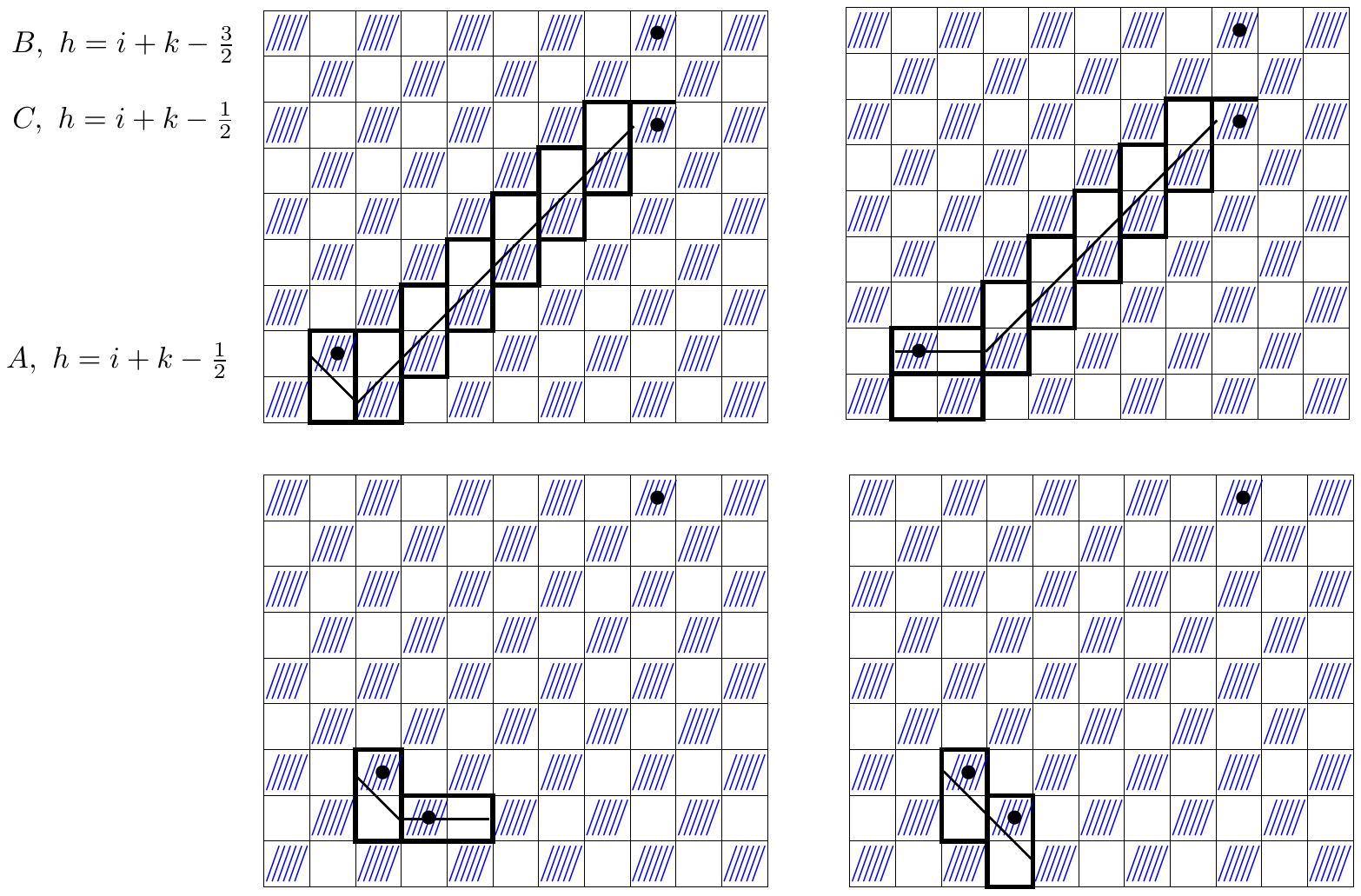}

 \vspace*{-10cm}
 
   {\footnotesize Figure 7. Level curves connecting dots $A$ and $C$ of same height.}
 
    \vspace*{.5cm}
  

   %

      \noindent We now briefly state a result, the first part of which was established in \cite{ACJvM}:
\begin{theorem} \label{a=1}Let the  size $n=2t+\epsilon$, $\epsilon\in\{0,1\}$, 
of the diamonds go to infinity, while keeping the overlap $\rho $ finite and, and letting 
the weight of the vertical domino's $a\to 1$ as  
$$
a= 1-\frac {  \beta}{\sqrt{n/2}} , ~~\mbox{with $\beta\geq 0$ fixed.}
$$
The 
 coordinates $(\xi,\eta)$ are scaled as follows,
\be \begin{aligned}
\xi_i = 4t+2\epsilon-2 u_i &,& \eta_i =2t  +   2[y_i \sqrt{t }]-1 ,~~\mbox{with}~ u_i \in \mathbb{Z} \ , y_i \in \BR.
\end{aligned}
\label{scK}\ee
With this scaling, the limit of the $\mathbb{L}$-kernel (\ref{Lkernel}) is given by the $\BK_{\beta, \rho}^  {\mbox{\tiny TAC}}$-kernel (\ref{dGUE}) :
\be
\begin{aligned}
 \lim_{n \to\infty}  ~(-a)^{(\eta_1-\eta_2)/2}(-\sqrt{t})^{(\xi_1-\xi_2)/2} &
   {\mathbb L} _{n,\rho}
    (\xi_1,\eta_1;\xi_2,\eta_2) \sqrt{t}
 =\BK_{\beta, \rho}^  {\mbox{\tiny TAC}} (u_1,y_1;u_2,y_2)
 ,\end{aligned}
 \label{limL}\ee
 %
 with $\BK_{\beta, \rho}^  {\mbox{\tiny TAC}}$ defined in (\ref{dGUE}); the determinantal process defined by this kernel induces a probability distribution $\BP^{\mbox{\tiny\rm TAC}}$.
In particular, the limit statement holds for $\beta=0$; i.e., for $a=1$. 
Moreover in this scaling limit, the dot-particles $\in \BR$, belonging to the discrete levels $u\in \BZ$, interlace exactly as in Figure 3 or as indicated in Table (\ref{table}); they are uniformly distributed, given ${\bf x}={\bf x}^{(n)}$ and ${\bf y}={\bf y}^{(n)}$; i.e., the probability $\BP^{\mbox{\tiny\rm TAC}}$ induced by ${\mathbb T}^{\mbox{\tiny\rm TAC}}$ equals (see definition (\ref{unifmeas}) for $d\mu_ {{\bf x} {\bf y}}$)
 \be\begin{aligned}
 \BP^{\mbox{\tiny\rm TAC}}& \left(
  \bigcap_{1}^{n-1}
  \{{{\bf x}^{(k)}\!\! \in d{\bf x}^{(k)} ,~{\bf y}^{(k)}\!\!\in d{\bf y}^{(k)}  }  \}\Bigr|
 \begin{array}{c} \mbox{given ${\bf x}={\bf x}^{(n)}$ and ${\bf y}={\bf y}^{(n)}$}
   %
\end{array}\!\!\! \right)
     = \frac{d\mu_ {{\bf x} {\bf y}}}{ \mbox{Vol}~({\mathcal C} _{{\bf x} {\bf y}})} 
  . \end{aligned}
\label{uniform} \ee
\end{theorem} 
   
\begin{proof} We refer the reader to \cite{ACJvM} for the scaling limit (\ref{limL}). Concerning the last part, Proposition \ref{Prop2.1} states that all configurations are equally likely, given a fixed set of dots at levels $u=-\dt$ and $u=n$, so that in the limit the dot-particles are uniformly distributed, establishing statement (\ref{uniform}). \end{proof}


The tacnode kernel $\BK_{\beta, \rho}^  {\mbox{\tiny TAC}} $, as in (\ref{dGUE}), depends on two parameters $\beta,\rho$, 
and is a perturbation of the GUE-minor kernel $\BK^ {\mbox{\tiny minor}}  (n,x;n',x')  $,  defined for $n,n'\in \BZ$, rather than $\mathbb N$  (see (\ref{8.6})), with ingredients (see footnote 1 for the circle $\Gamma_0$ and the imaginary line $L$)
  \be\begin{aligned}
{\cal K}^\beta (\lambda,\kappa)&:=\oint_{\Gamma_0}  \frac{d\zeta} {(2\pi \I)^2 }\int_{ L
 } \frac{d\om}{ \om-\zeta   }
   \frac{e^{-2\zeta^2+4 \beta \zeta }} { e^{-2\omega^2+4 \beta \omega }}
\frac{\zeta^{\kappa } }{\omega^{\lambda +1}} 
\\
   {\cal A}^{\beta,y }_{v }(\kappa)&:=  
 ~
  \oint_{\Gamma_0}  \frac{d\zeta}{ (2\pi \I)^2} 
   \!\! \int_{L
    } \frac{d\om}{\zeta\!-\! \om  }
  \frac{e^{-\zeta^2-2  y  \zeta}} 
  {e^{-2\om^2+4 \beta \om}}  
  \frac{\zeta^{-v }}{\om^{\kappa +1}} 
%
+
  \int_{L
   }\frac{d\zeta}{2\pi \I}\frac{e^{\zeta^2-2\zeta (y +2\beta) }}
  {\zeta^{v+\kappa +1 }}
 %
\\
{\cal B}^ {\beta,y }_{u}(\lambda)&:= 
  \oint_{\Gamma_0}  \frac{ d\zeta}{ (2\pi \I)^2}   \int_{ L
   } \frac{d\om}{ \zeta\!-\!\om   }
  ~~\frac{e^{-2\zeta^2+4\zeta \beta }} {e^{- \om^2-2\om  y    }} 
  \frac{\zeta^{\lambda  }}{\om^{-u  }}
%
+\oint_{\Gamma_0}\frac{d\omega}{2\pi \I} ~~
 \frac{\omega^{u +\lambda  }}
  {e^{ \omega^2-2 \omega (y +2\beta)}}.\end{aligned}
  \label{defAB}\ee

One easily shows that the perturbation term in (\ref{Lkernel}) is actually a finite sum given by the equality  ${\tiny \stackrel{(*)}{=}}$ in (\ref{2min}). The second form  ${\tiny \stackrel{(**)}{=}}$ 
 is due to the natural involution of the double Aztec diamond, which consists of flipping the double diamond about a $\xi$-axis through the middle and once again flipping about an $\eta$-axis through the middle; this involution is preserved in the limit. So we have:
\be
\begin{aligned}
&\BK_{\beta, \rho}^  {\mbox{\tiny TAC}} (u_1,y_1;u_2,y_2)
 \\
 &\stackrel{(*)}{=}  \BK^ {\mbox{\tiny minor}}  (u_1,\beta-y_1 ;~u_2 ,\beta-y_2 ) \\
&~~~+2\sum_{\lambda=0}^{\max(\rho-1,\rho-1-u_2)}  \bigl((\Id - {\cal K}^\beta ( \lambda-\rho ,\kappa-\rho ))^{-1}  {\cal A}^{\beta,y_1-\beta}_{u_1 }\bigr)(\lambda-\rho)  {\cal B}^{\beta,y_2-\beta}_{u_2 }(\lambda-\rho) 
  _{_{ }}
 \\
 &\stackrel{(**)}{=}  \BK^ {\mbox{\tiny minor}}  (\rho-u_2,\beta+y_2 ;~\rho-u_1 ,\beta+y_1 ) \\
&~~~+2\sum_{\lambda=0}^{\max(\rho-1,u_1-1)}  \bigl((\Id - {\cal K}^\beta ( \lambda-\rho ,\kappa-\rho ))^{-1}  {\cal A}^{\beta,-y_2-\beta}_{\rho-u_2 }\bigr)(\lambda-\rho) 
 {\cal B}^{\beta,-y_1-\beta}_{\rho-u_1 }(\lambda-\rho) 
  _{_{ }}
  . \label{2min}\end{aligned}\ee


\section{Notations, addition formulas and matrix identities }

\subsection{Basic notations} Define the standard Hermite polynomials $H_n$, related polynomials $\tilde H_n$ and $P_n$,  
 \footnote{See e.g. \cite{JoNo}, formula (6.21) for the first identity in (\ref{Phi-Psi}), while the second identity is by integration by parts, with contours $L$ and $\Gamma_0$ as in footnote 1.}:
\be
\begin{aligned}
H_n(x)&:=\frac{n!}{2\pi \I}\oint_{\Gamma_0}e^{-z^2+2xz}\frac {dz}{z^{n+1}}=\frac{2^ne^{x^2}}{\I\sqrt{\pi}}\int_{L}e^{w^2-2xw}w^ndw
=(2x)^n+\ldots\\
\widetilde H_n(x)&:=\frac 1{n!} H_n(x)
\\  
%
P_n(x)&:=  \begin{aligned}&\frac{1}{n!~\I^n}H_n(\I x)=\frac 1{\I^n}\widetilde H_n(\I x)~~~\mbox{for $n\geq 0$} 
\end{aligned} 
\end{aligned}
\label{herm1}\ee
with $H_0=\widetilde H_0=P_0=1$ and
\be
H_n(x)=\tilde H_n(x)=P_n(x)=0
~~~~\mbox{for $n< 0$}.
\label{herm2}\ee
Also define another set of polynomials by a recurrence, setting $Q_0=1$ and $Q_{-1}=0$,
\be
Q_n(\eta)=\left\{\begin{aligned}&\frac 12 (n+1)Q_n(\eta)=\eta Q_{n-1}(\eta)+Q_{n-2}(\eta)~~~\mbox{for $n\geq 0$}\\
&Q_{-k}(\eta)=\frac{1}{2^{k-1}}H_{k-2}(-\eta)~~~\mbox{for $n=-k\leq -2$}.
\end{aligned}\right.\label{herm3}\ee 
and for all $n\in \BZ$, define the functions $\Phi_n$ and $\Psi_n$, which by integration by parts can be written as a linear combination of an error function and a Gaussian:
\be
\begin{aligned}
 \Phi_n (\eta) &:=\frac{1}{2\pi \I} \int_L \frac{e^{v^2+2\eta v}}{v^{n+1}} dv  
 = \left\{\begin{aligned}
  &  \frac{2^n}{\sqrt{\pi} n!} \int^\infty_0 \xi^n e^{-(\xi-\eta)^2} d\xi \quad , n\quad \geq 0
\\ &
   \frac{ e^{-\eta^2}}{\sqrt{\pi}} \frac{H_{-n-1}(-\eta)}{2^{-n} } \quad ,\quad n \leq -1
 \end{aligned}\right. \\ 
  &= F(\eta)P_n(\eta)+F'(\eta)Q_{n-1}(\eta) 
 \\ \\
  \Psi_n(\eta)&:=G(\eta)P_n(\eta)+G'(\eta)Q_{n-1}(\eta),
  \end{aligned}
 \label{Phi-Psi}\ee
  with
 $$
 \begin{aligned}
 G(y) &= \frac{1}{\sqrt{\pi}} \int^\infty_y e^{-\xi^2} d\xi \quad , \quad F(y) = \frac{1}{\sqrt{\pi}} \int^y_{-\infty} e^{-\xi^2} d\xi = 1-G(y).
\end{aligned}
$$
%
Note that
\be
\begin{aligned}
\Phi_n(y)&= G(-y)P_n(y)-G'(-y)Q_{n-1}(y)\\
&=(-1)^n\left(G(-y)P_n(-y)+G'(-y)Q_{n-1}(-y)\right) 
 =(-1)^n \Psi_n(-y),
\end{aligned}
\label{PhiPsi}\ee
and so for $n\leq -1$, we have $P_n(y)=0$ and thus
\be
\Phi_n(y)=-\Psi_n(y)\mbox{   for   }n\leq -1.
\label{PhiPsi-}\ee
Also define functions $g_{y } (\kappa)$ and $h_{y } {(\lb)}$, which have the following expressions in terms of $\Phi_\kappa$ and $\widetilde H_{\lb}$,
\be
\begin{aligned}
  g_{y } (\kappa) &:= \int_{L} \frac{d\om}{2\pi i}~ e^{\om^2-2(\beta-y )\om} \om^{-\kappa-1}=\Phi_\kappa(y- \beta)
  \\
   h_{y } {(\lb)}&:=  \int_{\Gamma_0} \frac{d\zeta}{2\pi i } ~e^{- \zeta^2+2(\beta-y )\zeta} \zeta^{\lb}= \widetilde H_{-\lb-1}(\beta-y)
\end{aligned}\label{gh}\ee

\noindent
Given $\beta\geq 0$, the following quantities will be used throughout
\be
c_k :=  {2}^{k/2}  \tilde H_k(\beta \sqrt{2})\mbox{~~~and~~~}\tilde c_k:= {2}^{k/2}  \Phi_k(-\beta \sqrt{2}).
\label{ck}\ee
Define the square matrix of size $n $:
\be
 \widetilde {\cal H}^{( -x)}_{n}  :=
 \left( \widetilde H_{n-j}( -x_i)\right)_{1\leq i,j\leq n} 
 \mbox{with}~~~\det \widetilde {\cal H}^{( -x)}_{n}  =c'_n \Dt_n(x)\label{HermMatrix}\ee
with $$c'_n=\frac{  (-2)^{n(n-1)/2}}{(n-1)!\ldots 2! 1!}$$

\bigbreak

\subsection{ Addition formulas} The following will be useful later:
\begin{lemma}\label{L2.1}
Given $b,c \in \mathbb{C}$
 \be
\begin{aligned}
\sum_{k+\ell=n} b^k c^\ell \widetilde H_k (u) \widetilde H_\ell (v) = 
\left\{
\begin{array}{llll}
\widetilde H_n (bu+cv) \ \textrm{for} \ b^2+c^2=1
\\
 P_n (bu + cv) \ \textrm{for} \ b^2+c^2=-1
\\
\frac{2^n}{n!}(bu+cv)^n\ \textrm{for} \ b^2+c^2=0\end{array}
\right.
\end{aligned}
\label{L2.1F}\ee
\end{lemma}

\begin{proof}
The first identity (for $b^2+c^2=1$) depends on the generating function for Hermite polynomials
\[
e^{-z^2+2z x} = \sum^\infty_{n=0} \frac{z^n}{n!} H_n (x) = \sum^\infty_{n=0} z^n \widetilde H_n (x),
\]
and on expanding in $z$ the following identity
\[
z^{-z^2(b^2+c^2)+2z(bu+cv)}    =   e^{-(bz)^2+2bzu}              e^{-(cz)^2+2czv}.
\]
The identity for $b^2+c^2=-1$ is obtained from the first identity by 
replacing $b\to \I b$, $c\to \I c$. The third identity is obtained from the above exponential identity by setting $b^2+c^2=0$.
\end{proof}  

\bigbreak

\begin{corollary}\label{C2.2}
The inverse of the lower-triangular matrix of size $\al$,
 \[
 C_\al: =
 \left(\begin{array}{cccc}
 c_0 &  & O &\\
 c_1 & c_0 &  &\\
 \vdots & \vdots & \ddots &\\
 c_{\al-1} & c_{\al-2} & \ldots  & c_0
 \end{array}\right)
\mbox{with $c_k=2^{k/2}\widetilde H_k(\beta \sqrt{2})$} \]
 is given by
\be
 C_\al ^{-1}=
 \left(\begin{array}{cccc}
 a_0 &  & O &\\
 a_1 & a_0 &  &\\
 \vdots & \vdots & \ddots &\\
 a_{\al-1} & a_{\al-2} & \ldots  & a_0
 \end{array}\right)
 \mbox{with $a_k=2^{k/2}\I^k \widetilde H_k(\I \beta \sqrt{2})$}.\label{Cinv}\ee
 Moreover, for $n\geq 0$,
\be \begin{aligned}
\sum_ {{p+q=n}\atop {p,q\geq 0}} c_{p} P_q (y-\beta) &= \widetilde H_n (y+\beta).
\\
\sum_ {{p+q=n}\atop {p,q\geq 0}} a_{p} \widetilde H_q(-y+ \beta) & =P_n (-y-\beta)
\\
\sum_{{k+\ell=n}\atop {k,\ell\geq 0}} \widetilde H_k(-u) P_\ell(v)&=
\frac {2^n}{n!}(v-u)^n,\end{aligned}\label{Hermid}
\ee
implying the following matrix identity for $x,y\in \BR^n$, (see (\ref{HermMatrix}))
 \be  \begin{aligned}
   \widetilde {\cal H}^{(-x)}_{n}\times  \left(P_{i-1}( y_j)\right) _{1\leq i,j\leq n}
 & =\left(\frac {(2(y_j-x_i))^{n-1}}{(n-1)!} \right)_{1\leq i,j\leq n}\label{HP}.\end{aligned}\ee
  
\end{corollary}

\begin{proof} Note (\ref{Cinv}), or what is the same $\sum_{\ell+k=n} a_kc_\ell=\dt_{n,0}$, follows from (\ref{L2.1F})
 with $  \I c=  b=\I v=u=\I \sqrt{2}\beta $.
Setting $b=-i$, $c=\sqrt{2}$ and $u=\I(y-\beta) ,~~v=\beta$ in (\ref{L2.1F}) establishes the first identity (\ref{Hermid}) and setting $b= \I\sqrt{2}$, $c=1$ and $u=\I\beta \sqrt{2},~~v=-y+\beta$ in (\ref{L2.1F}) the second identity.  Setting $b= 1,~c=-i$ and $u\mapsto -u,~~v\to  \I v$ in (\ref{L2.1F}) proves the last identity.
\end{proof}

Define the matrix $C^{\circlearrowleft}_{\al}$ as the matrix $C_\al$, but written upside down,
\be
C^{\circlearrowleft}_{\al}=\left(\begin{array}{ccccccc}
c_{\al-1}&c_{\al-2}&\ldots&c_1&c_0\\
c_{\al-2}&c_{\al-3}&\ldots&c_0& \\
\vdots &\vdots & \\ \\ &&&O\\
c_0&
\end{array}\right)~~\mbox{and}~~
C^{\circlearrowleft~-1}_{\al}=\left(\begin{array}{ccccccc}
 & & & &a_0\\
 & & &a_0&a_1 \\
    &  O & &&
  \\  \\ &&&&\vdots \\
a_0&a_1&\ldots&&a_{\al-1}
\end{array}\right)
\label{Cupside}\ee
Then we have,

\begin{corollary}\label{C2.3}

$$ C^{\circlearrowleft~-1}_{\al} 
(\widetilde H_{\al-1}( y+ \beta),\ldots,\widetilde H_0( y+ \beta))^\top
=(P_0(y-\beta),\ldots, P_{\al-1}(y-\beta))^{\top}$$
\end{corollary}

\subsection{Determinants and inner products}
\begin{lemma}\label{LA1}
Given $N$ row vectors $w_\al=(w_{\al,1},\ldots,w_{\al,M}) \in \BR^M$ and $N$ column vectors $v_\al=(v_{1,\al},\ldots,v_{M,\al})^\top \in \BR^M$, with $M\geq N$, the following holds:
$$
\begin{aligned}
\det\left(\la w_\al ,~v_\beta  \ra\right)_{1\leq \al,\beta \leq N}
=\det (WV)=\left\la \Om_w,\Om_v \right\ra \end{aligned}
$$
where
$$
\Om_w:= \sum_1^M w_{1i} e_i\wedge\ldots \wedge \sum_1^M w_{Ni} e_i
,~~
\Om_v:= \sum_1^M v_{i1} e_i\wedge\ldots \wedge \sum_1^M v_{iN} e_i$$
and $W$ and $V$ are two matrices of sizes $N\times M$ and $M\times N$,
$$
W:=\left(\begin{array}{l}w_{1}\\ \vdots  \\ w_{N}\end{array}\right) \mbox{   and   }
V:=\left( v_1 \ldots   v_N   \right).
$$
\end{lemma}

\begin{proof} The first identity follows from the definition of matrix multiplication, and the second equality from the Cauchy-Binet
 identity.
\end{proof}
 \begin{corollary}\label{CA1'} If $M=N$, then 
 $$
\begin{aligned}
\det\left(\la w_\al ,~v_\beta  \ra\right)_{1\leq \al,\beta \leq N}
=\left\la \Om_w,\Om_v \right\ra =\det (W)\det(V)\end{aligned}
$$
 
  \end{corollary}

 \begin{corollary}\label{CA1''}Given $k+m$ vectors $e_1,\ldots,e_k,e_{k+1},\ldots,e_{k+m}$ and two matrices of sizes $(\ell,m)$ and $(k+m,k+\ell)$ with $m\geq \ell$, 
 \[
 B = (b_{ij})_{1\leq i \leq \ell \atop 1 \leq j \leq m}   \qquad \textrm{and} \qquad 
 V = (v_{ij})_{1\leq i \leq k+m \atop 1 \leq j \leq k+\ell}    
 .\]
 Then for
 \[
 \Om_V = \sum^{k+m}_{i=1} v_{i1} e_i \wedge \ldots \wedge      \sum^{k+m}_{i=1} v_{i,k+\ell} e_i ,
 \]
 we have
 \[
 \langle e_1    \wedge \ldots \wedge       e_k  \wedge \sum^m_{i=1} b_{1i} e_{k+i} \wedge \ldots \wedge     \sum^m_{i=1} b_{\ell,i}   e_{k+i} \ , \ \Om_V \rangle
=  \det     \left(\begin{array}{c} 
 (v_{ij})_{1\leq i \leq k \atop 1 \leq j \leq k+\ell }
 \\[0.5cm]
 \hline 
 B ~(v_{k+i,j})_{1\leq i \leq m \atop 1 \leq j \leq k+\ell}           \end{array}\right).
 \]

 \end{corollary}

 \begin{corollary}\label{CA1'''} Letting $w_{\alpha}$ and $v_{\beta}\in {\mathbb R}^{2N}$, we set for $1\leq k\leq N$,
 $$\begin{aligned}
  w_k&=(a_k,{\bf 0})\\
  w_{N+k}&=({\bf 0},b_k)
 \end{aligned}
 ~~~~~~~\mbox{with}~~\left\{   \begin{array}{l}
 a_k=(a_{k1},\ldots,a_{kN})\\
 {\bf 0}=(0,\ldots,0) \\
 b_k=(b_{k1},\ldots,b_{kN})
 \end{array}\right\}\in \BR^N
 $$
 $$
 v_k=\left(\begin{array}{c}c_k\\d_k
 \end{array}\right),~~~v_{N+k}=\left(\begin{array}{c}e_k\\f_k
 \end{array}\right)
 \mbox{with}~~\left\{   \begin{array}{c}
 c_k=(c_{1k },\ldots,c_{N,k })^{\top},~~~e_k=(e_{1k },\ldots,e_{N,k })^{\top}\\
 d_k=(d_{1k},\ldots,d_{ N,k})^{\top},~~~f_k=(f_{1k},\ldots,f_{ N,k})^{\top}.
 \end{array}\right.
 $$
 Then
 $$
\begin{aligned}
\det\left(\la w_\al ,~v_\beta  \ra\right)_{1\leq \al,\beta \leq 2N}
=&\det \left(\begin{array}{cc} 
\left(\left\la a_k,c_\ell\right\ra\right)_{1\leq k,\ell\leq N}&\left(\left\la a_k,e_\ell\right\ra\right)_{1\leq k,\ell\leq N}\\
\left(\left\la b_k,d_\ell\right\ra\right)_{1\leq k,\ell\leq N}&\left(\left\la b_k,f_\ell\right\ra\right)_{1\leq k,\ell\leq N}
 \end{array}\right)
\\ \\=&\det((a_{ ki})_{1\leq k,i \leq N}) \det((b_{ ki})_{1\leq k,i \leq N}) \\
&~~~~~~~\times \det \left(\begin{array}{cc} 
(c_{ik})_{1\leq i,k\leq N}&(e_{ik})_{1\leq i,k\leq N}\\
(d_{ik})_{1\leq i,k\leq N}&(f_{ik})_{1\leq i,k\leq N}
 \end{array}\right).   \end{aligned}
$$
\end{corollary}


\subsection{Expressions for the functions ${\mathcal A},~{\mathcal B}$ and ${\mathcal K}$}

In this section, we express ${\cal A}$, ${\cal B}$ and ${\cal K}$, of  (\ref{defAB}), in terms of the basic functions (\ref{herm1}),(\ref{herm2}),(\ref{herm3}),(\ref{Phi-Psi}) and (\ref{ck}) of previous sections, as follows:

\begin{proposition}\label{PAB}

For $0\leq \kappa\leq \max (\rho-u-1,\rho-1)$, one has:
$$\begin{aligned}  {\cal A}^{\beta,y-\beta }_{u }(\kappa-\rho)
&= \Phi_{ \kappa-\rho+u}(-y-\beta)
-\Id_{u\geq 1}
\sum_{\al=0}^{u-1}\tilde c_{\kappa-\rho+u-\al}  \tilde H_{ \al}(\beta-y).
 \\
 {\mathcal B}^{\beta,y-\beta}_u  (\kappa-\rho) 
 &=
  \Id_{\kappa\leq \rho-1} \sum^{\rho-\kappa }_{\al=1} c_{\rho-\kappa-\al} \Psi_{\al-u-1} (y-\beta) 
  + \Id_{-u\geq  1}\sum^{ -u-1}_{\al=0} c_{\rho-\kappa-u-1-\al} P_{\al } (y-\beta)   
   \end{aligned}
$$
 In particular, for $u\leq 0$  and $0\leq \kappa \leq \rho-u-1$, one has
$$\begin{aligned}  {\cal A}^{\beta,y-\beta }_{u }(\kappa-\rho)& =  
 \frac{e^{-(y+\beta)^2}}{\sqrt{\pi}}
  \frac{H_{\rho-\kappa-u-1}(y+ \beta)}{2^{\rho-\kappa-u }}
\\
   \end{aligned}
$$
and for $u\leq 0$ and  $\rho\leq \kappa\leq \rho-u-1$, one has
$${\mathcal B}^{\beta,y-\beta}_u  (\kappa-\rho)
 =\tilde H_{\rho-\kappa-u-1}(y+ \beta)
 .$$
Moreover the kernel ${\cal K}^\beta$ has the following form and also the following matrix representation for $\lb, \kappa\geq 0$:
 \be\begin{aligned}
{\cal K}^\beta ( \lambda-\rho, \kappa-\rho)&:=\oint_{\Gamma_0}  \frac{d\zeta} {(2\pi \I)^2 }
\int_{L} \frac{d\om}{ \om-\zeta   }
   \frac{e^{-2\zeta^2+4 \beta \zeta }} { e^{-2\omega^2+4 \beta \omega }}
\frac{\zeta^{ \kappa -\rho} }{\omega^{  \lambda -\rho+1}} 
\\
&=\left\{\begin{aligned}&\sum_{\al=0} ^{\rho-1-\kappa }  c_{\rho-\kappa-\al-1}\tilde c_{ \lambda-\rho+\al+1} 
  \mbox{     for  } 0\leq \kappa\leq \rho-1 
\\   
&0 \mbox{     for  } \kappa\geq \rho
\end{aligned}\right.
\end{aligned}
\label{Kscrip}\ee 
 $$
~~~~~~~~~=
\left(
\begin{array}{cc}
   {\mathcal K}^{(\beta)}  ( \lambda-\rho,\kappa-\rho )\Bigr|_{[0,\rho-1]}   &  O \\ \\
\ast  & O
\end{array}
\right) ,
$$
with Fredholm determinant given explicitly by a $\rho\times \rho$ determinant, as given in (\ref{ineq}). The GUE-minor kernel $\BK^ {\mbox{\tiny minor}}$ has the following expression in terms of the $H_k$'s and the $\Phi_k$'s: (see notation (\ref{gh}))
 \be
 \begin{aligned}
 \frac 12 \BK^ {\mbox{\tiny minor}} & (u_1,y_i;u_2,y_j)
\!\!\! &=\!\! -\Id_{u_1>u_2}   \BH^{ u_1-u_2 } (2(y_j\!-\!y_i)) + \Id_{u_1\geq 1}\sum_{\lb = 0}^{u_1-1}  h_{y_i}(\lb-u_1) g_{y_j}(\lb-u_2) .
 \end{aligned}
 \label{8.6'}\ee


\end{proposition}

\proof 
From the definition (\ref{defAB}), and the definition (\ref{ck}) of $\tilde c_k$, one checks
{  
 $$\begin{aligned}  {\cal A}^{\beta,y-\beta }_{u }(\kappa-\rho)& =  
  \int_{L}\frac{d\zeta}{2\pi \I}\frac{e^{\zeta^2-2\zeta (y + \beta) }}
  {\zeta^{u+\kappa-\rho +1 }}
  \\
  &- \Id_{u\geq 1} \sum_{\al=0}^{u-1}\int_L\frac {d\omega}{2\pi i}\frac{ e^{2\omega^2-4\beta \om}}{\om^{\kappa+\al-\rho+2}}\int_{\Gamma_0}\frac {d\zeta}{2\pi i} 
  \frac{e^{-\zeta^2-2  (y-\beta)  \zeta}}{\zeta^{u-\al}}
 \\
 &= \Phi_{ \kappa-\rho+u}(-y-\beta)
\\&-\Id_{u\geq 1}\sum_{\al=0}^{u-1} {(\sqrt{2})^{\kappa-\rho+\al+1}}{ }\Phi_{\kappa-\rho+\al+1}(-\beta\sqrt{2})  \widetilde H_{u-1-\al}(\beta-y)
 \\
 &= \Phi_{ \kappa-\rho+u}(-y-\beta)
-\Id_{u\geq 1}
\sum_{\al=0}^{u-1}\tilde c_{\kappa-\rho+u-\al}  \widetilde H_{ \al}(\beta-y).
 \\
   \end{aligned}
$$}
%
Concerning the expression ${\mathcal B}^{\beta,y-\beta}_u  $, it suffices to show the following:
   \be\begin{aligned}  
&{\mathcal B}^{\beta,y-\beta}_u  (\lambda-\rho)\\
 &=\tilde H_{\rho-\lambda-u-1}(y+ \beta) ~~~\mbox{for}  ~\left\{~\begin{array}{l}  u\leq 0  \\   \lb\geq \rho \end{array} \right.
\\& =
   \sum^{\rho-\lb }_{\al=1} c_{\rho-\lb-\al} \Psi_{\al-u-1} (y-\beta) 
  +\Id_{u\leq -1} \sum^{ -1}_{\al=u} c_{\rho-\lb-\al-1} P_{\al-u} (y-\beta)  ~~~\mbox{for}  ~\left\{~\begin{array}{l}  u\leq 0  \\   \lb<\rho    \end{array} \right.
  \\
  & =  \sum^{\rho-\lb }_{\al=1} c_{\rho-\lb-\al} \Psi_{\al-u-1} (y-\beta) ~
  ~~\mbox{for}  ~\left\{~\begin{array}{l}  u>0  \\   \lb< \rho   \end{array} \right.
  \\&=0~~\mbox{for $u> 0 $ and $\lb\geq \rho$  }.
\end{aligned}\label{Bexpr}\ee
At first, we express 
${\mathcal B}^{\beta,y-\beta}_u  (\lambda-\rho)$ in terms of the $\Phi_k$'s:%
\be\begin{aligned}  
{\mathcal B}^{\beta,y-\beta}_u &(\lambda-\rho) \\
=& \int_{\Gamma_0} \frac{d\omega}{2\pi i} \frac{e^{-\omega^2+2\omega(y+ \beta)}}{\omega^{\rho-\lambda-u}}
-\Id_{\rho-\lb\geq 1} \sum^{\rho-\lb-1}_{\alpha=0} \int_{\Gamma_0} \frac{d\zeta}{2\pi i} \frac{e^{-2\zeta^2+4\zeta \beta}}{\zeta^{\rho-\lambda-\alpha}}  \int_L \frac{d\omega}{2\pi i}\frac{e^{\omega^2+2 (y-\beta)\omega}}{\omega^{\alpha-u+1}}
\\
=& \widetilde H_{\rho-\lambda-u-1}(y+ \beta)
\\
& -\Id_{\rho-\lb\geq 1} \sum^{\rho-\lambda-1}_{\alpha=0}  {(\sqrt{2})^{\rho-\lambda-\alpha-1}} \widetilde   H_{\rho-\lambda-\alpha-1}(\beta\sqrt{2}) \Phi_{\alpha-u}(y-\beta)
\\
 =&\tilde H_{\rho-\lambda-u-1}(y+ \beta)
  - \Id_{\rho-\lb\geq 1}  \sum^{\rho-\lambda-1}_{\alpha=0}  c_{\rho-\lambda-\alpha-1} \Phi_{\alpha-u}(y-\beta).
 \end{aligned}
 \label{B'}\ee
%
 %
 $\bullet$ For $u=-\dt \leq 0$ and $ \rho\leq\lb $, we have
  \be\begin{aligned}  
{\mathcal B}^{\beta,y-\beta}_u &(\lambda-\rho)=\tilde H_{\rho-\lambda-u-1}(y+ \beta). \label{B''} 
\end{aligned}
\ee
 $\bullet$ For $u=-\dt \leq 0$ and $ \rho>\lb $, we have, using (\ref{Hermid}) and (\ref{herm3}), Corollary \ref{C2.2}, 
  $$ \begin{aligned}  
&\hspace*{-1cm}{\mathcal B}^{\beta,y-\beta}_u  (\lambda-\rho) 
 =\tilde H_{\rho-\lambda-u-1}(y+ \beta)
 -    \sum^{\rho-\lambda-1}_{\alpha=0}  c_{\rho-\lambda-\alpha-1} \Phi_{\alpha-u}(y-\beta)
 \\
 &\hspace*{-1cm}=  \sum^{\rho-\lb-1}_{\al=u} c_{\rho-\lb-\al-1} P_{\al-u} (y-\beta)
 -    \sum^{\rho-\lambda-1}_{\alpha=0}  c_{\rho-\lambda-\alpha-1} \Phi_{\alpha-u}(y-\beta)
 \\
   \end{aligned}$$
%
%
%
 \be \begin{aligned}  
 &= \Id_{u\leq -1} \sum^{ -1}_{\al=u} c_{\rho-\lb-\al-1} P_{\al-u} (y-\beta)+
 \sum^{ \rho-\lb-1}_{\al=0} c_{\rho-\lb-\al-1} P_{\al-u} (y-\beta)
\\&~~ -F(y-\beta)
   \sum^{\rho-\lb-1}_{\al=0} c_{\rho-\lb-\al-1} P_{\al-u} (y-\beta)
   -F'(y-\beta) \sum^{\rho-\lb-1}_{\al=0} c_{\rho-\lb-\al-1} Q_{\al-u-1}   (y-\beta) 
 \\
 &= \Id_{u\leq -1} \sum^{ -1}_{\al=u} c_{\rho-\lb-\al-1} P_{\al-u} (y-\beta)
+G(y-\beta)
   \sum^{\rho-\lb-1}_{\al=0} c_{\rho-\lb-\al-1} P_{\al-u} (y-\beta)
  \\&~~ +G'(y-\beta) \sum^{\rho-\lb-1}_{\al=0} c_{\rho-\lb-\al-1} Q_{\al-u-1}   (y-\beta) 
  \\&= \Id_{u\leq -1}\sum^{ -1}_{\al=u} c_{\rho-\lb-\al-1} P_{\al-u} (y-\beta) +
   \sum^{\rho-\lb }_{\al=1} c_{\rho-\lb-\al} \Psi_{\al-u-1} (y-\beta) .\end{aligned}
 \ee
$\bullet$ For $u\geq 1$ and $\rho>\lb$, using again Corollary \ref{C2.2} and $Q_{-1}=0 $ in the first equality,
 \be
 \begin{aligned}
 &{\mathcal B}^{\beta,y-\beta}_u  (\lambda-\rho)\\
 &=  \sum^{\rho-\lb-1}_{\al=u} c_{\rho-\lb-\al-1} P_{\al-u} (y-\beta)-F(y-\beta)
   \sum^{\rho-\lb-1}_{\al=u} c_{\rho-\lb-\al-1} P_{\al-u} (y-\beta)
 \\
 &~~ -F'(y-\beta) \sum^{\rho-\lb-1}_{\al=u} c_{\rho-\lb-\al-1} Q_{\al-u-1}   (y-\beta)
- \sum^{u-1}_{\al=0} c_{\rho-\lb-\al-1} \Phi_{\al-u} (y-\beta)
\\ &   =    \sum^{\rho-\lb-1}_{\al=u} c_{\rho-\lb-\al-1}  \Bigl(G(y-\beta) P_{\al-u}(y-\beta)+G'(y-\beta)   Q_{\al-u-1} (y-\beta)\Bigr)
 \\
 &~~     + \sum^{u-1}_{\al=0} c_{\rho-\lb-\al-1} \Psi_{\al-u} (y-\beta)
  =  \sum^{\rho-\lb }_{\al=1} c_{\rho-\lb-\al} \Psi_{\al-u-1} (y-\beta).\label{B}
\end{aligned}
\ee
This shows the different cases (\ref{Bexpr}), which establishes the formula for ${\mathcal B}^{\beta,y-\beta}_u  (\lambda-\rho)$ stated   in Proposition \ref{PAB}. 

The kernel  ${\mathcal K}^{\beta}$ in (\ref{defAB}) can then be expressed in terms of the expressions  (\ref{herm1}) and (\ref{Phi-Psi}) for $\widetilde H$ and $\Phi$, upon expanding $(1-\zeta /\om)^{-1}$ in the integral,
\be
\begin{aligned}
{\mathcal K}^{\beta}(\lb-\rho,\kappa-\rho)
&=\sum_{\al=0}^{\rho-\kappa-1}
(\sqrt{2})^{\lb-\kappa} \widetilde H_{\rho-\kappa-\al-1}(\beta\sqrt{2}) \Phi_{\lb-\rho+\al+1}(-\beta\sqrt{2} )
\\
&=\Id_{\kappa\leq \rho-1} \sum_{\al=0}^{\rho-\kappa-1} c_{\rho-\kappa-\al-1}\tilde c_{\lambda-\rho +\al+1}.
\end{aligned}
\ee
From this expression for the kernel we also have that 
$$\begin{aligned}\det &\left(\Id-{\mathcal K}^{ \beta}  ( \lambda ,\kappa  )\right)_{[-1,-\rho]}\\
& =\det\left(
\Id-\left(
\begin{array}{cccccccc}
c_0\\
c_1&c_0&&  \mbox{\Large O}\\
\vdots\\
& & \ddots&\ddots\\
c_{\rho-1}&\ldots&&c_1&c_0
\end{array}\right)
\left(
\begin{array}{cccccccc}
\widetilde c_{0}&\widetilde c_{-1}&\ldots&& \widetilde c_{1-\rho}\\
\widetilde c_1&\widetilde c_0&\ldots&& \widetilde c_{2-\rho}   \\
\vdots\\ 
& & \ddots&\ddots&\widetilde c_{-1}\\
\widetilde c_{\rho-1}&\ldots&&\widetilde c_1&\widetilde c_0
\end{array}\right)
\right),
\end{aligned}
 $$
yielding the explicit expression (\ref{ineq}) for the Fredholm determinant, upon using the inverse matrix $C_\rho^{-1}$, as in (\ref{Cinv}), together with the fact that $\det  C_{\rho}=1$, since $c_0=1$.

Finally, upon expanding $(1-z/w)^{-1}$ in the double integral (\ref{8.6}), the GUE-minor kernel can be expressed as:
\be
\begin{aligned}
&\frac 12 \BK^ {\mbox{\tiny minor}}  (u_1,\beta-y_i;u_2,\beta-y_j)  
\\
&
= -\Id_{u_1>u_2}   \BH^{u_1-u_2} (2(y_j\!-\!y_i))+\sum_{\al=0}^{u_1-1} \oint_{\Gamma_0}\frac {dz}{2\pi \I}\frac{e^{-z^2+2z(\beta-y_i)}}
{z^{u_1-\al }}
\int_L  \frac {dw}{2\pi \I} \frac{e^{w^2+2w(y_j-\beta)}}
{w^{\al-u_2+1}}
\\&
=-\Id_{u_1>u_2}  \BH^{u_1-u_2} (2(y_j\!-\!y_i))+\sum_{\al=0}^{u_1-1} \tilde H_{u_1-\al-1}(\beta-y_i)
\Phi_{\al-u_2}(y_j-\beta),
\end{aligned}
\label{8.6''}\ee
yielding (\ref{8.6'}) in the notation (\ref{gh}), thus ending the proof of Proposition \ref{PAB}.\qed


\section{One-level density}

In this section, we give the density of the  particles of the tacnode process, with kernel (\ref{dGUE}), for each of the even levels $\xi\geq 2n-2\rho$; remember the coordinates $\xi $ of  Figure 4. It will be convenient to use a new coordinate $u=n-\xi/2$; see also Figure 5.  
  The following Theorem will confirm the densities (\ref{outover}) and (\ref{Kout}) of Theorem \ref{Th1.4} for the edge-tacnode process, outside and inside the overlap. We will denote all densities and probabilities related to the tacnode process with a superscript ${}^{\mbox{\tiny TAC}}$. See footnote 5 for the constant $c_{\rho,\dt}$.
%
  %
  \begin{theorem}\label{mainTh1}
  %
%
 The density of particles $z^{(u)}\in \BR^{\max(\rho-u,\rho)}$ at level $u\leq \rho$ is given by the product of two determinants, up to a multiplicative constant\footnote{The formula below has the be understood as follows: for $u\geq 0$ the bracketed part of the  first matrix is absent, whereas for $u\leq 0$, the bracketed part of the  second matrix is absent. },
$$
\begin{aligned}
  \BP^{\mbox{\tiny\rm TAC}}&({\bf z}^{(u)}\in d{\bf z})\\
 =&\det \left(\BK_{\beta, \rho}^{\mbox{\tiny\rm TAC}} 
 (u,z_i;u,z_j) \right)_{1\leq i,j\leq \max(\rho-u,\rho)} d{\bf z}
  \\  
=&c_{\rho,\max(-u,0)}~
\\
&\times\det \left(  
 \begin{array}{c}
 y ^0.\Id_{u\leq -1}\\
\vdots\\ \vdots\\
y ^{-u-1}.\Id_{u\leq -1}\\ \\
\Phi_{-u}(\beta-y  )\\
\vdots\\
\Phi_{\rho-u-1}(\beta\!-\!y  )\\
\end{array}
\right)
\times\det \left(  
 \begin{array}{c}
y ^{0}~ \frac {e^{-(y +\beta)^2}}{\sqrt{\pi}}\\
\vdots\\
y ^{\rho-u-1}~ \frac {e^{-(y +\beta)^2}}{\sqrt{\pi}}\\
\Phi_{0}(\beta+y ).\Id_{u\geq 1}\\
\vdots\\ \vdots\\
\Phi_{\max(0,u)-1}(\beta\!+\!y  ).\Id_{u\geq 1}
\\
 \end{array}
\right)_{y=z_1,\ldots,z_{\rho+\max(-u,0)}}\\
\end{aligned}
$$

 \vspace*{-5.0cm}

 \hspace*{.9cm}-u$\left\{\begin{array}{c}\\ \\  \\ \\ \end{array}\right.$


 \vspace*{.5cm}

 \hspace*{10.9cm}$\left.\begin{array}{c}\\   \\ \\ \end{array}\right\}u$

  \vspace*{.9cm}
 
 \noindent This gives in particular the formulas for the densities stated in Theorem \ref{Th1.4}, the density (\ref{outover}) outside the overlap and the one (\ref{Kout}) inside the overlap.

 \end{theorem}

\proof 
 \underline{\em Outside the overlap}, that is for $u\leq 0$,  set $u=u_1 = u_2=-\dt \leq 0$; it then follows from (\ref{2min}) and (\ref{8.6'}) that the double Aztec kernel can be expressed as an inner product of vectors $ v^o_{y}$ and $ w^o_{y}$ of size $n:=\rho+\dt=$ size of overlap $+$ distance to overlap:
 $$
\begin{aligned}
&\frac 12\mathbb{K}^{ \textrm{\tiny{TAC}}}_{\beta,\rho}   (-\dt,y_1 ; -\dt, y_2)
\\
& =  \sum^{n-1 }_{\lambda=0} \left(\Bigl((\Id - {\mathcal K}^{(\beta)} (\lambda-\rho,\kappa-\rho)\Bigr)^{-1}  {\mathcal A}^{\beta,y_1-\beta}_{-\dt} (\kappa-\rho)\right) {\mathcal B}^{\beta,y_2-\beta}_{-\dt} (\lambda-\rho)
\\
   &= \left\la \left(\begin{array}{c }\left((\Id\!-\!{\cal K}^\beta)^{-1}
{\mathcal A}_{-\dt}^{\beta,y_1-\beta}(\lb\!-\!\rho)\right)_{0\leq \lb\leq n-1}
 \end{array}\right),
\left(\begin{array}{cc}\left({\mathcal B}_{-\dt}^{\beta,y_2-\beta}
 (\lb\!-\!\rho)\right)_{0\leq\lb\leq n-1} 
\end{array}  \right)
 \right\ra
 \\
 &=:\left\la   v^o_{y_1},w^o_{y_2}   \right\ra.
\end{aligned}$$
Since $0\leq \lambda\leq \rho+\dt-1 $, the resolvent $(\Id - {\mathcal K}^{(\beta)}  ( \lambda-\rho,\kappa-\rho ))^{-1}$ can be represented as the inverse of a square block matrix of size $n=\rho+\delta$,
 with two square blocks of sizes $\rho$ and $\dt$, with $ {\mathcal K}^{ \beta }  ( \lambda\!\!-\!\!\rho,\kappa\!\!-\!\!\rho )$ as in Proposition \ref{PAB}, 
 $$
\left(
\begin{array}{cc}
\Id_\rho - {\mathcal K}^{ \beta }  ( \lambda\!\!-\!\!\rho,\kappa\!\!-\!\!\rho )\Bigr|_{[0,\rho-1]}  &  O \\ \\
\ast  & \Id_{\dt}  
\end{array}
\right)^{-1}
\!\!\!\!=
\left(
\begin{array}{cc}
(\Id_\rho - {\mathcal K}^{ \beta }  ( \lambda\!\!-\!\!\rho,\kappa\!\!-\!\!\rho )\Bigr|_{[0,\rho-1]} )^{-1}  &  O \\ \\
\ast  & \Id_{\dt} 
\end{array}
\right) 
$$
with Fredholm determinant (see also expression (\ref{ineq}))
$$\det \left(\Id_\rho - {\mathcal K}^{ \beta }  ( \lambda ,\kappa  )\Bigr|_{[-1,-\rho]} \right)^{-1}
.$$

Since at level $\xi=2n-2u=2(n+ \dt)$, there are $
  n=\rho+\dt$ particles, one computes the $n$-correlation, using the formula for ${\cal A}_{-\dt}$ in Proposition \ref{PAB}, formula (\ref{herm3}) and Corollary \ref{CA1'}; namely, one finds for $\dt\geq 0$,
 $$
\begin{aligned}
  \det \Bigl( \tfrac 12    \mathbb{K}^{{\textrm{\tiny TAC}}}_{\beta,\rho}  &(-\dt,y_i ; -\dt,y_j)\Bigr)_{1\leq i,j \leq n }
  =\det \left( \left\la v^o_{y_i} , ~w^o_{y_j} \right\ra\right)_{1\leq i,j \leq n  }
\\
&=\det \left (v^o_{y_1},\ldots,v^o_{y_{n}}\right)
  \det \left (w^o_{y_1},\ldots,w^o_{y_{n}}\right)
\\
&=   \det \left(\Id - {\mathcal K}^{(\beta)} (\kappa-\rho ; \lambda-\rho)\right)^{-1}_{0\leq \kappa,\lambda\leq n-1}
\\
&\qquad   \det \Bigl({\mathcal A}^{\beta,y_{i+1}-\beta}_{-\dt} (\kappa-\rho) \Bigr)_{0\leq k,i\leq n-1} \det \Bigl({\mathcal B}^{\beta,y_{j+1}-\beta}_{-\dt} (\lambda-\rho) \Bigr)_{0\leq \lambda,j\leq n-1} 
\\
&= \det \left(\Id-{\mathcal K}^{(\beta)}(\kappa-\rho, \lambda-\rho)\right)^{-1}_{0\leq \kappa, \lambda\leq \rho-1} \Delta_{\rho+\dt} (y_1,\ldots,y_{n})
\\
& \qquad \det \Bigl({\mathcal B}^{\beta,y_{j+1}-\beta}_{-\dt} (\lambda-\rho)\Bigr)_{0\leq \lambda, j\leq \rho+\dt-1} \frac{1}{2^{n}} \prod^{\rho+\dt}_{i=1} \frac{e^{-(y_i+ \beta)^2}}{\sqrt{\pi}},
\end{aligned}
$$
containing the matrix, indexed by $\lb=n-1,\ldots, \rho,\dots,0$ in that order,
$$
\begin{aligned}
&\Bigl({\mathcal B}^{\beta,y_{j+1}-\beta}_{-\dt} (\lambda-\rho)\Bigr)_{0\leq \lambda, j\leq n-1} 
\\&=\left(
\begin{aligned}
&\tilde H_0(y+\beta)
\\
& \quad \vdots
\\
&\tilde H_k (y+\beta)
\\
& \quad \vdots
\\
&  \tilde H_{\delta-1} (y+\beta)
\\
& c_0\Psi_{\dt}(y-\beta)+\mbox{$\sum_{\al=0}^{\dt-1}$}c_{\dt -\al }P_{\al }(y-\beta)
\\
& \quad \vdots
\\
& c_0\Psi_{\dt+k}(y-\beta)+\ldots+c_k\Psi_\dt(y-\beta)+\mbox{$\sum_{\al=0}^{\dt-1}$}c_{\dt +k-\al }P_{\al }(y\!-\!\beta)
\\
& \quad \vdots
\\
&  c_0\Psi_{\dt+\rho -1}(y-\beta)+\ldots+c_{\rho-1}\Psi_{\dt }(y-\beta)+\mbox{$\sum_{\al=0}^{\dt-1}$}c_{ \dt     +\rho-1 -\al}P_{\al }(y\!-\!\beta)
\end{aligned}
\right)
\end{aligned} \Bigr|_{y=y_1,\dots , y_{n}}$$
Notice that the mixed entries of this column contain polynomials of degree at most $\dt-1$. Therefore, they can be removed by means of row operations with the first $\dt$ entries. So, by row operations and since $c_0=1$ by (\ref{ck}), the matrix of the vectors leads, modulo the factor $\prod_1^{\dt-1} 2^i/i!$, to a matrix of size $n=\rho+\dt$ with columns $y=y_i$ for $1\leq k\leq n$ given by:
 $$
 \left(
\begin{array}{c}
 1
\\
  \quad \vdots
\\
 y^k  
\\
  \quad \vdots
\\
   y^{\delta-1}  
\\
   \Psi_{\dt}(y-\beta)
\\
  \quad \vdots
\\
   \Psi_{\dt+k} (y-\beta)\\
  \quad \vdots
\\
    \Psi_{n -1} (y-\beta)
\end{array}
\right)
$$
Thus, using (\ref{PhiPsi}), we have shown Theorem \ref{mainTh1} for $u=-\delta$ (a level at a distance $\dt$ outside the overlap) and thus the density of particles is given by the second formula (\ref{outover}) for ${\bf y}^{(n)}$ of Theorem \ref{Th1.4}; the first formula for ${\bf x}^{(n)}$ follows then from the involution of the double Aztec diamond, alluded to just before formula (\ref{2min}); this reflected itself in the second formula (\ref{2min}) of the tacnode kernel ${\mathbb K}^{\rm \tiny \mbox{\tiny TAC}}_{\beta,\rho}$.



\bigbreak

 \noindent \underline{\em Inside the overlap}, namely for $1\leq u\leq \rho-1$, the kernel is expressible in terms of two vectors $ v^\iota_{y}$ and $ w^\iota_{y}$ (defined below) of size $\rho+u$,  using (\ref{2min}) and (\ref{8.6'}) :
\be\begin{aligned}
\frac 12&{\mathbb K}^{{\textrm{\tiny TAC}}} _{\beta,\rho}(u,y_1;u,y_2) 
\\& =  \sum^{\rho-1 }_{\lambda=0} \left(\Bigl((\Id - {\mathcal K}^{(\beta)} (\lambda-\rho,\kappa-\rho)\Bigr)^{-1}  {\mathcal A}^{\beta,y_1-\beta}_u (\kappa-\rho)\right) {\mathcal B}^{\beta,y_2-\beta}_u (\lambda-\rho)
\\
&~~~~+  \sum_{\al=1}^u \tilde H_{u-\al}(\beta-y_1)
\Phi_{\al-u-1}(y_2-\beta)\\
 =& \left\la \left(\begin{array}{c }\left((\Id\!-\!{\cal K}^\beta)^{-1}
{\mathcal A}_{u}^{\beta,y_1-\beta}(\lb\!-\!\rho)\right)_{0\leq \lb\leq \rho-1}
\\ \\
\left(\tilde H_{u-\al }(\beta-y_1)\right)_{1\leq \al\leq u }\end{array}\right),
\left(\begin{array}{cc}\left({\mathcal B}_u^{\beta,y_2-\beta}
 (\lb\!-\!\rho)\right)_{0\leq\lb\leq \rho-1}\\ \\
 \left(\Phi_{\al-u-1}(y_2-\beta)\right)_{1\leq \al\leq u }
\end{array}  \right)
 \right\ra
 \\
 &=:\left\la   v^\iota_{y_1},w^\iota_{y_2}   \right\ra
\end{aligned}\label{Kin}\ee


\noindent Step 1. {\em The wedge product of the vectors (written coordinate-wise), setting $u=\dt$,
\be \begin{aligned}
 w^\iota_{y}  = &   \left(\begin{array}{lll}
\Bigl({\mathcal B}^{\beta,y-\beta}_\dt (\lb-\rho)\Bigr)_{0 \leq \lambda \leq \rho-1}
\\
 \Bigl(\Phi_{\al-\dt-1}  (y-\beta)\Bigr)_{1 \leq \al \leq \dt }
\end{array}
\right)\\
=&    \sum^{\rho}_{\lb=1}  {\mathcal B}^{\beta,y-\beta}_\dt      (\lb-\rho-1) e_{\lb } + \sum^\dt_{\al=1} \Phi_{\al-\dt-1}  (y-\beta) e_{\rho+\al} ,
   \end{aligned}
\label{w-in}\ee 
 can be expressed as follows, 
 \[
 \begin{aligned}
 &w^\iota_{y_1} \wedge \ldots \wedge w^\iota_{y_\rho}=  \det \Bigl( \Psi_{\mathcal B}^{(\rho)} (y)\Bigr)~~        e_1 \wedge \ldots \wedge      e_{\rho-\dt} \wedge \tilde u_1 \wedge \ldots \wedge    \tilde u_\dt ,
 \end{aligned}
 \]
where the matrix $\Psi_{\mathcal B}^{(\rho)} (y)$ of size $\rho$ is defined as  
\be\Psi_{\mathcal B}^{(\rho)} (y):=\Bigl(\Psi_k (y_\ell\!-\!\beta)\Bigr)_{-\dt \leq k \leq \rho-\dt-1   \atop     1\leq \ell \leq \rho} 
=\left(\begin{array}{c}-\frac{e^{-(y-\beta)^2}}{2^\dt\sqrt{\pi}}H_{\dt-1}(\beta-y)\\
\vdots
\\
-\frac{e^{-(y-\beta)^2}}{2\sqrt{\pi} }H_{0}
\\
(-1)^0\Phi_0(\beta-y)
\\
\vdots
\\
(-1)^{\rho-\dt-1}\Phi_{\rho-\dt-1}(\beta\!-\!y)
\end{array}\right)_{y=y_1,\ldots,y_\rho}
\label{PsiB}\ee
 and
$$\tilde u_k:=e_{\rho-\dt+k}-a_0 e_{\rho+1}-\ldots a_{k-1}e_{\rho+k}
 \mbox{   for   } 1\leq k\leq \dt,$$
 with $a_i$ the entries of $C^{-1}_\dt $, as in formula (\ref{Cinv}) of Corollary \ref{C2.2}.
}%


\noindent
 {\em Proof of Step 1}: Define column vectors ${\bf u}$ and $\widetilde {\bf u}$  (whose components are $\rho$ vectors)\footnote{Define the antidiagonal matrix $D_\al$ of size $\al$,
 \[
 D_\al = 
 \left(\begin{array}{ccc}
  &  & -1 
  \\
  O & -1 &  \\
   & \ddots & O \\
   -1 &   &  \end{array}\right)
 .\]}  :
  \be
\begin{aligned}
 {\bf u}  := \left(\begin{array}{c}u_{-\rho+\delta+1}\\
 \vdots\\u_0  \\ u_1\\ \vdots \\ u_\dt\end{array}\right) 
 = \left(\begin{array}{c|c|c}
 C_{\rho-\dt} & 0 & 0 
 \\
 \hline 
 0 & C_\dt & D_\dt 
 \end{array}\right) {\bf e}~~\mbox{with}~~
  {\bf e}  =  \left(\begin{array}{c} e_1\\ \vdots \\ e_{\rho-\dt} 
  \\ \vdots \\ e_\rho 
   \\ \vdots \\ e_{\rho+\dt}\end{array}\right)\end{aligned}
\label{u1}\ee
and
\be \begin{aligned}
{\bf \tilde u}  &: = \left(\begin{array}{c} \tilde u_{-\rho+\dt+1}\\ \vdots \\
\tilde u_0 \\ \tilde u_1\\ \vdots \\ \tilde u_\dt \end{array}\right)  
 := \left(\begin{array}{c|cc}
C^{-1}_{\rho-\dt} & O    \\
 \hline 
  O & C^{-1}_\dt 
 \end{array}\right)  \ {\bf u} 
  = \left(\begin{array}{c|c|c}
 I_{\rho-\dt} & 0 & 0 \\
 \hline 
 0 & I_\dt & C^{-1}_\dt D_\dt
 \end{array}\right)  \ {\bf e} 
\\
&=(e_1,\ldots,e_{\rho-\dt},
  e_{\rho-\dt+1} - a_0 e_{\rho+1},\ldots, e_\rho - a_0 e_{\rho+1} - \ldots-a_{\dt-1}e_{\rho+\dt})^\top
 \end{aligned}\label{utilde}\ee
 where $C_\al$ is the lower-triangular matrix (\ref{Cinv}), (remember $c_0=1$), where the entries of $C^{-1}_\dt$ are as in (\ref{Cinv}), (with $a_0=1$) and where $A$ is, modulo a sign, the matrix (\ref{Cupside}) of $a_i$'s; namely
 \[
 \ A:= C^{-1}_\dt D_\dt =-C^{\circlearrowleft~  -1}_{\dt}
 .\]
Written out, the $u_\ell$'s amount to:
 \be \label{ue}
\begin{aligned}
u_\ell  = &     \sum^{\rho-\dt+\ell }_{\lb=1} c_{\rho-\dt+\ell -\lb}   \    e_{\lb }    \quad\quad\quad  \textrm{for} \ -\rho+\dt+1 \leq \ell \leq 0
\\
            = &    \sum^{\rho-\dt+\ell}_{\lb=1} c_{\rho-\dt+\ell -\lb}   e_{\lb}    -   e_{\rho+\dt-\ell+1} \quad \textrm{for} \  1\leq \ell \leq \dt.
\end{aligned}
\ee
For  $0 \leq \al \leq u-1$, we have by (\ref{PhiPsi}) and (\ref{PhiPsi-}), that
 $$
\begin{aligned}
& \Phi_{\al-u} (y-\beta) = - \Psi_{\al-u} (y-\beta).
\end{aligned}
$$
Using (\ref{B}), we have the following vector-identity for $w^\iota_{y }$ in (\ref{w-in}), 
 $$
\begin{aligned}
w^\iota_{y}  
=& \sum_{\lb=1}^{\rho }\left(\sum_{\al=1}^{\rho-\lb+1 }
c_{\rho-\lb-\al+1 }   \Psi_{\al-\dt-1}(y-\beta) \right)e_{\lb }-  \sum^\dt_{\al=1} \Psi_{\al-\dt-1}  (y-\beta)      e_{\rho+\al}
\\
= &   \sum^\rho_{\al=1} \Psi_{\al-\dt-1} {(y-\beta)}
\left(\sum^{\rho-\al +1}_{\lb=1}  c_{\rho-\lb-\al +1}   e_{\lb } \right)   
-  \sum^\dt_{\al=1} \Psi_{\al-\dt-1}  (y-\beta)      e_{\rho+\al}
\\ 
=&    \sum^\dt_{\al=1}  \Psi_{\al-\dt-1}  (y-\beta)  
\left(\sum^{\rho-\al+1}_{\lb=1}  c_{\rho-\al+1-\lb}   e_{\lb } -e_{\rho+\al}  \right)  
\\
  &  \qquad + \sum^\rho_{\al=\dt+1}  \Psi_{\al-\dt-1}  (y-\beta) 
\left(\sum^{\rho-\al+1}_{\lb=1}  c_{\rho-\al+1-\lb}   e_{\lb}    \right)  
\end{aligned}
$$
$$\begin{aligned}
=  &     \sum^\dt_{\ell=1}  \Psi_{-\ell}  (y-\beta)  
\left(\sum^{\rho-\dt+\ell }_{\lb=1}  c_{\rho-\dt+\ell-\lb}  ~ e_{\lb} -e_{\rho+\dt-\ell+1}  \right)  
\\
    &  \qquad + \sum^0_{\ell=-\rho+\dt+1}  \Psi_{-\ell}  (y-\beta) 
 \sum^{\rho-\dt+\ell}_{\lb=1}  c_{\rho-\dt+\ell-\lb}  ~ e_{\lb}      
\\
=& \sum^{\dt}_{\ell=-\rho+\dt+1} \Psi_{-\ell} (y-\beta) u_\ell,
 \end{aligned}
$$
and so
\[
w^\iota_{y_1} \wedge\ldots\wedge w^\iota_{y_\rho} = \det \Bigl(\Psi_k (y_\ell-\beta)\Bigr)_{-\dt \leq k \leq \rho-\dt-1   \atop     1\leq \ell \leq \rho}          u_{-\rho+\dt+1}     \wedge\ldots\wedge u_\dt .
\]
%
%
 Now, using the latter, one finds the formula below, which in terms of the new set of vectors $ \tilde u_{-\rho+\dt+1},\ldots,
\tilde u_0,\tilde u_1,\ldots, \tilde u_\dt$, as in (\ref{utilde}), reads
%
 %
  %
$$
 \begin{aligned}
 \Om_w&=w^\iota_{y_1}  \wedge\ldots\wedge w^\iota_{y_\rho}  
\\ &= \det \left(\Psi_{\mathcal B} (y)\right) u_{-\rho+\dt+1}       \wedge \ldots \wedge   u_\dt
 \\
 &= \det \left(\Psi_{\mathcal B} (y)\right) \det (C_{\rho-\dt}) \det(C_\dt)
  e_1 \wedge \ldots \wedge    e_{\rho-\dt}    \wedge   (e_{\rho-\dt+1} - a_0 e_{\rho+1}) \wedge  
 \\
 & \qquad\qquad\qquad\qquad\qquad \qquad\quad  \ldots \wedge  (e_\rho - a_0 e_{\rho+1} - \ldots-a_{\dt-1}e_{\rho+\dt})
 \\
 &= \det (\Psi_{\mathcal B} (y)) e_1 \wedge \ldots \wedge      e_{\rho-\dt} \wedge \tilde u_1 \wedge \ldots \wedge    \tilde u_\dt
 \end{aligned}
$$
 ending the proof of Step 1. 

 \bigbreak
Step 2.  {\em Here we write:} 
 $$
\begin{aligned}
v^\iota_{y } =
 \left(\begin{array}{l}
 \Bigl (\Id - {\mathcal K})^{-1} {\mathcal A}^{\beta,y-\beta}_\dt  (\lb-\rho -1) \Bigr)_{1\leq \lb \leq \rho}
 \\
 \Bigl(\widetilde H_{\dt-\al} (\beta-y) \Bigr)_{1\leq \al \leq \dt}
 \end{array}\right)   =     
  \left(\!\begin{array}{c}
 v^\iota  _{1,y}
  \\
  \vdots
  \\
 v^\iota_{\rho+\dt,y}
 \end{array}\!\!\right)   
 = \sum^{\rho+\dt}_{i=1}   v^\iota_{i,y} e_i~,
 \end{aligned}
 $$
 and 
 $$\Om_v= v^\iota_{y_1}   \wedge\ldots\wedge v^\iota_{y_\rho}
=\sum^{\rho+\dt}_{i=1} v_{i,y_1}   e_i \wedge \ldots \wedge    \sum^{\rho+\dt}_{i=1}  v_{i,y_\rho}  e_i ~.$$
 \bigbreak
Step 3. {\em To compute (\ref{Kin}), one uses Corollary \ref{CA1'}, but also the wedge products obtained in steps 1 and 2: 
 $$\begin{aligned}
&\det \Bigl( \tfrac 12 \mathbb{K}^{{\rm {\tiny TAC}}} _{\beta,\rho}  ( \dt,y_i ;  \dt,y_j)\Bigr)_{1\leq i,j \leq \rho   }
\\
 &=\det \left( \left\la v^\iota_{y_i} , ~w^\iota_{y_j} \right\ra\right)_{1\leq i,j \leq \rho   }
\\
  \\&=\bigl\langle   
 v^\iota_{y_1}   \wedge\ldots\wedge v^\iota_{y_\rho}
 , 
 w^\iota_{y_1}  \wedge\ldots\wedge w^\iota_{y_\rho}
      \bigr\rangle
\\&=\left \langle \Om_v,\Om_w\right\rangle
\\ &=\det  \Psi_{\mathcal B} (y) 
 \langle   \sum^{\rho+\dt}_{i=1} v_{i,y_1}   e_i \wedge \ldots \wedge    \sum^{\rho+\dt}_{i=1}  v_{i,y_\rho}  e_i~~,~~ e_1  \wedge \ldots \wedge   e_{\rho-\dt} \wedge \tilde u_1 \wedge \ldots \wedge \tilde u_\dt  \rangle
 \\  &=\det \Psi_{\mathcal B} (y_1,\ldots,y_\rho)  .~\det \Psi_{\mathcal A} (y_1,\ldots, y_\rho),\end{aligned} $$
 with $\Psi_{\mathcal A}$ defined as
  \be
 \begin{aligned}
 \Psi_{\mathcal A}  (y_1,\ldots,y_\rho) 
:=(\Id-{\mathcal K}^{(\beta)}\bigr|_{[-1,-\rho]})^{-1} 
\left(\begin{array}{l}
\frac {e^{-(y+\beta)^2}}{2^{\rho-\dt }\sqrt{\pi}}
H_{\rho-\dt-1 }(y+\beta)
\\
\vdots
\\
\frac {e^{-(y+\beta)^2}}{2^{ }\sqrt{\pi}}
H_{0 } 
\\
(-1)^1\Phi_{0}( y +\beta)
\\
\vdots
\\
(-1)^{\dt }\Phi_{\dt-1}( y+\beta)
\end{array}\right)_{y=y_1,\ldots,y_\rho}
.\end{aligned}
\label{PsiA}\ee}
 Indeed, to compute $\Psi_{\mathcal A}$, we use Corollary \ref{CA1''}, 
$$
 \begin{aligned}
 \Psi_{\mathcal A}& (y_1,\ldots,y_\rho) \\&=
 \left(\begin{array}{l}
 (v_{i,y_j})_{1\leq i \leq \rho-\dt \atop 1 \leq j \leq \rho}
 \\
 (I_\dt \bigl | A_\dt) (v_{\rho-\dt+i,y_j})_{1\leq i \leq 2\dt \atop 1 \leq j \leq \rho}
  \end{array}\right)  
  \\
  &=
   \left(\begin{array}{l}
  v_{1,y }
  \\
  \vdots
  \\
  v_{\rho-\dt,y }
  \\
  v_{\rho-\dt+1,y }-a_0 v_{\rho+\dt,y }
  \\
  v_{\rho-\dt+2,y }-a_0 v_{\rho+\dt-1,y } -a_1 v_{\rho+\dt,y }
  \\
  \vdots
  \\
  v_{\rho,y } - a_0 v_{\rho+1,y } - a_1 v_{\rho+2,y }-\ldots-a_{\dt-1} v_{\rho+\dt,y } 
     \end{array}\right)_{y=y_1,\ldots,y_\rho}  
 \end{aligned}$$
 \be \begin{aligned}  %
\\ \\ & = \left(
 \begin{aligned}
& (\Id-{\mathcal K})^{-1} {\mathcal A}^{\beta,y-\beta}_\dt (-\rho)
 \\
 & \qquad  \vdots
 \\
 & (\Id-{\mathcal K})^{-1}  {\mathcal A}^{\beta,y-\beta}_\dt (-\dt-1)
 \\
 & (\Id-{\mathcal K})^{-1}  {\mathcal A}^{\beta,y-\beta}_\dt (-\dt) - a_0 \widetilde H_0 (\beta-y)
 \\
 & \qquad \vdots
 \\
 & (\Id - {\mathcal K})^{-1}    {\mathcal A}^{\beta,y-\beta}_\dt   (-\dt+\mu-1) - \sum_{p+q=\mu-1}  a_p \widetilde H_q (\beta-y)  
 \\
 &\qquad \vdots
 \\
& (\Id-{\mathcal K})^{-1}  {\mathcal A}^{\beta,y-\beta}_\dt (-1) - \sum_{p+q=\dt-1}  a_p \widetilde H_q (\beta-y) 
 \end{aligned}
 \right)_{y=y_1,\ldots,y_\rho}
\\
& =\left( \sum\limits^\rho_{\mu=1} (\Id-{\mathcal K})^{-1} {\mathcal A}^{\beta,y-\beta}_\dt (  -\rho+\mu-1) e_\mu - \sum ^\dt_{\mu=1} e_{\mu+\rho-\dt} \sum\limits_{p+q=\mu-1} a_p \widetilde H_q (\beta-y) \right)_{y=y_1,\ldots,y_\rho}
 \end{aligned}.
 \label{PsiA'}\ee
 The matrix $\Id-{\mathcal K}$ in (\ref{PsiA'}) is in short the matrix $ (\Id - 
  {\mathcal K} (\lb , \kappa ) )_{-1 \leq \lb,\kappa\leq -\rho}$, 
 where, by formula (\ref{Kscrip}), the $\mu^{\textrm{th}}$ column for $1 \leq \mu \leq \rho$ is given by
$$
\begin{aligned}
(\Id - {\mathcal K} (\lb -\rho -1,\kappa-\rho-1)) e_\mu 
 =     \Bigl(\dt_{\lb,\mu} - \sum^{\rho-\mu}_{\al=0} c_{\rho-\mu-\al} \bar c_{\lb-\rho+\al}\Bigr)_{1\leq \lb \leq \rho}
.\end{aligned}
$$
%
%
Acting with $(\Id-{\mathcal K}) $ on this matrix, using (\ref{Cinv}) and the second equation (\ref{Hermid}), together with Lemma \ref{L2.1}, one finds for its $(\lb,y)$th entry with $1\leq \lb \leq \rho,$ 
$$
\begin{aligned}
 {\mathcal A}^{\beta,y-\beta}_\dt& (\lb-\rho-1)  - \sum\limits^\dt_{\mu=1} ((\Id-{\mathcal K}) e_{\mu+\rho-\dt} )(\lb)  \sum\limits_{p+q=\mu-1} a_p \widetilde H_q (\beta-y)
\\
=& {\mathcal A}^{\beta,y-\beta}_\dt   (\lb-\rho-1)  
  -\sum\limits_{\mu=1}^\dt\Bigl(\dt_{\lb,\rho-\dt+\mu}-{\mathcal K}(\lb-\rho-1,\mu-\dt-1)\Bigr) \sum\limits_{p+q=\mu-1}  a_p \widetilde H_q (\beta-y)  
  \\
= \ & {\mathcal A}^{\beta,y-\beta}_\dt (\lb-\rho-1)
 -  \sum^\dt_{\mu=1} \dt_{\lb,\rho-\dt+\mu} \sum^{\mu-1}_{q=0}   a_{\mu-q-1} \widetilde H_q (\beta-y)
\\
& + \sum^\dt_{\mu=1}   \sum^{\dt-\mu}_{\al=0} c_{\dt-\mu-\al} \tilde c_{\lb-\rho+\al} \sum^{\mu-1}_{q=0} a_{\mu-q-1} \widetilde H_q (\beta-y)
\end{aligned}
$$
$$
\begin{aligned}
 = \ & \Phi_{\lb-\rho+\dt-1} (-y-\beta) - \sum^{\dt-1}_{\al=0} \tilde c_{\lb-\rho+\dt-\al-1} \widetilde H_\al (\beta-y)
\\
& - \sum^{\lb-(\rho-\dt)+1}_{q=0}  a_{\lb-\rho+\dt-q-1}  \widetilde H_q (\beta-y)
 \\
 &
  + \sum^{\dt-1}_{q=0} \widetilde H_q (\beta-y) \sum^{\dt-q-1}_{\al=0} \tilde c_{\lb-\rho+\al} \sum^{\dt-\al}_{\mu=q+1} a_{\mu-q-1} c_{\dt-\al-\mu}
\\
= \ & \Phi_{\lb-\rho+\dt-1} (-y-\beta) - \sum^{\dt-1}_{\al=0} \tilde c_{\lb-\rho+\dt-\al-1} \widetilde H_\al (\beta-y)
\\
& - \sum^{\lb-(\rho-\dt)-1}_{q=0}  a_{\lb-(\rho-\dt)-q-1}  \widetilde H_q (\beta-y)
  + \sum^{\dt-1}_{q=0} \widetilde H_q (\beta-y) \sum^{\dt-q-1}_{\al=0} \tilde c_{\lb-\rho+\al}  \dt_{\al,\dt-q-1}
\\
= \ & \Phi_{\lb-\rho+\dt-1} (-y-\beta) - \sum^{\lb-(\rho-\dt)-1}_{q=0} a_{\lb-(\rho-\dt)-q-1}  \widetilde H_q (\beta-y)=:\star
\end{aligned}
$$
Then for $ 1\leq\lb\leq\rho-\dt$, we have by (\ref{PhiPsi}) and (\ref{PhiPsi-})
$$
\begin{aligned}
\star = \ &   \Phi_{\lb-\rho+\dt-1} (-y-\beta)=\frac {e^{-(y+\beta)^2}}{2^{\rho-\dt-\lb+1}\sqrt{\pi}}
H_{\rho-\dt-\lb }(y+\beta)=\frac {e^{-(y+\beta)^2}}{2\sqrt{\pi}} y^{\rho-\dt-\lb}+\ldots
,\end{aligned}
$$
and for $\rho-\dt+1 \leq \lb \leq \rho$, we have, 
using (\ref{Hermid}) and (\ref{PhiPsi}),
$$
\begin{aligned}
\star =   &  \Phi_{\lb-\rho+\dt-1} (-y-\beta)  - P_{\lb-\rho+\dt-1} (-y-\beta)
\\
=   & - G (-y-\beta) P_{\lb-\rho+\dt-1} (-y-\beta) - G' (-y-\beta) Q_{\lb-\rho+\dt-2} (-y-\beta)
\\
=  & - \Psi_{\lb-\rho+\dt-1}   (-y-\beta)
\\
= &(-1)^{\lb-\rho+\dt }\Phi_{\lb-\rho+\dt-1}(y+\beta)
.\end{aligned}
$$
 This establishes the formula (\ref{PsiA}) for $\Psi_{\mathcal A}$.
 
 So we have
 $$
\begin{aligned}
& \det \left(\frac 12 \BK_{\beta, \rho}^  {\mbox{\tiny \rm TAC}}   ( \delta,y_i; \delta,y_j) \right)_{1\leq i,j\leq \rho} 
 = \det \Psi_{\mathcal B} (y_1,\ldots,y_\rho)   \det \Psi_{\mathcal A} (y_1,\ldots, y_\rho)   
  \end{aligned}
 $$ 
 with the matrices $\Psi_{\mathcal B}$ and $\Psi_{\mathcal A}$ as in (\ref{PsiB}) and (\ref{PsiA}), 
 which by row operations leads to formula (\ref{Kout}), thus ending the proof of Theorem \ref{mainTh1}.\qed 
 
 
  \section{Two-level density}
  
  This section deals with the joint density of the particles ${\bf z}^{(n)}={\bf x}^{(n)}$ and ${\bf z}^{(-\dt)}={\bf y}^{(n)}$ in the double Aztec diamond model, at two levels $u=\rho+\dt=n$ and $u=-\dt$, for $\dt\geq 0$; thus the two levels are symmetric with regard to the overlap; see Figure 5. This will establish formula (\ref{joint}) of Theorem \ref{Th1.4} for the diamond model, if we show that $\Gamma_{{\bf x}{\bf y}} =\mbox{Vol}({\mathcal C}_{{\bf x}{\bf y}})$, where $ {\mathcal C}_{{\bf x} {\bf y}} $ is the double cone (\ref{cone}); this identity will be shown in the next section. 
  
 \begin{theorem}\label{Th:JProb}
 The joint probability of the particles at distinct levels for the tacnode process is given by:
     \be\begin{aligned}
 \BP^{\mbox{\tiny\rm TAC}}& \left(
  \{{{\bf x}^{(n)}\in d{\bf x} ,~{\bf y}^{(n)}\in d{\bf y}  }  \}
 \right)
 =\frac{\rho_n^{ \mbox{\tiny GUE}}({\bf x}\!-\!\beta) 
 ~ \rho_n^{\mbox{\tiny GUE}}( {\bf y}\!+\!\beta)}
 { \det(\Id- {\mathcal K}^{\beta}
 )_{[-1, -\rho]}    }
 \frac { \Gamma_{{\bf x}{\bf y} }
  d{\bf x}d{\bf y} }{\mbox{\em Vol} ({\mathcal C}_{{\bf x} })~ \mbox{\em Vol} ({\mathcal C}_{ {\bf y} })}
, \end{aligned}
 \label{joint1}\ee   
where          
       \be\begin{aligned}
\Gamma_{{\bf x} {\bf y}} 
&= (-1)^{n-1}  \det\left[    \frac { (y_j- x_i) ^{n-1}  }{ (n-1)!}   ~
 -\sum_{k=1}^{n} {\mathcal L}_{ik} 
 {\mathbb H}^{n+\dt}( y_j-x_k )\right]_{1\leq i,j\leq n} .
\end{aligned}\label{vol1}\ee
  In this formula,  ${\mathbb H}^{n+\dt}$ is the Heavyside function, as in (\ref{8.6}) and ${\mathcal L}$ is a $n\times n$ matrix, of rank\footnote{In particular the sum of the columns $=0$.} $\rho$ and of trace $ 0$ with entries
\be
 {\mathcal L}_{ij} =\left\{\begin{aligned}&~ {(-1)^\dt  }{\dt}
 \frac          {\hat R^{(\dt-1)}_{ij}(x_i-\beta) }   {R'(x_j-\beta)}       
 \mbox{  if $i\neq j$}
\\  &\frac{(-1)^{\dt }}{ \dt+1 } \frac{R^{(\dt+1)}(x_i-\beta)}{R'(x_i-\beta)}  
 \mbox{  if $i= j$,}\end{aligned}\right.
 \label{vol2}\ee
 with polynomials $R(z)$ and $\hat R_{ij}(z)$ of degree $n$ and $n-2$:
 \be \begin{aligned}
 R(z)&:=\prod_{1\leq k\leq n} (z-x_k+\beta) \mbox{     and   }
 \hat R_{ij}(z) :=\prod_{{1\leq k\leq n}\atop{k\neq i,j}}(z-x_k+\beta).
 \end{aligned}  \label{R}\ee

\end{theorem}  

\noindent {\em Remark}: Although formula (\ref{vol1}) involves rational functions, the determinant will nevertheless be a polynomial in the $x_i$ and $y_i$'s.

 \vspace*{.5cm}
\noindent  Before giving the proof we need some preliminary facts. Setting $-u=\dt\geq 0$, 
recall from Proposition \ref{PAB}, the expression for $0\leq \kappa\leq \rho+\dt-1$,
$$\begin{aligned}  {\cal A}^{\beta,y-\beta }_{-\dt }(\kappa-\rho)& =  
 \frac{e^{-(y+\beta)^2}}{\sqrt{\pi}}
  \frac{H_{\rho-\kappa+\dt-1}(y+ \beta)}{2^{\rho-\kappa+\dt }}
   \end{aligned}
$$
The expression for 
${\cal A}^{\beta,x-\beta }_{\rho+\dt }(\lambda-\rho)$, which seemingly appears in the calculation, will actually disappear. 
 
Also, recall from Corollary \ref{C2.3} the matrix $C^{\circlearrowleft}_{\rho} $. Then from (\ref{B}), (\ref{B'}) and (\ref{B''}), it follows that for $1\leq i\leq \rho+\dt$ and $0\leq \lb\leq \rho+\dt-1$,     
   \be\begin{aligned}  
 \bullet~~&{\mathcal B}^{\beta,y-\beta}_{-\dt}  (i-\rho-1) \\
&=\widetilde H_{\rho-i+\dt }(y+ \beta)
  - 
   \sum^{\rho-i}_{\alpha=0}  c_{\rho-i-\alpha } \Phi_{\alpha+\dt}(y-\beta)
 \\
 &=\tilde H_{\rho-i+\dt }(y+ \beta)
  - 
   \{C^{\circlearrowleft}_{\rho} \bigl(\Phi_{\dt}(y-\beta), \ldots, \Phi_{\dt+\rho-1}(y-\beta) \bigr)^{\top} \}_i
\\
\\&=C^{\circlearrowleft }_{\rho}\left\{ \Bigl( C^{\circlearrowleft ~-1}_{\rho}(\widetilde H_{\rho+\dt-1 }(y+ \beta), \ldots,\widetilde H_{ \dt }(y+ \beta))^\top\Bigr)_i  - \Phi_{\dt+i-1}(y-\beta) \right\}_i ~~~~ \mbox{for $1\leq i\leq \rho$} 
\\ \\
&{\mathcal B}^{\beta,y-\beta}_{-\dt}  (i-\rho-1) 
 =\widetilde H_{\rho-i+\dt }(y+ \beta)
  ~~~~ \mbox{for $\rho+1\leq i\leq \rho+\dt$}  \end{aligned}
 \label{5B}\ee
%
Also, by Proposition \ref{PAB} and (\ref{PhiPsi-}),
 \be\begin{aligned} \bullet~~  {\mathcal B}^{\beta,x-\beta}_{\rho+\dt}  (i\!-\!\rho\!-\!1)  
& =  \sum^{\rho-i }_{\al=0} c_{\rho-i-\al} \Psi_{\al-\rho-\dt } (x-\beta) 
 =-   \sum^{\rho-i }_{\al=0} c_{\rho-i-\al} \Phi_{\al-\rho-\dt } (x- \beta) 
\\ &= 
  - 
   \{C^{\circlearrowleft}_{\rho} \left(\Phi_{-\dt-\rho}(x\!-\!\beta), \ldots, \Phi_{-\dt-1}(x\!-\!\beta)\right)^{\top}\}_i
 ~~~ \mbox{for $ 1\leq i\leq \rho $} 
\\
 {\mathcal B}^{\beta,x-\beta}_{\rho+\dt}  (i\!-\!\rho\!-\!1)  
& =0~~~ \mbox{for $\rho+1\leq i\leq \rho+\dt$}. \end{aligned}\ee
 Remember the expressions $g_{y } (\kappa)$  and $ h_{y } {(\lb)}$ from (\ref{gh}). For future use, we also define a new expression $\widehat g_{y } (\kappa)$,
\be\begin{aligned}
   \widehat g_{y } (\lb+\dt) :& =g_{y } (\lb+\dt)+f_{\lb}({\bf x-\beta},{y-\beta })=\Phi_{\lambda+\dt}(y- \beta)+ f_{\lb}({\bf x}-\beta,{y-\beta })
    \\
   h_{x } {(\lb\!-\!\rho\!-\!\dt)}& = \widetilde H_{ \rho+\dt -\lb-1}(\beta-x),
\end{aligned}\label{ghath}\ee
with $f_{\lb}({\bf x},{y })$ defined such that for $1\leq i\leq n$,
 \be\begin{aligned}
 \left\la h_{x_i}(\lb-n) ,  f_{\lb}({\bf x},{y  })\right\ra_{\lb \geq 0}  &=\sum_{\lb=0}^{n-1}
\widetilde H_{ n -\lb-1}( -x_i)  f_{\lb}({\bf x},{y  })
  = -
\BH^{ \rho+2\dt } (2(y -x_i) ).
 \end{aligned}\label{f}\ee
  This linear system of $n$ equations in the $n$ expressions $f_{\lb}$ will be discussed later on in (\ref{lin1}).

  \proof Remember the joint density of the $n:=\rho+\dt$ particles $y_i$ at level $-\dt$ and $x_i$ at level $\rho+\dt$ is given by the following determinant
 $$\begin{aligned}
\frac 1{2^{2n}}& ~p^{\mbox{\tiny\rm TAC}}(-\dt, {\bf y}; \rho+\dt, {\bf x})\\
:=&\det\left(\begin{array}{c|c}
\left(\tfrac 12 \mathbb{K}^{\textrm{\tiny TAC}}_{\beta,\rho}   (-\dt,y_i ; -\dt, y_j)\right)_{1\leq i,j\leq n}
& \left(\tfrac 12\BK^{\textrm{{\tiny TAC}}}_{\beta,\rho} 
  (-\dt, y_i ; n, x_j) \right)_{1\leq i,j\leq n} \\ \\
  \hline \\
   \left(\tfrac 12\BK^{\textrm{\tiny TAC}}_{\beta,\rho} (n, x_i ; -\dt,  y_j)\right) _{1\leq i,j\leq n}
& \left(\tfrac 12\BK^{\textrm{{\tiny TAC}}}_{\beta,\rho}     
        ( n, x_i ; n,  x_j)\right)_{1\leq i,j\leq n }
      \end{array}\right)  \end{aligned} 
$$
So, we need to compute the kernel $\BK^{\textrm{{\tiny TAC}}}_{\beta,\rho}$ evaluated at different points, appearing in this matrix. Recall from (\ref{2min}) and (\ref{8.6'}) that
   \be\begin{aligned}
&                 \frac{1}{2} \  \BK^{\textrm{{\tiny TAC}}}_{\beta,\rho} (u_1, y_1 ; u_2,  y_2) 
\\
& 
                  = -\Id_{u_1>u_2}   \BH^{u_1-u_2} (2(y_2-y_1)) + \Id_{u_1\geq 1}\sum_{\al = 0} ^{u_1-1}g_{y_2}(\al-u_2) h_{y_1}(\al-u_1)
\\
& +\sum_{\lambda=0}^{\max(\rho-1,\rho-1-u_2)}  \bigl((\Id - {\cal K}^\beta ( \lambda-\rho ,\kappa-\rho ))^{-1}  {\cal A}^{\beta,y_1-\beta}_{u_1 }\bigr)(\lambda-\rho)  {\cal B}^{\beta,y_2-\beta}_{u_2 }(\lambda-\rho) 
  _{_{ }}.
\end{aligned}\label{symm}\ee
It follows that, using (\ref{f}) in equality $\stackrel{*}{= }$ below and $n=\rho+\dt$,
 \be \begin{aligned}
  \frac 12& \ \mathbb{K}^{\textrm{{\tiny TAC}}}_{\beta,\rho}   (-\dt,y_i ; -\dt, y_j)
\\
& =  \sum^{n-1 }_{\lambda=0} \left(\Bigl((\Id - {\mathcal K}^{(\beta)} (\lambda-\rho,\kappa-\rho)\Bigr)^{-1}  {\mathcal A}^{\beta,y_i-\beta}_{-\dt} (\kappa-\rho)\right) {\mathcal B}^{\beta,y_j-\beta}_{-\dt} (\lambda-\rho)
 =:\left\la   v ^-_{y_i},w^- _{y_j}   \right\ra,
\end{aligned}\label{K1}\ee
{  \be  \hspace*{-2cm} \begin{aligned}
                 \frac{1}{2}    \ \BK^{\textrm{{\tiny TAC}}}_{\beta,\rho}& (n, x_i ; -\dt,  y_j) 
                   =    \sum_{\lb = 0}^{n-1} h_{x_i}(\lb-n) g_{y_j}(\lb+\dt) 
-    \BH^{\rho+2\dt} (2(y_j-x_i)) \\
&~~ +\sum_{\lambda=0}^{ n-1}  \bigl((\Id - {\cal K}^\beta ( \lambda-\rho ,\kappa-\rho ))^{-1}  {\cal A}^{\beta,x_i-\beta}_{n }\bigr)(\lambda-\rho)  {\cal B}^{\beta,y_j-\beta}_{-\dt }(\lambda-\rho) 
  \\
  &\stackrel{*}{= }  \sum_{\lb = 0}^{n-1} h_{x_i}(\lb-n) \Bigl(  g_{y_j}(\lb+\dt)
+  f_{\lb}({\bf x},{y_j -\beta})\Bigr)
\\
&~~ +\sum_{\lambda=0}^{n-1}  \bigl((\Id - {\cal K}^\beta ( \lambda-\rho ,\kappa-\rho ))^{-1}  {\cal A}^{\beta,x_i-\beta}_{n }\bigr)(\lambda-\rho)  {\cal B}^{\beta,y_j-\beta}_{-\dt }(\lambda-\rho)
 \\&=:\left\la h_{x_i} ,  \widehat g^{~-}_{y_j}  \right\ra
 +\left\la   v ^+_{x_i},w^- _{y_j}   \right\ra,
 \end{aligned}\label{K2}\ee}
  and
  \footnote{In both expressions below the sum $ \sum_{\lambda=0}^{ \rho+\dt-1 }$ could be replaced by $ \sum_{\lambda=0}^{ \rho-1 }$, since the extra terms are $=0$.}
\be\begin{aligned}
                  \frac{1}{2}  &  \BK^{\textrm{{\tiny TAC}}}_{\beta,\rho} 
  (-\dt, y_i ; n, x_j)  
\\
&  = \sum_{\lambda=0}^{ n-1 }  \bigl((\Id - {\cal K}^\beta ( \lambda-\rho ,\kappa-\rho ))^{-1}  {\cal A}^{\beta,y_i-\beta}_{-\dt }\bigr)(\lambda-\rho)  {\cal B}^{\beta,x_j-\beta}_{n }(\lambda-\rho) 
 = :\left\la   v ^-_{y_i},w^+ _{x_j}   \right\ra
,\end{aligned}\label{K3}\ee
\be\hspace*{-2cm}\begin{aligned}
                  \frac{1}{2} &\ \BK^{\textrm{{\tiny TAC}}}_{\beta,\rho}     
        ( n, x_i ; n,  x_j) 
\\
 & 
                   =  \sum_{\lb = 0}^{n-1} h_{x_i}(\lb-n) g_{x_j}(\lb-n) 
\\
&~~ +\sum_{\lambda=0}^{ n-1 }  \bigl((\Id - {\cal K}^\beta ( \lambda-\rho ,\kappa-\rho ))^{-1}  {\cal A}^{\beta,x_i-\beta}_{n }\bigr)(\lambda-\rho)  {\cal B}^{\beta,x_j-\beta}_{n }(\lambda-\rho) 
  \\
&=:\left\la h_{x_i} ,   g^+_{x_j} \right\ra
+\left\la   v ^+_{x_i},w^+ _{x_j}   \right\ra.
\end{aligned}\label{K4}\ee
Then combining (\ref{K1}), (\ref{K2}), (\ref{K3}) and (\ref{K4}), we compute, using Corollary \ref{CA1'''} in the third equality and using row operations in the second identity: 
$$
\begin{aligned} 
&\frac 1{2^{2n}}~p^{\mbox{\tiny\rm TAC}}(-\dt,{\bf y}; n, {\bf x})\\
&=\det\!\left(\begin{array}{c|c}\!
\left(\left\la   v ^-_{y_i},w^- _{y_j}   \right\ra\right)_{1\leq i,j\leq n}
& \left(\left\la   v ^-_{y_i},w^+ _{x_j}   \right\ra\right)_{1\leq i,j\leq n} \\ \\
  \hline \\
   \left(\left\la h_{x_i} , \widehat g^{~-}_{y_j} \right\ra
\! +\! \left\la   v ^+_{x_i},w^- _{y_j}   \right\ra\right) _{1\leq i,j\leq n}
& \left(\left\la h_{x_i} ,   g^+_{x_j} \right\ra
\! +\! \left\la   v ^+_{x_i},w^+ _{x_j}   \right\ra\right)_{1\leq i,j\leq n}
\end{array}\right)
\\  \\ 
&
=\det\left(\begin{array}{c|c}
\left(\left\la   v ^-_{y_i},w^- _{y_j}   \right\ra\right)_{1\leq i,j\leq n}
& \left(\left\la   v ^-_{y_i},w^+ _{x_j}   \right\ra\right)_{1\leq i,j\leq n} \\ \\
  \hline \\
   \left(\left\la h_{x_i} , \widehat g^{~-}_{y_j} \right\ra
 \right) _{1\leq i,j\leq n}
& \left(\left\la h_{x_i} ,   g^+_{x_j} \right\ra
 \right)_{1\leq i,j\leq n}
\end{array}\right)
\end{aligned}$$

$$=\det((v^-_{y_i})_{1\leq k,i \leq n}) \det((h_{x_i})_{1\leq k,j \leq n}) \\
   \det \left(\begin{array}{cc} 
(w^- _{y_j})_{1\leq j,k\leq n}&(w^+_{x_j})_{1\leq j,k\leq n}\\
(\widehat g^{~-}_{y_j})_{1\leq j,k\leq n}&(g^+_{x_j})_{1\leq j,k\leq n}
 \end{array}\right)
$$
\be\begin{aligned}
&=\det  \left(\Id-{\mathcal K}^{\beta}\bigr|_{[-\rho,-1]} \right)^{-1}
 \det  \left(  {\cal A}^{\beta,y_i-\beta}_{-\dt } (j \!-\!\rho\!-\!1)   \right)_{1\leq i,j\leq n}
\\&\hspace*{5cm}\times \det  \left(h_{x_i}(j\!-\!\rho\!-\!\dt\!-\!1)\right)_{1\leq i,j\leq n}~\det M_{yx}, 
\end{aligned}\label{p1}
\ee
with\footnote{Remember that $\widehat g_{y}$ also depends on $x$. Also we have that $M_y$ depends on $y$ and on ${\bf x}=(x_1,\ldots,x_n)$, while $M_x$ merely depends on $x$.}
\be\begin{aligned}
M_{yx}&:= \left( M_{y_1} ~\ldots~M_{y_{n}}  ~M_{x_1} ~\ldots~M_{x_{\rho+\dt}} \right)\\
&= \left( \!\! \left(\begin{array}{c}{\cal B}^{\beta,y_j-\beta}_{-\dt }(\lambda\!-\!\rho)  \Bigr|_{0\leq  \lb\leq n-1}\\ 
  \widehat g_{y_j}(\lb+\dt)\Bigr|_{0\leq  \lb\leq n-1}\end{array}\right)_{\!\!1\leq j\leq n}   \!\!\!
\left(\begin{array}{c  }{\cal B}^{\beta,x_j-\beta}_{n }(\lambda\!-\!\rho) \Bigr|_{0\leq  \lb\leq n-1}  \\   g_{x_j}(\lb\!-\! n)\Bigr|_{0\leq  \lb\leq n-1}\end{array} \right)_{\!\!\!1\leq j\leq n} \right) 
\label{Myx}\end{aligned} \ee
%
%
  %
 Setting
 $$
{\mathcal C}:= \left(\begin{array}{c | c | c  }
 C^{\circlearrowleft}_{\rho} & O & O\\
 \hline
   O &{\mathbb I}_{\rho+\dt}
   &O\\
   \hline
   O&O& {\mathbb I}_{\dt} 
 \end{array}\right),
 $$
The matrix $M_{yx}=\left( M_{y_1} ~\ldots~M_{y_{n}}  ~M_{x_1} ~\ldots~M_{x_{n}} \right)$ consists of $n=\rho+\dt$ columns depending on the $y_i$'s and on the $x_i$'s. So $M_{y}$ and $M_{x}$ are typical columns. Performing row operations does not change the determinant. Below we make a number of row operations on the $M_{y}$ and simultaneously on the $M_{x}$. The equality ${\simeq}$ means that the determinant  $M_{yx}$ remains unchanged upon using the new columns $M_{y}$ and $M_{x}$ below. The first equality uses (\ref{5B}) and (\ref{ghath});  equality $\stackrel{*}{\simeq}$ is obtained by adding the entries of the third group (of rows) to the ones of the first group and equality $\stackrel{**}{\simeq}$ by adding linear combinations of the entries of the second group to the first group. In $\stackrel{***}{\simeq}$ one is using Corollary \ref{C2.3} and in $\stackrel{****}{\simeq}$ one replaces the polynomials $\tilde H_{  \dt -i }(y+ \beta)$ by the polynomials $P_{  \dt -i }(y-\beta)$.
 Thus we have
 $$ \begin{aligned}
 M_y&={\mathcal C}\left(\begin{array}{l} 
    \Bigl( C^{\circlearrowleft ~-1}_{\rho}
 (\widetilde H_{n-1 }(y\!+\! \beta), \ldots,\widetilde H_{ \dt }(y\!+\! \beta))^\top\Bigr)_i  - \Phi_{\dt+i-1}(y\!-\!\beta)   
   \Bigr|_{1\leq i\leq \rho}
   \\  \\
  \tilde H_{  \dt -i }(y+ \beta) \Bigr|_{1\leq i\leq \dt}
\\ \\
\Phi_{ \dt+i-1}(y-\beta)+f_{i-1}({\bf x}-\beta,{y-\beta })\Big|_{1\leq i\leq \rho }
\\ \\
\Phi_{ \dt+i-1}(y-\beta)+f_{i-1}({\bf x}-\beta,{y -\beta})\Big|_{\rho+1\leq i\leq n } \end{array}\right)
  \end{aligned}$$
  $$\begin{aligned} 
  &\stackrel{(*)}{\simeq} {\mathcal C}\left(\begin{array}{l} 
    \Bigl( C^{\circlearrowleft ~-1}_{\rho}
 (\widetilde H_{\rho+\dt-1 }(y\!+\! \beta), \ldots,\widetilde H_{ \dt }(y\!+\! \beta))^\top\Bigr)_i  +f_{i-1}({\bf x}\!-\!\beta,{y\!-\!\beta })
   \Bigr|_{1\leq i\leq \rho}
   \\  \\
  \tilde H_{  \dt -i }(y+ \beta) \Bigr|_{1\leq i\leq \dt}
\\  
\ast 
\\ 
\ast
 \end{array}\right)
\\
  &\stackrel{(**)}{\simeq} 
  {\mathcal C}\left(\!\!\begin{array}{l} 
    \Bigl( C^{\circlearrowleft ~-1}_{\rho+\dt}
 (\widetilde H_{\rho+\dt-1 }(y\!+\! \beta), \ldots,\widetilde H_{ 0 }(y\!+\! \beta))^\top\Bigr)_{\dt+i} \! \!+f_{i-1}({\bf x}\!-\!\beta,{y \!-\!\beta})
   \Bigr|_{1\leq i\leq \rho}
   \\  \\
  \tilde H_{  \dt -i }(y+ \beta) \Bigr|_{1\leq i\leq \dt}
\\  
\ast
\\  
\ast
 \end{array}\!\!\!\right)
\end{aligned}$$
$$\begin{aligned}
 &\stackrel{(***)}{\simeq}
 {\mathcal C}\left(\begin{array}{l} 
     %
 P_{\dt+i-1}(y-\beta)+f_{i-1}({\bf x}-\beta,{y -\beta})
   \Bigr|_{1\leq i\leq \rho}
   \\  
  \tilde H_{  \dt -i }(y+ \beta) \Bigr|_{1\leq i\leq \dt}
\\  
\ast
\\  
\ast
 \end{array}\right)
 \end{aligned}$$
 $$\begin{aligned}
 &\stackrel{(****)}{\simeq}\mbox{Constant}\times
 {\mathcal C}\left(\begin{array}{l} 
     %
 P_{\dt+i-1}(y-\beta)+f_{i-1}({\bf x}-\beta,{y -\beta})
   \Bigr|_{1\leq i\leq \rho}
   \\  
  P_{  \dt -i }(y- \beta) \Bigr|_{1\leq i\leq \dt}
\\  
\ast
\\  
\ast
 \end{array}\right),
  \end{aligned}$$
%
%
with a constant of absolute value $=1$, and also we find
 $$ \begin{aligned}
 M_x&={\mathcal C}\left(\begin{array}{l} 
  - 
       \Phi_{\lb-\rho-\dt }(x-\beta)   \Bigr|_{0\leq \lb\leq \rho-1}\\ \\ 
  O \Bigr|_{1\leq i\leq \dt}
\\ \\
\Phi_{\lb-\rho-\dt}(x-\beta) \Big|_{0\leq \lb\leq \rho-1}
\\ \\
\Phi_{\lb-\rho-\dt}(x-\beta) \Big|_{\rho\leq \lb\leq n-1} \end{array}\right)
 \simeq   
  {\mathcal C}\left(\begin{array}{l} 
  O \Bigr|_{1\leq i\leq \dt}   \\ \\ 
  O \Bigr|_{1\leq i\leq \dt}
\\ \\
\Phi_{\lb-n}(x-\beta) \Big|_{0\leq \lb\leq \rho-1}
\\ \\
\Phi_{\lb-n}(x-\beta) \Big|_{\rho\leq \lb\leq n-1} \end{array}\right)
 .\end{aligned}$$
 Combining $M_y$ and $M_x$ in one matrix $M_{yx}$, one finds
\be 
 \begin{aligned}
 \det& M_{yx} 
 \\=&
 \mbox{\footnotesize{
$   
 \det{\mathcal C}\det \left( \begin{array}{l} 
  P_{ i }(y_j- \beta) \Bigr|_{0\leq i\leq \dt-1}
\\     %
 P_{\dt+i }(y_j\!-\!\beta)+f_{i }({\bf x}-\beta,{y_j\! -\!\beta})
    \Bigr|_{0\leq i\leq \rho-1}
   \\   \hline
   \\
\hspace*{2cm}\mbox{\huge$\ast$}
 \end{array} \vline
  \begin{array}{l} 
 \hspace*{2cm}\mbox{\large$O$}
\\ \\  \\ \hline  
\vspace*{.4cm}\Phi_{\lb-n }(x_j-\beta) \big|_{0\leq \lb\leq n-1}
 \end{array}  \right)  $
} 
}\\ =&(\det C^{\circlearrowleft}_{\rho})\det\left(\begin{array}{cc} 
     \Phi_{-\lb }(x_j-\beta)  
   \end{array}
  \right) _{1\leq \lb,j\leq n} \\
    &
\hspace*{2.1cm} \det\left(  
  \begin{array}{l} 
  P_{ i }(y_j- \beta) \Bigr|_{0\leq i\leq \dt-1}
\\     %
 P_{\dt+i }(y_j\!-\!\beta)+f_{i }({\bf x}-\beta,{y_j\! -\!\beta})
    \Bigr|_{0\leq i\leq \rho-1}
  \end{array}\right)_{1\leq j\leq n}
 \end{aligned}\label{Mx} \ee
Combining the calculations (\ref{Mx}) and (\ref{p1}), using $\det C^{\circlearrowleft}_{\rho}=1$ (see (\ref{Cinv})), Proposition \ref{PAB}, (\ref{Phi-Psi}) and (\ref{HermMatrix}) and taking into account the GUE-distribution (\ref{GUE}) and the volume (\ref{Bary1}) of the cones $\mbox{Vol}({\mathcal C}_{\bf x})$ and $\mbox{Vol}({\mathcal C}_{\bf y})$, one finds 
 \be\begin{aligned}
&  p^{\mbox{\tiny\rm TAC}}(-\dt,{\bf y}; n, {\bf x})
\\
&=\frac{2^{2n}}{(-1)^{n-1}}\det  \left(\Id-{\mathcal K}^{\beta} \right)^{-1}
 \det  \left(  {\cal A}^{\beta,y_i-\beta}_{-\dt } (j \!-\!\rho\!-\!1)   \right)_{1\leq i,j\leq n}
   \det\left(\begin{array}{cc} 
     \Phi_{-i }(x_j-\beta)  
   \end{array}
  \right) _{1\leq i,j\leq n}
 \end{aligned}
\ee
$$\times\det  \left(h_{x_i}(j\!-\!n\!-\!1)\right)_{1\leq i,j\leq n}
  \det\left(  
  \begin{array}{l} 
  P_{ i }(y_j- \beta) \Bigr|_{0\leq i\leq \dt-1}
\\     %
 P_{\dt+i }(y_j\!-\!\beta)+f_{i }({\bf x}-\beta,{y_j\! -\!\beta})
    \Bigr|_{0\leq i\leq \rho-1}
  \end{array}\right)_{1\leq j\leq n}
$$
$$\begin{aligned}
=&\frac{2^{2n}}{(-1)^{n-1}}\det  \left(\Id-{\mathcal K}^{\beta} \right)^{-1}\det\left(\frac{e^{-(y_i+\beta)^2}}{\sqrt{\pi}}
  \frac{H_{j-1 }(y_i+ \beta)}{2^{j  }}
\right)_{1\leq i,j \leq n}
\!\!\!\!\!\!\!\!\!\!\!\!\!\!\!%
 \\
 &\times\det\left(\frac{e^{-(x_i-\beta)^2}}{\sqrt{\pi}}
  \frac{H_{j-1   }(\beta-x_i )}{2^{j }}
\right)_{1\leq i,j\leq n}
\\
& \det \widetilde {\cal H}^{(\beta-x)}_{n} \det\left(  
  \begin{array}{l} 
  P_{ i }(y_j- \beta) \Bigr|_{0\leq i\leq \dt-1}
\\     %
 P_{\dt+i }(y_j\!-\!\beta)+f_{i }({\bf x}\!-\!\beta,{y_j\! -\!\beta})
    \Bigr|_{0\leq i\leq \rho-1}
  \end{array}\right)_{1\leq j\leq n}
%
\\ =&  \frac{ 2^{n(n-1)}}{ \pi^n\det  \left(\Id-{\mathcal K}^{\beta} \right)^{ }}\bigl(\Dt_{n}(x)\prod_{i=1}^{n}e^{-(x_i-\beta)^2}\bigr)\bigl(\Dt_{n}(y)\prod_{i=1}^{n}e^{-(y_i+\beta)^2}\bigr)   \Gamma_{\bf xy}
 \\=&
\frac{\prod_1^{n-1}( j!)^2}{\det  \left(\Id-{\mathcal K}^{\beta} \right)} ~\frac{\rho_n({\bf x}-\beta) ~\rho_n({\bf y}+\beta)}{\Dt_n( {\bf x})\Dt_n( {\bf y})}~\Gamma_{\bf xy}
\\=&
 {} ~\frac{\rho_n({\bf x}-\beta) ~\rho_n({\bf y}+\beta)}{\mbox{Vol}({\mathcal C}_{\bf x})\mbox{Vol}({\mathcal C}_{\bf y})      \det  \left(\Id-{\mathcal K}^{\beta} \right)}~\Gamma_{\bf xy}
 , \end{aligned}
  $$
 where $\Gamma_{\bf xy}$ is defined as: 
 \be\begin{aligned}
\Gamma_{\bf xy}:=& \frac { \det \widetilde {\cal H}^{(\beta-x)}_{n}
}{(-1)^{n-1}2^{n(n-1)}}
 \det\left(  
  \begin{array}{l} 
  P_{ i }(y_j- \beta) \Bigr|_{0\leq i\leq \dt-1}
\\     %
 P_{\dt+i }(y_j\!-\!\beta)+f_{i }({\bf x}\!-\!\beta,{y_j\! -\!\beta})
    \Bigr|_{0\leq i\leq \rho-1}
  \end{array}\right)_{1\leq j\leq n}.
  \end{aligned}\label{Gamma1}\ee
   This shows Theorem \ref{Th:JProb}, except for the expression (\ref{vol1}) for $\Gamma_{\bf xy}$. To do so, in order to find the expressions $f_k$, recall (\ref{f}) is a linear system of $n=\rho+ \dt$ equations (parametrized by $x_1,\dots, x_{n}$) in $n$ variables $ f_{0}({\bf x},{y }),\ldots , f_{n-1}({\bf x},{y })$; its solution is thus given by Cramer's rule; so,  
   for $1\leq k \leq \rho $, one has \underline{\em a first expression} for the $f_k$'s,
 \be \begin{aligned}
  & f_{k-1}({\bf x },{y  })=  -
   \det\left(\widetilde H_{ n-j }(  -x_i)\right)_{1\leq i,j \leq n}  ^{-1}
    \hspace*{1.5cm}
    \begin{array}{c}k
 \\
 \downarrow \end{array} \\
  & \times\det\left( \widetilde H_{ n -1}(  -x_i),~
   \ldots,~\widetilde H_{ n  -k+1}(  -x_i),  
   {\mathbb H}^{  \rho+2\dt }(2(y\!-\!x_i)),
   \widetilde H_{ n  -k-1}(  -x_i),~\ldots,\right.
   \\
   &\hspace*{10cm}\left. \widetilde H_{ 0}(  -x_i)\right)_{1\leq i\leq n}.
\end{aligned}
\label{lin1}\ee
Linear system (\ref{f}) can also be written as a matrix equation, using notation (\ref{HermMatrix}):
$$
\widetilde {\cal H}^{(-x)}_{n}
  \left( f_{k-1}({\bf x},y_j)  \right)_{1\leq k,j\leq n}=
  \left(-\BH^{ \rho+2\dt } (2(y_j -x_i) )\right)_{1\leq i,j\leq n}$$
 and so one finds, upon setting $f_i=0$ for $i\leq -1$, \underline{\em a second expression}
\be
\left(f_{k-\dt-1}({\bf x},y_j)\right)_{1\leq k,j\leq n}
=
\left( \begin{array}{c} O  \\ \\ 
\hline \\    
{\mathcal P}^{\mbox{\tiny rows}}_{[1,\rho]}\left[(\widetilde {\cal H}^{(-x)}_{n}) ^{-1}
  \left(-\BH^{ \rho+2\dt } (2(y_j -x_i) )\right)_{1\leq i,j\leq n}\right] \end{array}\right) \!\!\! \begin{array}{c}\updownarrow \dt \\ \\ \\    \updownarrow\rho\\ \end{array}\label{lin2}\ee
where ${\mathcal P}^{\mbox{\tiny rows}}_{[1,\rho]}[.]$ stands for the projection of a matrix onto the first $\rho$ rows.

So, using this second form of the solution (\ref{lin2}) and using in $\stackrel{*}{=}$ below equation (\ref{HP}) of Corollary \ref{C2.2}, one finds for $\Gamma_{\bf xy}$, defined in (\ref{Gamma1}),
$$\begin{aligned}
&(-1)^{n-1}2^{n(n-1)}\Gamma_{\bf xy} 
\\
 =&\det\widetilde {\cal H}^{(\beta-x)}_{n}
   \det\left[  \left(P_{ i-1 }(y_j- \beta) \right) _{1\leq i,j\leq n }~+~\left ( f_{i-\dt-1}({\bf x}-\beta,y_j-\beta)\right)_{1\leq i,j\leq n}\right]
 \\
 =&  \det\left[ \widetilde {\cal H}^{(\beta-x)}_{n}  \left(P_{ i-1 }(y_j- \beta) \right) _{1\leq i,j\leq n }~+~\widetilde {\cal H}^{(\beta-x)} \left ( f_{i-\dt-1}({\bf x}-\beta,y_j-\beta)\right)_{1\leq i,j\leq n}\right]
 \end{aligned}$$
 \be\begin{aligned} 
 \stackrel{*}{=}&  \det\left[   \left(\frac {(2(y_j- x_i))^{n-1}  }{ (n-1)!} \right) _{1\leq i,j\leq n }~\right.
 \\
 &    \left.+ \left(O   \Bigr |  {\mathcal P}^{\mbox{\tiny (columns)}}_{[\dt+1,n]} \widetilde {\mathcal H}_n^{(\beta-x)}\right) \left( \begin{array}{c} O  \\ \\ 
\hline \\    
{\mathcal P}^{ \mbox{\tiny (rows)} }_{[1,\rho]}\left[(\widetilde {\cal H}^{(\beta-x)}_{n}) ^{-1}
  \left(-\BH^{ \rho+2\dt } (  2(y_j -x_i)  )\right)_{1\leq i,j\leq n}\right] \end{array}\right) 
  \right]_{ }
\\
 \stackrel{*}{=}&  \det\left[    \frac {(2(y_j- x_i))^{n-1}  }{ (n-1)!}   ~
 -\sum_{\gamma=1}^{n} \left(\sum_{\al=1}^\rho \widetilde H_{\rho-\al}(\beta -x_i)({\mathcal K}^{(\beta -x)}_n)_{\al,\gamma} \right){\mathbb H}^{n+\dt}(2(y_j-x_\gamma))\right]
 \\
 =& 2^{n(n-1)} \det\left[    \frac { (y_j- x_i) ^{n-1}  }{ (n-1)!}   ~
 -\sum_{k=1}^{n} {\mathcal L}_{ik} 
 {\mathbb H}^{n+\dt}( y_j-x_k )\right]_{1\leq i,j\leq n}\end{aligned}\label{Gamma2}\ee 
using the matrix identity (\ref{HP}), and upon defining the matrices 
$$
{\mathcal K}^{( -x)}_n:= (\widetilde {\cal H}^{( -x)}_{n}) ^{-1}~~
\mbox{and}~~{\cal L}_{ik}:=2^\dt \sum_{\al=1}^\rho  
\widetilde H_{\rho-\al}(\beta-x_i)  
({\mathcal K}^{( \beta-x)}_n)_{\al,k}.$$
Indeed the $(i,j)$th entry of the second matrix (under the determinant) in $\stackrel{*}{=}$ (formula (\ref{Gamma2}) above) can be written as
 $$\begin{aligned}
\sum_{\al=1}^\rho &
\widetilde H_{\rho-\al}(\beta-x_i)
   \sum_{k=1}^{n}  
({\mathcal K}^{( \beta-x)}_n)_{\al,k} {\mathbb H}^{\rho+2\dt}(2(y_j-x_k))
 \\&=\frac{1}{2^{\dt}}\sum_{k=1}^{n} {\mathcal L}_{ik} 
 {\mathbb H}^{\rho+2\dt}(2(y_j-x_k))
  = {2^{n-1}}\sum_{k=1}^{n} {\mathcal L}_{ik} 
 {\mathbb H}^{\rho+2\dt}( y_j-x_k )
 .\end{aligned}$$
One then checks that the matrix ${\mathcal L}$ has the form stated in (\ref{vol2}). This ends the proof of Theorem \ref{Th:JProb}.\qed

\bigbreak
\noindent {\em Remark:} Notice that ${\mathcal L}=\Id$ for $\dt=0$ and for $\dt\geq 1$, one also checks that
 $$\begin{aligned}
 \sum_{k=1}^{n} {\mathcal L}_{ik} 
\frac{( y\!-\!x_k )^{ n+\dt-1 }}{( n+\dt-1)!}-\frac{(  y\!-\!x_i )^{ n-1 }}{( n-1)!}
- c_{\rho,\dt} (y\!-\!x_i)^{\dt-1}R^{(\dt)}(x_i\!-\!\beta)=\sum_0^{\dt-2} S_i(x)y^i ,
  \end{aligned} $$
  where $S_i(x)$ are polynomials of total degree $n-i-1$ in the $x_i$ with coefficients which are symmetric functions of the $x_j$'s with $j\neq i$; in particular,
 $$
 c_{\rho,\dt}=\frac {(-1)^n}{n!} \mbox{   for } \dt=1
 .$$
 One thus finds
   \be \begin{aligned}
&\Gamma_{\bf xy} = (-1)^{n-1}\det
\\
& \left[ \sum_{k=1}^{n} {\mathcal L}_{ik} 
\frac{( y_j-x_k )^{ n+\dt-1 }}{( n+\dt-1)!}\Id_{y_j<x_k}- c_{\rho,\dt} (y_j\!-\!x_i)^{\dt-1}R^{(\dt)}(x_i\!-\!\beta)-\sum_0^{\dt-2} S_i(x)y_j^i \right]_{1\leq i,j\leq n}.
 \label{remark}\end{aligned} \ee
In particular for $\dt=0$, one has:
\be
\Gamma_{\bf xy} =
\det\left[\frac{(x_i-y_j)^{n-1}}{(n-1)!} \Id _{y_j<x_i}\right]_{1\leq i,j\leq n}.
\label{86'}\ee
Another way of obtaining this expression is to start directly from the final expression in (\ref{Gamma2}), namely
 for $n=\rho$ or $\dt=0$, one has
$$\begin{aligned}
& (-1)^{n-1}2^{n(n-1)} \Gamma_{\bf xy} 
\\ &=  \det\left[   \left(\frac {(2(y_j\!-\! x_i))^{n-1}  }{ (n-1)!}\right) _{1\leq i,j\leq n }+
 \widetilde {\cal H}^{(\beta-x)}_{n} ( \widetilde {\cal H}^{(\beta-x)}_{n} )^{-1}
 \left(-\BH^{ \rho  } (  2(y_j\! -\! x_i)  )\right)_{1\leq i,j\leq n} 
 \right]
 \\
 &=\det\left[   \left(\frac {(2(y_j- x_i))^{n-1}  }{ (n-1)!}\right) _{1\leq i,j\leq n }-
 \left(\frac{(2(y_j -x_i))^{n -1}}{(n -1)!}\Id _{y_j\geq x_i}\right)_{1\leq i,j\leq n} 
 \right]
 \\
 &=(-1)^{n-1}2^{n(n-1)} \det\left[   \left(\frac {( x_i-y_j ) ^{n-1}  }{ (n-1)!}\Id _{y_j< x_i}\right) _{1\leq i,j\leq n } \right].
\end{aligned}$$
Finally, applying formula (\ref{remark}) for $\dt=1$ yields
  \be\begin{aligned}
\Gamma_{\bf xy}&=
-\det\left[
\frac {1}{n!} \left( \sum_{k=1}^{n} {\mathcal L}_{ik} 
 {( x_k-y_j )^{ n  }} \Id_{y_j<x_k}-  { } R'(x_i-\beta)\right)
 \right]_{1\leq i,j\leq n},
 \end{aligned}\label{86''}\ee
with the polynomial $R$ as in (\ref{R}).%

\section{The Volume of the double Cone}



Keeping in mind notations (\ref{spectra2}), (\ref{spectra3}), the symbol $\curlyeqprec$ in (\ref{xy-interlace}) and footnote 4, we give in this section formulas for the volume of the double cone (\ref{cone}); as illustrated in Figure 3, this cone is expressed as follows by Theorem \ref{Th1.1}, for given $ {\bf x},{\bf y}\in \BR^{n}$, with $n\geq \rho$:
   \be
   {\mathcal C}^{(n)}_{{\bf x},{\bf y}} =\left\{\begin{array}{l}
   {\bf y}  \succ {\bf y}^{(n-1)} \succ\ldots \succ   {\bf y}^{(\rho)}
   ~\mbox{and}~  {\bf x} \succ {\bf x}^{(n-1)}\succ \ldots \succ {\bf x}^{(\rho)} 
   \\
    %
   {\bf y}^{(\rho)}    
   \curlyeqprec   {\bf z}^{(1)} 
\curlyeqprec\ldots \curlyeqprec 
  {\bf z}^{(\rho-1)}
\curlyeqprec 
 {\bf x}^{(\rho)} 
 \\ 
 \mbox{for any $ {\bf x}^{(i)},~{\bf y}^{(i)}\in \BR^i$  and 
$ {\bf z}^{(j)}\in \BR^\rho$}
 \end{array} \right\} \subset \BR^{n(n-1)} .\label{cone7}\ee
Also remember the definition (\ref{Gamma1}) and (\ref{Gamma2}) of $\Gamma^{(n)}_{\bf xy}$ in section 6: it is given for  {$n=\rho+\dt\geq\rho$} by two alternative formulas:
\be\begin{aligned}
\Gamma^{(n)}_{{\bf x} {\bf y}} 
&=  
  \frac{\det(\widetilde {\cal H}^{(\beta -x)}_{n})}{(-1)^{n-1}2^{n(n-1)} } 
 \det\left(  
  \begin{array}{l} 
  P_{ i }(y_j- \beta) \Bigr|_{0\leq i\leq \dt-1}
\\     %
 P_{\dt+i }(y_j\!-\!\beta)+f_{i }({\bf x}\!-\!\beta,{y_j\! -\!\beta})
    \Bigr|_{0\leq i\leq \rho-1}
  \end{array}\right)_{1\leq j\leq n}
  \\
  &=  (-1)^{n-1} \det\left[    \frac { (y_j- x_i) ^{n-1}  }{ (n-1)!}   ~
 -\sum_{k=1}^{n} {\mathcal L}_{ik} 
 {\mathbb H}^{n+\dt}( y_j-x_k )\right]_{1\leq i,j\leq n} ,
\end{aligned}\label{vol1'}\ee
 with the matrix $\widetilde {\cal H}^{(\beta -x)}_{n} $ as in (\ref{HermMatrix}), with the $P_i$'s as in (\ref{herm1}) and the $f_i$ as in (\ref{lin1}) and where ${\mathcal L}$ is the $n\times n$ matrix defined in (\ref{vol2}).

\begin{proposition}\label{volprop1}
The volume of the double cone $ {\mathcal C}^{(n)} _{{\bf x},{\bf y}}  $ is given for 
\newline\noindent \underline{(i) $n=\rho$} 
\be
\mbox{Vol}({\mathcal C}^{(\rho)}_{{\bf xy}})
=\det \left[
\frac {( x_i-y_j )^{n-1}} {(n-1)!}\Id_{x_i> y_j}   
  \right]_{1\leq i,j\leq \rho}
=\Gamma^{(\rho)}_{{\bf x} {\bf y}} \label{Va}\ee
\newline\noindent \underline{(ii) $n=\rho+\dt\geq\rho$}, 
$$
\mbox{Vol}({\mathcal C}^{(n)}_{{\bf xy}})=\Gamma^{(n)}_{\bf xy}
.$$
%
%
\end{proposition}

\bigbreak 
\begin{corollary}\label{volcor1} In particular, for $n=\rho+1$ (i.e., $\dt=1$), 
  $$\begin{aligned}
\mbox{Vol}({\mathcal C}^{(\rho+1)}_{\bf xy})&=-
\det\left[
\frac {1}{n!} \left( \sum_{k=1}^{n} {\mathcal L}_{ik}\bigr|_{\beta=0} 
 {( x_k-y_j  )^{ n  }} \Id_{x_k>y_j }-  { } R'(x_i)\right)
 \right]_{1\leq i,j\leq n}.
 \end{aligned}$$
\end{corollary}
 
In section 6 it was already shown that the two expressions  (\ref{vol1'}) for $\Gamma_{\bf xy}$ are equal. So it remains to show they equal the volume $\mbox{Vol}({\mathcal C}_{\bf xy})$. Let us begin by showing this fact for $n=\rho$; i.e., for $\dt=0$.

\medbreak

\noindent {\em Proof of Proposition \ref{volprop1} for $n=\rho$}:
 For the sake of the proof it seems more convenient to reverse the order; so, we set $y_1>\ldots>y_\rho$ and $x_1>\ldots>x_\rho$, all points in $\BZ$. 
 One has the following one-to-one correspondence:
\be\begin{aligned}
 \left\{\begin{array}{l}\mbox{up-right nonintersecting}\\
\mbox{ paths from the points  }\\ \mbox{$y_i$ to $x_i$, through the}\\ 
	\mbox{levels $u=0,\ldots, \rho$ }\end{array}\right\} 
= \left\{\begin{array}{l}\mbox{$\rho+1$ sets of $\rho$ dots, intertwining}\\
\mbox{the dots $y_i$'s at level $u=0$ 
and}\\
 \mbox{the dots $x_i$'s at level $u=\rho$}
 \end{array} \!\! \right\}
.\end{aligned}\label{path_dots}\ee
The precise correspondence is given by putting a dot to the left of each horizontal stretch belonging to the up-right paths and at each level $0\leq u
\leq \rho-1$. 
Thus the number of intertwining sets of $\rho$ points or, in other terms, the volume ${\mathbb V}_{\bf xy}$, equals the number of nonintersecting paths from the $y_i$'s to the $x_i$'s; the transitions from the level $u=\rho-1$ to the final level $u=\rho$ consist of vertical stretches. 
To count the number of such paths, we will first assign the weight $z_\ell$ to each right-step along the horizontal level $\ell$ and $1$ to each vertical step; then the total weight of all up-right path from $x$ to $y\in \BZ$ equals 
$$
h_r(z)= \sum_{\sum_j i_j =r} z_1^{i_1}z_2^{i_2}\ldots ~\mbox{, with $r=x-y$};
$$
its generating function is given by $\sum_{r\geq 0}h_r(z)\al^r=\prod_{i\geq 1}(1- \al z_i)^{-1}$. The Gessel-Viennot formula tells us that the total weight of all nonintersecting up-right paths from the $x_i$ to $y_i$ equals the following determinant which also defines the skew-Schur polynomials  
 for partitions $\lb$ and $\mu$ defined from the $x_i$ and $y_i$'s:
  \be\det\left(h_{x_i-y_j}(z)\right)={\bf s}_{\lb/\mu}(z)\mbox{   for $\lb_i =x_i+i-\rho$ and $\mu_i=y_i+i-\rho$}.
\label{GV}\ee
To actually count the number of paths above, assign the weight $z_i=1$ to horizontal stretches as well. Thus, setting $$z=1^{\rho}=(\underbrace{1,\ldots,1}_{\rho},0,0,\ldots), $$ one obtains by means of the generating function, that
%
$$\begin{aligned}
\sum_{r\geq 0}h_r(z)\al^r \Bigr|_{z=1^\rho}&=\prod_{i}e^{-\log (1-\al z_i)}
 \Bigr|_{z=1^\rho}=\left( \frac 1{1-\al}\right)^\rho =\sum_{r\geq 0} \left( \begin{array}{c} r+ \rho-1\\ \rho-1 \end{array} \right) \al^r
 .
\end{aligned}
$$
So we conclude that 
\be
h_r(z)=\left(\begin{array}{c} r+\rho-1  \\ \rho-1
\end{array}\right).\label{h}\ee
Combining (\ref{path_dots}), (\ref{GV}) and (\ref{h}), the volume equals
$$
{\mathbb V}_{xy}=\det \left[   \left(\begin{array}{c} x_i-y_j+\rho-1  \\ \rho-1
\end{array}\right)\Id_{x_i\geq y_j}\right]_{1\leq i,j\leq \rho}
.$$
Considering now the sets $y_1>\ldots>y_\rho$ and $x_1>\ldots>x_\rho$ of points in $\BR$ at levels $u=0$ and $u=\rho$, one finds  
$$
\mbox{Vol}({\mathcal C}^{(\rho)}_{\bf x y})=\lim_{t\to \infty}\frac{{\mathbb V}_{tx,ty}}{t^{\rho(\rho-1}} =\det \left[
\frac {(x_i-y_j)^{\rho-1}} {(\rho-1)!}\Id_{x_i> y_j}   
  \right]_{1\leq i,j\leq \rho}  =  
  \Gamma_{\bf xy},
    $$
establishing formula (\ref{Va}), after changing back the order of the $x_i$ and $y_i$'s. 
 Notice $\Gamma_{{\bf x}{\bf y}}$ is $\beta$-independent. 
This establishes Proposition \ref{volprop1} for $n=\rho$.\qed



   \noindent  The rest of the proof of formula (\ref{vol1'}) will be done by induction on $n$. The following shorthand notation will be used in the statement and the proof of the next Lemmas.
\be\begin{aligned}
 &  X_j := \mbox{ column } (\widetilde H_{n-j}(-x_i))_{1\leq i\leq n}
 \\
 &  {\mathbb H}_{y } :=\mbox{ column }
 \left( {\mathbb H}^{  \rho+2\dt }_{y,x_i}\right)_{1\leq i\leq n}:=\mbox{ column }
 \left(- {\mathbb H}^{  \rho+2\dt }( 2(y-x_i) )\right)_{1\leq i\leq n}
 \\
   &\tau_n:=\det(X_1,\ldots,X_n)=\det (\widetilde {\cal H}_{n} ^{( -x)})=c_n\Dt_n(x)\\
& {\mathcal V}_{_{X_k  \curvearrowright Y }}
\tau_n :=  \mbox{  replacing row $X_k$ by a row $Y$ in   }\tau_n=\det(X_1,\ldots,X_n)
 \\
 & (j,k) := \sum_{i=1}^\ell(k_i+j_i)+\ell \dt
\end{aligned}\label{vol3}\ee
%
 
 
 \begin{lemma}\label{vollemma6.3}
 
 Given integers $0\leq \rho, \dt $ with $n=\rho+\dt$, and $4n$ column-vectors $C^+_j,~C^-_j,~X_j,~Y_j\in \BR^n$, for $1\leq j\leq n$, with
 $$
 C^-_j=\left(\begin{array}{c}0\\ \vdots \\ 0 \\\hline \\ C^-_{\dt+1,j}\\
 \vdots \\ C^-_{n,j}\end{array}\right)
{ {  \left. \begin{array}{c}\\ \\ \\ \end{array}\right\} \dt } \atop \left.    {\begin{array}{c}\\ \\ \\ \\  \end{array}}  \right\}\rho}   
 $$
 with ($1\leq i\leq \rho$)
 $$
 C^-_{\dt+i,j}=\frac 1{\tau}{\mathcal V}_{_{X_i  \curvearrowright Y_j }}\tau
, \mbox{   with~~   }\left\{ \begin{aligned}\tau&=\det (X_1,\ldots,X_n)\\
{\mathcal V}_{_{X_i  \curvearrowright Y_j }}\tau&=
\det (X_1,\ldots,X_{i-1}, Y_j, X_i,\ldots,X_n)
\end{aligned}\right. .$$
 Then 
 \be\begin{aligned}
 \det  (C^+_1&+C^-_1,\ldots,C^+_n+C^-_n ) 
 \\&=
  \frac 1 \tau\sum_{\ell=0}^\rho\sum_{_{  
      \{\hat j_1<\ldots<\hat j_{n-\ell}\} \cup \{j_1<\ldots<j_\ell\} 
   = (1,\ldots,n) 
    }\atop {\{ \hat i_1=1,\ldots,\hat i_\dt=\dt<\hat i_{\dt+1}<\ldots   <\hat i_{n-\ell}\} \cup \{\dt+1\leq \dt+ i'_1<\ldots<\dt+i'_\ell\} 
   = (1,\ldots,n)}}  \!\!\! (-1)^{(i',j) }
   \\&\hspace*{1cm}\times\det( C^+_{\hat i_r ,\hat j_s})_{1\leq r,s\leq n-\ell } 
      \prod _{r=1}^\ell 
 {\mathcal V}_{X_{i'_r   }\curvearrowright {Y}_{_ {j_r}  }  }\tau    
 \end{aligned}\label{vol4a}\ee

 \end{lemma}

   \begin{proof}
   Setting $$\Pi_{n,\ell}= \left\{\mbox{$\mbox{\bf $  \vr$} =(\vr_1,\ldots,\vr_n)$, permutations of  
                                                       $ (\underbrace{+,\ldots,+}_{n-\ell},\underbrace{-,\ldots,-}_{\ell})$}\right\}   
                                                      ,$$
one checks
 $$\begin{aligned}
 & \det\left[   C^+_1+  C^-_1,\ldots,   C^+_n+  C^-_n       \right]
 \\
 &=\sum_{\ell=0}^n \sum_{\vr\in \Pi_{n,\ell}}\det(C^{\vr_1}_1,\ldots,  C^{\vr_n}_n)
 \\& =\sum_{\ell=0}^n\sum_{_{  
      \{\hat j_1<\ldots<\hat j_{n-\ell}\} \cup \{j_1<\ldots<j_\ell\} 
   = (1,\ldots,n) 
    }}  
     (-1)^{ \mbox{\tiny sign}(\vr) }\det(C^+_{_{\hat j_1}},\ldots,C^+_{_{\hat j_{n-\ell}}}  ,
      C^-_{_{j_1}} ,\ldots,C^- _{j_{ \ell}} )
\\
&\stackrel{*}{ =}\sum_{\ell=0}^\rho\sum_{_{  
    %
   \{\hat j_1<\ldots<\hat j_{n-\ell}\} \cup \{j_1<\ldots<j_\ell\} 
   = (1,\ldots,n) 
    }\atop {\{ \hat i_1<\ldots<\hat i_{n-\ell}\} \cup \{\dt+1\leq i_1<\ldots<i_\ell\} 
   = (1,\ldots,n)}}  (-1)^{ \sum_1^\ell (i_r+j_r) }\det(C^+_{\hat i_r,{\hat j_s}} )_{1\leq r,s\leq n-\ell } 
    \det( C^-_{i_r,j_s}  )_{1\leq r,s\leq  \ell }  
   \end{aligned}$$                                                    
$$\begin{aligned}\stackrel{**}{ =}\sum_{\ell=0}^\rho
  &\sum_{_{  
   \{\hat j_1<\ldots<\hat j_{n-\ell}\} \cup \{j_1<\ldots<j_\ell\} 
   = (1,\ldots,n) 
    }\atop {\{ \hat i_1=1,\ldots,\hat i_\dt=\dt<\hat i_{\dt+1}<\ldots   <\hat i_{n-\ell}\} \cup \{\dt+1\leq \dt+ i'_1<\ldots<\dt+i'_\ell\} 
   = (1,\ldots,n)}}  (-1)^{ \sum_1^\ell (i'_r+j_r)+\ell \dt }
   \\
   &\hspace*{4cm}\det(C^+_{\hat i_r,\hat j_s} )_{_{1\leq r,s\leq n-\ell }} 
    \det( C^-_{\dt+i'_r,j_s} )_{_{1\leq r,s\leq  \ell }},
   \end{aligned}$$
with the exponent $(i',j):=\sum_1^\ell (i_r'+j_r)+\ell \dt$. In $\stackrel{*}{ =}$ the $i_j$'s can be taken $\geq \dt+1$, because if  $i_1\leq \dt$, then the determinant of the $C^-$'s would be $=0$, due to the fact that the first $\dt$ entries of $C_j$ vanish and so the matrix would have a zero row. That also implies that $\ell$ can be taken $0\leq \ell \leq \rho$. In view of the inequalities in the set of $i_j$'s in formula $\stackrel{*}{ =}$,  it is convenient to set $i'_r:=i_r-\dt$ in formula $\stackrel{**}{ =}$. Also notice that because of the second set of  inequalities under the summation sign, we must have that $\hat i_r=r$ for $1\leq r\leq \dt$, since $0\leq\dt \leq n-\ell$ and $0\leq \ell\leq \rho$.
Finally one needs a so-called ``higher Fay identity" $\stackrel{*}{=}$, (to be shown in the remark below)
\be
 \det( C^-_{\dt+i'_r,j_s} )_{_{1\leq r,s\leq  \ell }}=\det \left[\frac{{\mathcal V}_{X_{i'_r   }\curvearrowright {Y}_{_ {j_s}  }  }\tau}{\tau}\right]_{_{1\leq r,s\leq  \ell }}\stackrel{*}{=}\frac{\prod_1^\ell {\mathcal V}_{X_{i'_r   }\curvearrowright {Y}_{_ {j_r}  }  }\tau}{\tau}.
\label{fay}\ee
This enables us to establish formula (\ref{vol4a}) and thus proving Lemma \ref{vollemma6.3}.
 \end{proof}

\noindent {\em Remark}: The ``higher Fay identity" is an identity familiar in $\tau$-function theory; see e.g. \cite{ASV}. In this very specific context of determinant of matrices it can be established by induction on $\ell$: assuming the identity true up to  $\ell-1$, we prove it for $\ell$, by expanding the first row of the determinant in (\ref{fay}), yielding: 
$$
\begin{aligned}
  \sum_{s=1}^\ell & (-1)^{s-1}{\mathcal V}_{X_{i'_1   }\curvearrowright {Y}_{_ s  } 
}\tau
 \prod_{r=2}^\ell  {\mathcal V}_{X_{i'_r   }\curvearrowright {X}_{_ {\tilde i_r} }}
 \tau =
  \tau  \prod_1^\ell {\mathcal V}_{X_{i'_r   }\curvearrowright {Y}_{_ r  }    
} \tau,
\\
&\mbox{  for  } (\tilde i_2,\ldots,\tilde i_{\ell})=(1, \ldots,\hat s,\ldots,\ell),
\end{aligned}
$$ 
where the latter identity is just a Pl\"ucker relation\footnote{For $X=(x_{ij})_{{1\leq i\leq n }\atop{1\leq j\leq k}}  \in {\rm Gr}(k,n)$ with $k\leq n$, define $p_{i_1i_2 \ldots i_k}:=\det  (X_{i_\ell j})_{1\leq \ell, j\leq k}
$. }
$$\sum_{a=1}^{k+1} p_{i_1i_2 \ldots i_{k-1}j_a}
p_{j_1j_2 \ldots \hat j_a\ldots  i_k}=0,
$$
establishing the ``Fay identity".

   \begin{lemma}\label{vollemma2}
   The following holds:
   \be\begin{aligned}
 &2^{ n(n-1)}\Gamma_{\bf xy}  
\\&\stackrel{*}{ =}
\sum_{\ell=0}^\rho
   \sum_{_{  
   \{\hat j_1<\ldots<\hat j_{n-\ell}\} \cup \{j_1<\ldots<j_\ell\} 
   = (1,\ldots,n) 
    }\atop {\{ \hat i_1=1,\ldots,\hat i_\dt=\dt<\hat i_{\dt+1}<\ldots   <\hat i_{n-\ell}\} \cup \{\dt+1\leq \dt+ i'_1<\ldots<\dt+i'_\ell\} 
   = (1,\ldots,n)}}  (-1)^{ \sum_1^\ell (i'_r+j_r)+\ell \dt }
 %
   \\&\hspace*{4cm}\times\det( P_{\hat i_r-1}( y_{\hat j_s}))_{1\leq r,s\leq n-\ell } 
      \prod _{r=1}^\ell 
 {\mathcal V}_{X_{i'_r   }\curvearrowright {\mathbb H}_{y_{j_r}  }  }\tau_n \Bigr|_{{x\to x-\beta}\atop{y\to y-\beta}}   \\
 &\stackrel{**}{ =}
 \sum_{\ell=0}^\rho
   \sum_{_{  
   \{\hat j_1<\ldots<\hat j_{n-1-\ell}\} \cup \{j_1<\ldots<j_\ell\} 
   = (1,\ldots,n-1) 
    }\atop {\{ \hat i_1=1,\ldots,\hat i_{\dt'}=\dt'<\hat i_{\dt'+1}<\ldots   <\hat i_{n-1-\ell}\} \cup \{\dt'+1\leq \dt'+ i'_1<\ldots<\dt'+i'_\ell\} 
   = (1,\ldots,n-1)}}  \!\!\! \!\!\! \!\!\! \!\!\! \!\!\! \!\!\! \!\!\! (-1)^{ \sum_1^\ell (i'_r+j_r)+\ell \dt' }
 %
   \\&\hspace*{1cm}\times\det( P_{\hat i_r-1}( y_{\hat j_s})-
   P_{\hat i_r-1}( y_{\hat j_s +1}))_{1\leq r,s\leq n-\ell } 
      \prod _{r=1}^\ell 
 {\mathcal V}_{X_{i'_r   }\curvearrowright {\mathbb H}_{y_{j_r}  }-{\mathbb H}_{y_{j_r +1}  }  }\tau_n\Bigr|_{{x\to x-\beta}\atop{y\to y-\beta}} \label{vol4b}\end{aligned}\ee
where $\dt'=\dt-1$. Written out, we have
\be\begin{aligned}
 \prod _{r=1}^\ell &
 {\mathcal V}_{X_{i'_r   } \curvearrowright {\mathbb H}_{y_{j_r}  }  } \tau_n
 \\&=
   \det\left(  X_1,\ldots, X_{i'_1-1},{\mathbb H}_{y_{j_1}},X_{i'_1+1},\ldots, X_{i'_\ell-1},{\mathbb H}_{y_{j_\ell}},X_{i'_\ell+1},\ldots,X_{n}
 \right)\Bigr|_{{x\to x-\beta}\atop{y\to y-\beta}}.\end{aligned}\label{vol4c}\ee

      \end{lemma}

\begin{proof}
 Remember expression (\ref{Gamma1}) for $\Gamma_{\bf xy}$,
\be  \Gamma_{\bf xy}  
 =\frac{\det\widetilde {\cal H}^{( \beta -x)}_{n}}{2^{n(n-1)}} \det\left[   P_{ i-1 }(y_j-\beta )  ~+~  f_{i-\dt-1}({\bf x}-\beta ,y_j -\beta) \right]_{1\leq i,j\leq n}.
\label{Gamma'}\ee
Notice that for $\dt\geq 1$, the first row of the second matrix in the expression consists of all $1$'s, because $P_0=1$ and $f_i=0$ for $i<0$. It is therefore natural to subtract each column from the previous one and so obtain the determinant of a matrix of size $n-1$, and note in effect $\dt\to \dt'=\dt-1$ (as in $C_j^-$), with $\tau_n$ remaining unchanged; this gives:
   \be\begin{aligned}
& \Gamma_{\bf xy}   
 =(-1)^{n-1}\frac{\det\widetilde {\cal H}^{(   -x)}_{n} }{2^{n(n-1)}}
 \\
 &\times\!\det\left[   P_{ i-1 }(y_{j }  )-P_{ i-1 }(y_{j+1}  )  ~+~  f_{i\!-\!\dt\!-\!1}({\bf x}  ,y_j )-f_{i\!-\!\dt\!-\!1}({\bf x} ,y_{j+1} ) \right]_{1\leq i,j\leq n-1}
\Bigr|_{{x\to x-\beta}\atop{y\to y-\beta}}
  \label{Gamma''}\end{aligned}\ee 
  The expressions (\ref{Gamma'}) and (\ref{Gamma''}) for $\Gamma_{xy}$ have exactly the form (\ref{vol4a}) in Lemma \ref{vollemma6.3}. which therefore can be applied straightforwardly in both cases, with $\tau_n=\det \widetilde {\cal H}^{(   -x)}_{n} =\det \widetilde {\cal H}^{( \beta  -x)}_{n} $.
Remember,  for $1\leq k \leq \rho $, one has
 $$
 \begin{aligned}
  & f_{k-1}({\bf x },{y  })=  - 2^{n+\dt-1}
  \det\left(\widetilde {\cal H}^{( -x)}_{n}\right)^{-1}     \hspace*{1.6cm} \begin{array}{c}k
 \\
 \downarrow \end{array} \\
  & \times\det\left( \widetilde H_{ n -1}(  -x_i),~
   \ldots,~\widetilde H_{ n  -k+1}(  -x_i),  
   {\mathbb H}^{  n+\dt }( y-x_i) ,
   \widetilde H_{ n  -k-1}(  -x_i),~\ldots, \widetilde H_{ 0}(  -x_i)\right)_{1\leq i\leq n},
 \\
 &= \frac {1 }{\tau_n ({\bf x})}
   \det\left( X_1,~
   \ldots,~X_{k-1},  
   {\mathbb H}_y,X_{k+1},~\ldots,   X_n\right)_{1\leq i\leq n},
 \\
&= \frac { 1 }{\tau_n ({\bf x})}{\mathcal V}_{_{X_k  \curvearrowright {\mathbb H}_{y }}}
\det\left(X_1,\ldots,X_n\right)
\\&=      \frac{{\mathcal V}_{_{X_k  \curvearrowright {\mathbb H}_{y }}}
\tau_n }{\tau_n  }\mbox{  }\end{aligned}$$
It then follows that 
$$\begin{aligned}
    f_{k-1}({\bf x },{y  })-f_{k-1}({\bf x },{y'  })&=     \frac{{\mathcal V}_{_{X_k  \curvearrowright {\mathbb H}_{y }-{\mathbb H}_{y' }}}
\tau_n }{\tau_n  }.\mbox{  }\end{aligned}$$
  According to formula (\ref{vol4a}), one needs to compute the following determinant. Given $\{  i'_1<\ldots<i'_\ell\} 
   \in \{ 1,\ldots,\rho \}$ and $    \{j_1<\ldots<j_\ell\} 
   \in \{1,\ldots,n\} 
$, one checks from the higher Fay identity (\ref{fay}),
 \be\begin{aligned}
 \det\left(f_{i'_r -1}({\bf x},y_{j_s})\right)_{1\leq r,s\leq \ell}
&=    \det  
 \left( 
-\frac   
 {
 {\mathcal V}_{X_{i'_r   }\curvearrowright {\mathbb H}_{y_{j_s}  }  }\tau_n
 }
 {\tau_n}
\right)_{1\leq r,s\leq  \ell } 
= \frac{1}{\tau_n}\prod _{r=1}^\ell 
 {\mathcal V}_{X_{i'_r   }\curvearrowright {\mathbb H}_{y_{j_r}  }  }\tau_n
 \end{aligned}\label{vol6}\ee
 and
 \be\begin{aligned}
 \det&\left(f_{i'_r -1}({\bf x},y_{j_s})-f_{i'_r -1}({\bf x},y_{j_s+1})\right)_{1\leq r,s\leq \ell}
\\&= \frac{ 1}{\tau_n}\prod _{r=1}^\ell 
 {\mathcal V}_{X_{i'_r   }\curvearrowright {\mathbb H}_{y_{j_r}} -{\mathbb H}_{y_{j_r +1}}    }    
 \tau_n.
 \end{aligned}\label{vol6'}\ee
 This establishes Lemma \ref{vollemma2}.
 \end{proof}

%


\noindent {\em Proof of Proposition \ref{volprop1} for $n>\rho$}: Already knowing that $\begin{aligned}
\mbox{Vol}( {\mathcal C}^{(n)}_{\bf xy})&=   \Gamma^{(n)}_{\bf xy}  
\end{aligned}$ for $n=\rho$, we proceed by induction. Assuming $\mbox{Vol}( {\mathcal C}^{(n)}_{\bf xy}) =   \Gamma^{(n)}_{\bf xy}  
$  for arbitrary $n\geq \rho$, it suffices to show that $ 
\mbox{Vol}( {\mathcal C}^{(n+1)}_{\bf xy}) =   \Gamma^{(n+1)}_{\bf xy}  
 $, or what is equivalent, for any set $w_1\leq \ldots\leq w_{n+1}$ and $z_1\leq \ldots\leq z_{n+1}$ in $\BR$, it suffices to show
 \be\begin{aligned}
\prod_{i=1}^{n}\int_{z_{i }}^{z_{i+1}}dx_i \int_{w_{i }}^{w_{i+1}}dy_i ~~ \Gamma^{(n)}_{\bf xy} = \Gamma^{(n+1)}_{\bf zw}
,\end{aligned}
\label{integral}\ee 
since this identity obviously holds true upon replacing $\Gamma^{(n)}_{\bf xy} $ and $  \Gamma^{(n+1)}_{\bf zw}$ by the volumes 
$\mbox{Vol}( {\mathcal C}^{(n)}_{\bf xy})$ and $\mbox{Vol}( {\mathcal C}^{(n+1)}_{\bf zw})$.

Before proceeding we need some preliminary integrals. Since ${\widetilde H}'_\ell(y)=2 {\widetilde H}_{\ell-1}(y)$ and $P_n(x) =  \frac 1{\I^n}\widetilde H_n(\I x)$, we have that 
$$\frac{d}{dy} {\widetilde H}_\ell(-y)=
-2 {\widetilde H}_{\ell-1}(-y)   \mbox{   and   } \frac{d}{dy}P_\ell (y)=2P_{\ell-1}(y)
$$
and so, for $z_1<z_2$,
$$
\int_{z_1}^{z_2} dy ~P_{\ell-1}(y)=\frac 12 (P_\ell(z_2)-P_\ell(z_1))\mbox{   and   }
\int_{z_1}^{z_2} dy ~{\widetilde H}_{\ell-1}(-y)= \tfrac   { {\widetilde H}_{\ell }(-z_1)
 -{\widetilde H}_{\ell }(-z_2) }{  2}.
$$
Thus,
given $z_1<z_2$, $w_1<w_2$ and $x\in \BR$, we have the identity
$$\begin{aligned}
 & \int_{w_1}^{w_2} dy~  {\mathbb H}^{(\ell+1)}(y-x)
=  {\mathbb H}^{(\ell+2)}(w_2-x)- {\mathbb H}^{(\ell+2)}(w_1-x) 
,\end{aligned}
$$
and integrating once more, one has
\be\begin{aligned}
 & \int_{z_1}^{z_2} dx~\int_{w_1}^{w_2} dy  ~  {\mathbb H}^{(\ell+1)}(y-x)
\\
&=
 -{\mathbb H}^{(\ell+3)}(w_2-z_2)+ {\mathbb H}^{(\ell+3)}(w_2-z_1)+{\mathbb H}^{(\ell+3)}(w_1-z_2)- {\mathbb H}^{(\ell+3)}(w_1-z_1)
\end{aligned}
\label{intH}\ee
 So, we shall prove the identity (\ref{integral}) with $\Gamma_{\bf xy}$ as obtained in Lemma \ref{vollemma2}; namely, the first formula of (\ref{vol4b}). 
 %

Notice that the columns of $\det( P_{\hat i_r}( y_{\hat j_s}))_{1\leq r,s\leq n-\ell }$ and $ \prod _{r=1}^\ell 
 {\mathcal V}_{X_{i'_r   }\curvearrowright {\mathbb H}_{y_{j_r}  }  }\tau_n$ in formula (\ref{vol4b}), taken together, are labeled by distinct $y_{\hat j_s}$'s and $y_{j_r} $, whereas the rows  of $ \prod _{r=1}^\ell 
 {\mathcal V}_{X_{i'_r   }\curvearrowright {\mathbb H}_{y_{j_r}  }  }\tau_n$  are labeled by the $x_i$'s; of course, the first determinant does not contain $x_i$'s. Performing the integration (\ref{integral}), 
%
%
 %
this enables us to distribute the $y$-integrations over the different columns and the $x$-integrations over the different rows. One then uses the integration formulas (\ref{intH}) yielding the expression below, where in the first determinant $i'_k$ ranges over $i'_1<\ldots<i'_\ell$, and  remembering the expression (\ref{vol4c}) for  
$ \prod _{r=1}^\ell 
 {\mathcal V}_{X_{i'_r   }\curvearrowright {\mathbb H}_{y_{j_r}  }  }\tau_n$ in terms of the $X_i$ as defined in (\ref{vol3}), (remember it suffices to show (\ref{integral}))
$$
 \begin{aligned}
   \prod_{i=1}^{n}&\int_{z_{i }}^{z_{i+1}}dx_i \int_{w_{i }}^{w_{i+1}}dy_i ~~ \Gamma^{(n)}_{\bf xy}  \Bigr|_{\beta=0}
 \\
=&  \frac{1}{2^{n(n-1)}} \sum_{\ell=0}^\rho\sum_{_{  
    %
   \{\hat j_1<\ldots<\hat j_{n-\ell}\} \cup \{j_1<\ldots<j_\ell\} 
   = (1,\ldots,n) 
    }\atop {\{ \hat i_1<\ldots<\hat i_{n-\ell}\} \cup \{\dt+1\leq \dt+ i'_1<\ldots<\dt+i'_\ell\} 
   = (1,\ldots,n)}}  \!\!\! (-1)^{\sum_1^\ell (i_k'+j_k)+\ell \dt}   
   \\
   &~~\det\left[  \tfrac 1{ 2} ( \widetilde H_{n}(-z_{i })- \widetilde H_{n}(-z_{i+1 })) 
,  \tfrac 1{ 2} ( \widetilde H_{n-1}(-z_{i })-
 \widetilde H_{n-1}(-z_{i+1 })),
\ldots,\right.\\
&  ~~~~~~~~ \tfrac 1{ 2} ( \widetilde H_{n+1-(i'_{k}-1)}(-z_{i })- \widetilde H_{n+1-(i'_{k}-1)}(-z_{i+1 })), \\
&~~~~~~~~~~ \tfrac 14 \left(
  ({\mathbb H}^{\rho+2\dt+2}_{w_{j_k },z_{i+1}}
- {\mathbb H}^{\rho+2\dt+2}_{w_{j_k +1 },z_{i+1}}  )  
- ({\mathbb H}^{\rho+2\dt+2}_{w_{j_k },z_{i }}
- {\mathbb H}^{\rho+2\dt+2}_{w_{j_k +1 },z_{i }})
\right),\\
& \left.~~~~~  \tfrac 12 ( \widetilde H_{n+1-(i'_{k}+1)}(-z_{i })- \widetilde H_{n+1-(i'_{k}+1)}(-z_{i +1})) ,\right.\\
&~~~~~~~\left.\dots, \tfrac 12( \widetilde H_{2}(-z_{i })- \widetilde H_{2}(-z_{i +1})),\widetilde H_{1}(-z_{i+1})- \widetilde H_{1}(-z_{i })\right]_{1\leq i\leq n}
\\
&\times   \det\left[
 \tfrac 12 \left(P_{\hat i_r }(w_{\hat j_s +1})   -
P_{\hat i_r }(w_{\hat j_s })\right)
\right]_{1\leq r,s\leq n-\ell }~\Bigr|_{\beta=0}
\end{aligned}
$$
$$
\begin{aligned}
& \stackrel{*}{=} {2^{-n(n+1)}}\sum_{\ell=0}^\rho\sum_{_{  
    %
   \{\hat j_1<\ldots<\hat j_{n-\ell}\} \cup \{j_1<\ldots<j_\ell\} 
   = (1,\ldots,n) 
    }\atop {\{ \hat i_1<\ldots<\hat i_{n-\ell}\} \cup \{\dt+1\leq \dt+ i'_1<\ldots<\dt+i'_\ell\} 
   = (1,\ldots,n)}}  \!\!\! (-1)^{\sum_1^\ell (i_k'+j_k)+\ell \dt}
   \\
   &\det\left[  \widetilde H_{n}(-z_i) , ~\widetilde H_{n-1}(-z_i) ,\ldots   
 ,  \widetilde H_{n+1-(i'_{k}-1)}(-z_i)  ,  \right.\\
&  \left.   
 - {\mathbb H}^{\rho+2\dt+2}_{w_{j_k +1},z_i}
+{\mathbb H}^{\rho+2\dt+2}_{w_{j_k  },z_i} ,   
    \widetilde H_{n+1-(i'_{k}+1)}(-z_i) ,\dots,\widetilde H_1(-z_i),1\right]_{1\leq i\leq n+1}
\\
&\times   \det\left[
   P_{\hat i_r }(w_{\hat j_s})   -
P_{\hat i_r }(w_{\hat j_s+1}) 
\right]_{1\leq r,s\leq n-\ell }
\Bigr|_{\beta=0}
\\
&= \Gamma^{(n+1)}_{\bf z w}\Bigr|_{\beta=0}
\end{aligned}
$$
using in equality $\stackrel{*}{=}$ the following row operations on determinants of matrices:
$$
\det(a_i-a_{i+1}, b_i-b_{i+1},\ldots,c_i-c_{i+1})_{1\leq i\leq n}
=
\det(a_i,b_i,\ldots,c_i,1)_{1\leq i\leq n+1};
$$
this enables one to write the $n\times n$ determinant before equality $\stackrel{*}{=}$ as a $(n+1)\times (n+1)$ determinant. To see that the last equality holds, one uses the second formula of (\ref{vol4b}), with $\dt$ replaced by $\dt+1$ and thus $n$ by $n+1$. This shows the identity (\ref{integral}) and thus completes the induction. Since the expression for $\Gamma^{(n+1)}_{\bf z w}$ happens to be a $\beta$-independent volume, it is $\beta$-independent as well. This ends the proof of Proposition \ref{volprop1}.\qed

{\em Proof of Corollary \ref{volcor1}}: it appears in the remark at the end of previous section; see formula (\ref{86''}).\qed

\section{Equivalence of the two systems}

Probability (\ref{uniform}) in Theorem \ref{a=1} states that: 
 \be\begin{aligned}
 \BP^{\mbox{\tiny\rm TAC}}& \left(
  \bigcap_{1}^{n-1}
  \{{{\bf x}^{(k)}\!\! \in d{\bf x}^{(k)} ,~{\bf y}^{(k)}\!\!\in d{\bf y}^{(k)}  }  \}\Bigr|
 \begin{array}{c} \mbox{$ {\bf x}^{(n)}={\bf x}$ and $ {\bf y}^{(n)}={\bf y}$}
   %
\end{array}\!\!\! \right)
     = \frac{d\mu_ {{\bf x} {\bf y}}}{ \mbox{Vol}~({\mathcal C} _{{\bf x} {\bf y}})} 
  , \end{aligned}
\label{uniform8} \ee
and Theorem \ref{Th:JProb} states that
 \be\begin{aligned}
 \BP^{\mbox{\tiny\rm TAC}}& \left(
  \{{{\bf x}^{(n)}\in d{\bf x} ,~{\bf y}^{(n)}\in d{\bf y}  }  \}
 \right)
 =\frac{\rho_n^{ \mbox{\tiny GUE}}({\bf x}\!-\!\beta) 
 ~ \rho_n^{\mbox{\tiny GUE}}( {\bf y}\!+\!\beta)}
 {\det(\Id- {\mathcal K}^{\beta}
 )_{[-1, -\rho]}}
 \frac {  \mbox{\em Vol} ({\mathcal C}^{(n)}_{{\bf x}{\bf y} }) d{\bf x}d{\bf y} }{\mbox{\em Vol} ({\mathcal C}_{{\bf x} }) \mbox{\em Vol} ({\mathcal C}_{ {\bf y} })    }
, \end{aligned}
 \label{joint8}\ee
implying
 \be\begin{aligned}
 \BP^{\mbox{\tiny\rm TAC}} & \Bigl(
  \bigcap_{1}^{n }  
  \{{{\bf x}^{(k)}\!\! \in d{\bf x}^{(k)} ,~{\bf y}^{(k)}\!\!\in d{\bf y}^{(k)}  }  \}\! \Bigr)
 \\& 
 =\frac{\rho_n^{ \mbox{\tiny GUE}}({\bf x}\!-\!\beta) 
 ~ \rho_n^{\mbox{\tiny GUE}}( {\bf y}\!+\!\beta)}
 {\det(\Id- {\mathcal K}^{\beta}
 )_{[-1, -\rho]}}
 \frac {  \mbox{\em   }   d{\bf x}d{\bf y} d\mu_ {{\bf x} {\bf y}}}{\mbox{\em Vol} ({\mathcal C}_{{\bf x} }) \mbox{\em Vol} ({\mathcal C}_{ {\bf y} })    } . \end{aligned}
\label{uniform8} \ee
From (\ref{15''}) the probability of the same event for the coupled random matrix model equals
\be
\begin{aligned}
\BP^{\mbox{ }}&\Bigl(\bigcap_{k=1}^n\{ {\bf x}^{(k)}\in d{\bf x}^{(k)},~
 {\bf y}^{(k)}\in d{\bf y}^{(k)}\} 
    \Bigr)
\\
&
  =\frac{\rho_n^{ \mbox{\tiny GUE}}({\bf x}\!-\!\beta) 
   \rho_n^{\mbox{\tiny GUE}}( {\bf y}\!+\!\beta)}
{\BP^{\mbox{\tiny GUE}} \left(\bigcap _1^{\rho}\{y_i^{(i)}\leq  x_{1}^{(\rho-i+1)} 
 \}\right)}
~  \frac { d{\bf x}d{\bf y}  d\mu_ {{\bf x} ,{\bf y}} }{\mbox{ Vol} ({\mathcal C}_{{\bf x} }) \mbox{ Vol} ({\mathcal C}_{ {\bf y} }) }
   .\end{aligned}
  \label{15''8}\ee 
Since both expressions only differ by the $({\bf x},{\bf y})$-independent expressions in the denominator, we have that upon integration in ${\bf x}$ and ${\bf y}$, and noticing that probabilities integrate to $1$, one finds
$$
  \BP^{\mbox{\tiny GUE}} \left(\bigcap _1^{\rho}\{y_i^{(i)}\leq  x_{1}^{(\rho-i+1)} 
 \}\right)=
 \det(\Id- {\mathcal K}^{\beta}
 )_{[-1, -\rho]}
~,$$
 including the explicit formula for this Fredholm determinant, appearing in (\ref{ineq}). This establishes Theorem \ref{Th1.2} for both models, {\bf I} and {\bf II}, including equality (\ref{ineq}). This shows that the point processes for both models are equivalent. This immediately shows that 
 Theorem \ref{Th1.4} holds for the edge-tacnode process and for the coupled GUE-matrix problem as well. 
\qed


\begin{thebibliography}{10}






 
 
  
 \bibitem{ACJvM}
M.~Adler, Sunil Chhita, K. Johansson and P.~van Moerbeke, \emph{Tacnode GUE-minor Processes and Double Aztec Diamonds},   2013 (arXiv: 1303.5279).


 \bibitem{AFvM12}
M.~Adler, P.L. Ferrari, and P.~van Moerbeke, \emph{Non-intersecting random walks in the neighborhood of a symmetric Tacnode}, Ann. of Prob. {\bf 41}, 2599-2647 (2013) (arXiv:1007.1163).

 \bibitem{AJvM}
 M.~Adler, K.~Johansson, and P.~van Moerbeke, \emph{Double Aztec Diamonds and the Tacnode Process}, Adv in Math, {\bf 252}, 1-54  (2014). (arXiv:1112.5532).
 
  \bibitem{ASV}
M.~Adler, T. Shiota and P.~van Moerbeke, 
 {\sl Random matrices, Virasoro algebras and non-commutative KP} , Duke
Math. J. {\bf 94} , 379--431 (1998) (arXiv: solv-int/9812006)


 






  





  
  

  
  






\bibitem{Bary}
 Yu.~Baryshnikov, \emph{GUEs and queues}, Prob. Th. and Related Fields \textbf{119} (2001), 256--274.
  

\bibitem{CKP}
Henry Cohn, Richard Kenyon, and James Propp.
{\em A variational pricinple for domino tilings}, J. Amer. Math. Soc., \textbf{14}(2):297-346, 2001. 





\bibitem{DKZ10}
S.~Delvaux, A.~Kuijlaars, and L.~Zhang, \emph{{Critical behavior of
  non-intersecting Brownian motions at a tacnode}}, Comm. Pure Appl. Math.
  \textbf{64} (2011), 1305–--1383.

  
\bibitem{EKLP} N. Elkies, G. Kuperberg, M. Larsen and J. Propp, \emph{Alternating sign Matrices and Domino Tilings, part I,}  J. Algebraic Combin. 1 (1992), no. 2, 111--132.

\bibitem{EKLP2} N. Elkies, G. Kuperberg, M. Larsen and J. Propp, \emph{Alternating sign Matrices and Domino Tilings, part II,}  J. Algebraic Combin. 1 (1992), no. 2, 219--234.



  \bibitem{FS03}
  P.L. Ferrari and H.~Spohn, \emph{Step fluctations for a faceted crystal}, J.
    Stat. Phys. \textbf{113} (2003), 1--46.
  
  
  \bibitem{FV} P. L. Ferrari and B. Vet\H o, \emph{Non-colliding Brownian bridges and the asymmetric tacnode process},  Electron. J. Probab. \textbf{17} (2012), 1-17

  
  \bibitem{JPS}
W.~Jockush J. Propp and P. Shor, \emph{Random domino tilings and the arctic circle Theorem}. Preprint. Available at arXiv.org/abs/math.CO/9801068.  

\bibitem{Jo02b}
K.~Johansson, \emph{Non-intersecting paths, random tilings and random
  matrices}, Probab. Theory Related Fields \textbf{123} (2002), 225--280.

\bibitem{Jo03b}
K.~Johansson, \emph{Discrete polynuclear growth and determinantal processes},
  Comm. Math. Phys. \textbf{242} (2003), 277--329.


  
  












\bibitem{Johansson3}
 K. Johansson:
{\em   The Arctic circle boundary and the Airy process}, Ann. Probab. 33 (2005), no. 1, 1Ð30. 
(ArXiv. Math. PR/0306216 (2003))




 \bibitem{Joh10}
 K.~Johansson, \emph{{Non-colliding Brownian Motions and the extended tacnode
   process}}, arXiv:1105.4027 (2011).
  
  \bibitem{JoNo} K.~Johansson and E.~Nordenstam: {\em Eigenvalues of GUE minors}, Electron. J. Probab. 11 (2006), no. 50, 1342--1371.
 
  
  
  
  























 
 \bibitem{OR} A. Okounkov and N. Reshetikhin. {\em The birth of a random matrix}. Mosc.
Math. J., 6(3):553-566 (2006).


\end{thebibliography}
  \end{document}